\documentclass[11pt, letterpaper]{article}

\usepackage[T1]{fontenc}
\usepackage[utf8]{inputenc}
\usepackage{amsthm}         
\usepackage{microtype}      
\usepackage{newtxtext}      
\usepackage{newtxmath}      

\usepackage[margin=1in]{geometry}
\usepackage{setspace}       
\usepackage{fancyhdr}       
\usepackage[labelfont={bf,sf}, textfont={small}, labelsep=period]{caption} 
\usepackage{booktabs}       
\usepackage{xcolor}         

\definecolor{AcademicBlue}{RGB}{0, 50, 150}

\usepackage{amsmath, mathtools} 
\usepackage{thmtools}       

\declaretheoremstyle[
    headfont=\bfseries\sffamily\color{AcademicBlue!80!black}, 
    bodyfont=\itshape,
    spaceabove=1em,
    spacebelow=1em,
]{coloredthm}

\declaretheorem[style=coloredthm, name=Theorem]{theorem}
\declaretheorem[style=coloredthm, name=Proposition]{proposition}
\declaretheorem[style=coloredthm, name=Lemma]{lemma}
\declaretheorem[style=coloredthm, name=Corollary]{corollary}

\declaretheoremstyle[
    headfont=\bfseries\sffamily,
    bodyfont=\normalfont,
    spaceabove=1em,
    spacebelow=1em,
]{definitionstyle}

\declaretheorem[style=definitionstyle, name=Definition]{definition}
\declaretheorem[style=definitionstyle, name=Assumption]{assumption}

\declaretheorem[style=definitionstyle, name=Example]{example}

\newcommand{\E}{\mathbb{E}}
\DeclareMathOperator*{\argmax}{arg\,max}
\DeclareMathOperator*{\argmin}{arg\,min}

\usepackage[linesnumbered,ruled,vlined]{algorithm2e} 
\SetKwInput{KwInput}{Input}
\SetKwInput{KwOutput}{Output}

\usepackage[style=apa, backend=biber, natbib=true, maxcitenames=2, doi=true, url=false, uniquename=false]{biblatex}
\usepackage{graphicx}
\usepackage{subcaption}
\usepackage{float}
\usepackage{enumitem}       
\setlist{nosep}             

\usepackage{authblk}        

\usepackage{xr-hyper}
\usepackage[colorlinks=true, linkcolor=AcademicBlue, citecolor=AcademicBlue, urlcolor=AcademicBlue, breaklinks=true]{hyperref}

\usepackage[nameinlink, capitalize]{cleveref} 

\title{\textbf{\Large Fenchel--Young Estimators of Perturbed Utility Models}\\
\vspace{0.5em}\large } 

\author[a]{Xi Lin}
\author[a,b,*]{Yafeng Yin}
\author[a]{Tianming Liu}

\affil[a]{Department of Civil and Environmental Engineering, University of Michigan, Ann Arbor, MI 48109, USA}
\affil[b]{Department of Industrial and Operations Engineering, University of Michigan, Ann Arbor, MI 48109, USA} 

\affil[*]{Corresponding author. \protect\\ \textit{Email addresses:} \textrm{xilina, yafeng, tianmliu@umich.edu }.}

\date{}

\begin{document}

\maketitle

\begin{abstract}
\noindent 
Perturbed Utility Models (PUMs) represent discrete choice as utility maximization under convex perturbation or regularization, encompassing a broad class of choice models, including Multinomial Logit (MNL) and models that permit zero choice probabilities. Because, for general perturbations, maximum likelihood may be non-convex, and the negative log-likelihood is infinite when a sparse model assigns zero probability to an observed choice, we specialize the Fenchel--Young criterion to PUM estimation. With a fixed perturbation and utilities linear in parameters, the empirical objective is globally convex and remains finite at sparse predictions. We establish probability-level Fisher consistency, parameter identification, consistency, and asymptotic normality under explicit regularity conditions. We also develop computational methods for estimation, using structured probability evaluation for additive and tree-structured perturbations and a primal–dual algorithm for general non-separable perturbations.

We then introduce Parametric Basis Estimation (PBE), a bi-level method that jointly learns utility parameters and a tree-structured perturbation through convex combinations of leaf-level and nest-level basis functions. Under explicit regularity conditions, we establish joint identification and consistency of the resulting estimator when the bi-level problem is solved globally. We also develop an algorithm with stationarity guarantees under stated smoothness conditions. Synthetic experiments examine structural recovery, normalization, and sensitivity to misspecification. On Swissmetro, repeated respondent-grouped nested cross-validation shows that both additive and tree-structured PBE achieve statistically significant reductions in held-out Brier score relative to MNL.

\vspace{1em}

\noindent \textbf{Keywords:} Discrete choice models; Perturbed utility models; Fenchel--Young estimation; Perturbation learning
\end{abstract}

\section{Introduction}

Random-utility models provide a central framework for analyzing discrete decisions, linking observed choices to individual utility maximization.  \citet{mcfadden1974conditional}'s formulation of conditional logit established a foundation for empirical choice analysis, while developments in transportation connected this framework to disaggregate travel-demand estimation and the modeling of interrelated travel decisions \citep{mcfadden1974measurement,benakiva1974structure,ben1985discrete}. The Multinomial Logit (MNL) model remains widely used because it combines closed-form choice probabilities with tractable parameter estimation. With utilities linear in parameters and the usual scale normalization, its negative log-likelihood is a sum of log-sum-exp terms minus realized utilities and is therefore globally convex in the utility parameters. Closed-form probabilities make probability evaluation straightforward, while convexity allows parameter estimation through convex optimization.

The tractability of MNL comes with restrictions on substitution patterns, notably the Independence of Irrelevant Alternatives property \citep{hausman1984specification}. Classical extensions relax these restrictions in different ways. Multinomial Probit permits correlated normal utility errors but requires substantially more involved probability evaluation and estimation \citep{daganzo1977multinomial,daganzo1979multinomial}. Hierarchical and nested-logit formulations introduce dependence through structured relationships among alternatives \citep{benakiva1974structure,wen2001generalized}. These models broaden the range of behaviors that discrete-choice analysis can represent and motivate a more general question: how can flexible choice models retain the computational advantages that make MNL attractive for estimation?

The Perturbed Utility Model (PUM) framework provides a general optimization-based representation of discrete choice \citep{mcfadden2012theory, fudenberg2015stochastic}. Rather than deriving choice probabilities strictly from exogenously specified stochastic error distributions, a PUM defines choice probabilities as the solution to utility maximization with a convex perturbation, or regularization, function on the probability simplex. The decision maker selects a probability vector that maximizes expected systematic utility net of a convex penalty. Different perturbations generate different choice responses, including models with strictly positive probabilities and models that assign zero probability to some alternatives. The regularization perspective also connects discrete choice to information acquisition and rational inattention under appropriate specifications \citep{fosgerau2020discrete}.

The mathematical connection between PUMs and random-utility models rests on surplus functions and convex duality. In random-utility analysis, the Williams--Daly--Zachary results connect choice probabilities to derivatives of expected maximum utility  \citep{williams1977formation,daly1978improved}, with a modern convex-analytic treatment provided by  \citet{sorensen2022mcfadden}. For a PUM, convex conjugacy links the perturbation to a surplus function whose gradient yields choice probabilities when the optimizing probability vector is unique \citep{rockafellar1970convex}. This surplus coincides, up to an additive constant, with expected maximum utility when an additive random-utility representation exists, but such an interpretation is not required by the general PUM formulation. MNL is recovered through negative Shannon entropy \citep{anderson1992}, while other perturbations yield alternative probability maps. This same optimization perspective has also supported the development of perturbed-utility route-choice and traffic equilibrium models \citep{fosgerau2022perturbed,yao2024perturbed}.

This flexibility raises the question of how to estimate PUMs reliably from observed choices. Maximum likelihood provides a natural approach, but its computational properties depend on the perturbation. Even when choice probabilities are defined by a convex optimization problem, the resulting negative log-likelihood may be non-convex in the utility parameters. Sparse specifications such as Sparsemax \citep{martins2016softmax} introduce an additional difficulty: assigning zero probability to an observed choice makes its negative log-likelihood infinite. These considerations motivate an alternative estimation criterion that retains convexity in utility parameters for a fixed perturbation and a linear utility specification and remains finite at sparse predictions.

Beyond estimating utility parameters under a fixed perturbation, a second problem is to learn the perturbation jointly with the utility parameters. We call this task \textit{structural estimation}. It allows the data to inform the shape and nesting structure of the perturbation within a family of convex functions, rather than imposing a single perturbation in advance. Related work has considered learning nested-logit structures, incorporating flexible representations via machine learning, and estimating PUMs with semi-nonparametric perturbations \citep{aboutaleb2020learning,sifringer2020enhancing,liu2023end,yao2024aperturbed}. These developments motivate an estimation framework that accommodates flexible perturbations while explicitly addressing identification, statistical properties, and computation.


In this paper, we develop such an estimation framework for PUMs based on the Fenchel--Young loss construction of  \citet{blondel2019learning,blondel2020learning}. The criterion uses the convex surplus function associated with the model’s perturbation, linking the estimation objective directly to the optimization problem that generates choice probabilities. For a fixed perturbation and utilities linear in parameters, the resulting objective is globally convex, and its gradient takes the form of a discrepancy between predicted and observed feature moments. The operational criterion also remains finite when the model assigns zero probability to an observed choice. For the Shannon-entropy perturbation, it coincides with the MNL negative log-likelihood up to a fixed positive scale. Building on the established Fenchel--Young construction, we address two estimation problems: estimating utility parameters under a fixed perturbation and jointly learning utility parameters and the perturbation.

For a fixed perturbation, we establish probability-level Fisher consistency and provide conditions for parameter identification, consistency, and asymptotic normality. The analysis distinguishes recovery of choice probabilities from recovery of utility parameters and accommodates sparse probability maps under appropriate regularity conditions. We also develop computational methods that exploit the perturbation’s structure. For additive and tree-structured specifications, structured forward solvers evaluate probabilities within the estimation procedure. For general non-separable perturbations, we derive a saddle-point formulation and, under suitable duality conditions, an equivalent convex formulation with feature-moment constraints. A primal–dual algorithm updates utility parameters and probability vectors jointly, avoiding a complete forward solve at every parameter update.

For joint estimation, we introduce Parametric Basis Estimation (PBE), which jointly estimates utility parameters and a perturbation represented by convex combinations of basis functions. The construction accommodates additive and tree-structured specifications through shared leaf-level weights and node-specific nest-level weights. We introduce normalizations and conditions on covariate variation and the basis functions that ensure joint identification under correct specification. PBE estimates utility parameters through a convex Fenchel--Young inner problem and selects perturbation weights through an outer Brier-score criterion. Under explicit regularity conditions, the resulting estimator is consistent when the bi-level problem is solved globally, even when the same observations are used at both levels. For computation, ridge regularization ensures a unique inner solution, and implicit differentiation provides the gradient of the outer objective within stable active regimes. The proposed projected-gradient algorithm has stationarity guarantees under stated smoothness conditions; these guarantees are distinct from the global optimization requirement in the consistency result.

Numerical experiments examine computational performance, recovery of utility parameters and perturbation functions, and sensitivity to normalization and misspecification. An application to Swissmetro evaluates predictive performance using repeated respondent-grouped nested cross-validation. Both additive and tree-structured PBE achieve statistically significant reductions in held-out Brier score relative to MNL, although the two PBE specifications perform nearly identically in this application.

The remainder of this paper is organized as follows. Section \ref{sec:PUM} introduces PUMs and the convex duality linking perturbations, surplus functions, and choice probabilities. Section \ref{sec:FY} develops the Fenchel--Young estimation criterion and explains its geometry and convexity. Section \ref{sec:FY_solution} presents computational methods, and Section \ref{sec:asymptotic} establishes the statistical properties of estimation under a fixed perturbation. Section \ref{sec:Joint} introduces PBE and analyzes its identification, consistency, and computation. Section \ref{sec:Numerical} reports the numerical experiments and empirical application. Section \ref{sec:concluding} concludes.


\section{Foundations of Perturbed Utility Models} \label{sec:PUM}

\subsection{Basics and Examples}

We introduce the PUM framework developed by \citet{mcfadden2012theory} and \citet{fudenberg2015stochastic}, in which choice probabilities are determined through a convex optimization problem. This formulation exposes the relationship between the perturbation and the choice probability mapping, providing the foundation for the Fenchel--Young estimation framework developed below.

\subsubsection{The Primal Formulation: Utility Maximization with Regularization}

Consider a decision-maker $n$ facing a finite set $\mathcal C$ of $K:=|\mathcal C|$ mutually exclusive alternatives. Let $\mathbf x_{ni}\in\mathbb R^d$ denote the attributes of alternative $i$ and $\boldsymbol\beta\in\mathbb R^d$ the vector of utility parameters. Define $\mathbf X_n\in\mathbb R^{K\times d}$ with row $i$ equal to $\mathbf x_{ni}^{\top}$, and write the systematic utilities as $\mathbf V_n(\boldsymbol\beta)=\mathbf X_n\boldsymbol\beta$, with $V_{ni}=\boldsymbol\beta^\top\mathbf x_{ni}$.

In a PUM, the decision-maker selects a probability vector to maximize expected systematic utility net of a convex perturbation. Let
\[
    \Delta:=\{\mathbf q\in\mathbb R^K_+:
    \mathbf 1^\top\mathbf q=1\}
\]
be the probability simplex. The choice probability vector solves
\begin{equation}
    \label{eq:pum_primal}
    \mathbf p_\Lambda(\mathbf V_n)
    = \argmax_{\mathbf q\in\Delta}
    \{\mathbf q^\top\mathbf V_n-\Lambda(\mathbf q)\},
\end{equation}
where $\Lambda:\Delta\to\mathbb R\cup\{+\infty\}$ is the \textit{perturbation function}. Allowing the value $+\infty$ accommodates perturbations that are finite only on a subset of the simplex. The \textit{effective domain} is
\[
    \operatorname{dom}\Lambda
    :=\{\mathbf q\in\Delta:\Lambda(\mathbf q)<+\infty\}.
\]
We use $\operatorname{ri}$ to denote the \textit{relative interior}, meaning the interior within the smallest affine space containing the set. In particular,
\[
    \operatorname{ri}(\Delta)
    =\{\mathbf q\in\Delta:q_i>0\text{ for every }i\}.
\]

Unless stated otherwise, we maintain the following assumptions:
\begin{itemize}
    \item[(A1)] $\Lambda$ is proper, closed, and strictly convex on its effective domain.
    \item[(A2)] $\Lambda$ is continuously differentiable on
    $\operatorname{ri}(\operatorname{dom}\Lambda)$, with differentiability understood relative to the affine hull of its effective domain.
\end{itemize}
Here, properness means that $\Lambda$ is finite at some point and never takes the value $-\infty$, while closedness means lower semicontinuity on $\Delta$. Properness, lower semicontinuity, and compactness of $\Delta$ ensure that Eq.~\eqref{eq:pum_primal} has a solution for every finite utility vector. Strict convexity makes this solution unique. Assumption (A2) supports derivative-based characterizations of the probability mapping but does not require differentiability at boundary points.

Boundary behavior depends on the perturbation and determines whether optimal probabilities can be zero. A probability mapping has \textit{full support} if $p_{\Lambda,i}(\mathbf V)>0$ for every alternative $i$ and every finite utility vector $\mathbf V$. The examples below illustrate different mechanisms. The Shannon-entropy perturbation is finite on the entire simplex, but its marginal penalty tends to $-\infty$ as a probability approaches zero, yielding full support. The Cauchy perturbation is infinite on the simplex boundary, which also rules out zero optimal probabilities. In contrast, the quadratic perturbation underlying Sparsemax has finite boundary derivatives and permits zero probabilities. Such boundary solutions are characterized by optimality conditions that account for the nonnegativity constraints. We therefore allow both full-support and sparse probability mappings rather than imposing a common boundary condition.

\subsubsection{Convex Conjugacy and the Generalized Williams-Daly-Zachary Theorem}

The Williams--Daly--Zachary literature relates choice probabilities in additive random-utility models to derivatives of expected maximum utility \citep{williams1977formation,daly1978improved}; \citet{sorensen2022mcfadden} provide a modern convex-analytic treatment and demand-inversion result. For a general PUM, the analogous gradient relation follows directly from convex conjugacy \citep{rockafellar1970convex} and does not require a random-utility representation. Define the \textit{surplus function} as the optimal value of the regularized utility-maximization problem:
\begin{equation}
    \label{eq:pum_dual}
    \Omega_\Lambda(\mathbf V):=\sup_{\mathbf q\in\Delta}
    \{\mathbf q^\top\mathbf V-\Lambda(\mathbf q)\}.
\end{equation}

To express this relation using standard convex-conjugate notation, define the extension of $\Lambda$ to $\mathbb R^K$ by
\[
    \bar\Lambda(\mathbf q):=
    \begin{cases}
        \Lambda(\mathbf q), & \mathbf q\in\Delta,\\
        +\infty, & \mathbf q\notin\Delta.
    \end{cases}
\]
Then
\[
    \Omega_\Lambda(\mathbf V)
    =\bar\Lambda^*(\mathbf V)
    :=\sup_{\mathbf q\in\mathbb R^K}
    \{\mathbf q^\top\mathbf V-\bar\Lambda(\mathbf q)\}.
\]
Assigning $+\infty$ outside $\Delta$ simply incorporates the probability constraints into the function. When the probability mapping admits an additive random-utility representation with finite expected maximum utility, $\Omega_\Lambda$ coincides with that expected maximum up to an additive constant.

Convex conjugacy gives the generalized Williams--Daly--Zachary relation below.

\begin{proposition}[PUM Duality]
    Under (A1), $\Omega_\Lambda$ is finite, convex, and differentiable on $\mathbb R^K$. Its gradient is the unique optimizer in Eq.~\eqref{eq:pum_primal}:
    \begin{equation} \label{eq:duality}
        \mathbf p_\Lambda(\mathbf V)=\nabla\Omega_\Lambda(\mathbf V).
    \end{equation}
\end{proposition}

To see why, compactness of the simplex and (A1) ensure a finite optimal value and a unique optimizer. The surplus is convex because it is a supremum of affine functions of $\mathbf V$, and uniqueness of the optimizer yields differentiability and the gradient identity. This result does not require (A2).

This duality is central to our estimation strategy. It implies that estimating $\boldsymbol{\beta}$ amounts to choosing parameters so that the probability vectors generated by Eq.\eqref{eq:duality} align with observed choice realizations under an appropriate loss function.
When $\Lambda$ is fixed, we suppress its subscript and write $\Omega$ and $\mathbf p$ for $\Omega_\Lambda$ and $\mathbf p_\Lambda$.

\subsubsection{Additive Separability and Tree-Structured PUMs}

A tractable subclass of PUMs arises when the perturbation function decomposes into a sum of univariate convex functions \citep{mcfadden2012theory}. This structure yields semi-analytical solutions and encompasses prevalent discrete choice models like MNL and Sparsemax. 

\begin{definition}[Additive Separable PUM] \label{defi:separable}
    A PUM is \textbf{additive separable} if
    \begin{equation}
        \Lambda(\mathbf{q}) = \sum_{i \in \mathcal{C}} h(q_i),
    \end{equation}
    where $h:[0,1]\to\mathbb R\cup\{+\infty\}$ is proper, closed, and strictly convex, and is continuously differentiable on $\operatorname{ri}(\operatorname{dom}h)$. We require that the resulting perturbation be finite at some point of \(\Delta\).
\end{definition}

The common scalar primitive $h$ imposes symmetry across alternatives in the perturbation and is adopted for parsimony and for the shared-weight PBE construction below. Additive separability more generally allows alternative-specific primitives \(h_i\). The same derivation applies with a separate scalar response for each alternative. 

Define the scalar response
\[
    \psi_h(z):=\argmax_{u\in[0,1]}\{uz-h(u)\}.
\]
Properness and closedness ensure existence, while strict convexity makes the response single-valued. The response is nondecreasing in \(z\). It equals \((h')^{-1}(z)\) when the optimizer lies in the interior of the effective domain, and otherwise incorporates the relevant boundary optimality conditions. Under a constraint qualification ensuring the existence of a finite normalization multiplier, the probabilities satisfy
$p_i(\mathbf V)=\psi_h(V_i-\lambda)$ and $\sum_i p_i(\mathbf V)=1$. The probability vector is unique; the multiplier is unique whenever the normalization map is strictly decreasing at its root. This formulation includes both full-support probability mappings and truncated responses such as Sparsemax. The same conjugate-response construction yields link-wise flow recovery in the perturbed-utility route-choice setting of \citet{fosgerau2026convex}, where network flow-conservation constraints replace the simplex normalization used here.

The shared additive specification does not explicitly encode nest-specific interactions among alternatives. To introduce such interactions, the following extension adds convex penalties to aggregate probability masses along a tree.

\begin{definition}[Tree-Structured Perturbed Utility Model] \label{defi:hierarchical}
    A PUM is \textbf{tree-structured} if $\mathcal C$ is the leaf set of a tree $\mathcal T$ and
    \begin{equation}
        \Lambda(\mathbf{q}) = \sum_{i \in \mathcal{C}} h(q_i) + \sum_{s \in \mathcal{T}_{\mathrm{int}}} \Phi_s(y_s).
    \end{equation}
    Here, $\mathcal{T}_{\mathrm{int}}$ denotes the internal nest nodes; $y_s = \sum_{i \in \mathcal{L}(s)} q_i$ is the aggregate probability mass of all leaves $\mathcal{L}(s)$ descending from $s$; $h$ satisfies Definition~\ref{defi:separable}; and each $\Phi_s$ is proper, closed, and convex, and differentiable on the relative interior of its effective domain. The root term may be omitted because its mass is identically one.
\end{definition}

The tree may contain multiple layers of internal nodes. For an internal node $s$, $\mathcal L(s)$ is its set of descendant leaves. Let $\mathcal P(i)$ be the ancestor path of leaf $i$, and let $\mu_s\in\partial\Phi_s(y_s)$ be a node marginal penalty. Under a constraint qualification ensuring the existence of the required multipliers and subgradients, the KKT system can be written as
\begin{equation}
    p_i(\mathbf{V}) = \psi_h \left( V_i - \lambda - \sum_{s \in \mathcal{P}(i)} \mu_s \right),
\end{equation}
with $y_s=\sum_{i\in\mathcal L(s)}p_i$ and $\sum_i p_i=1$. Thus this display is a coupled fixed-point/KKT system, not a closed-form solution, because the $\mu_s$ depend on the probabilities being solved for.

Existence and uniqueness follow from the primal problem independently of this multiplier representation: the feasible simplex is compact, the objective is upper semicontinuous, and the strictly convex leaf term makes the maximizer unique. Equivalently, introducing the node masses as flow variables gives a convex-cost flow problem on an acyclic tree, with unit inflow at the root and flow conservation at internal nodes. The multipliers $(\lambda,\mu_s)$ are the associated node potentials. Conditional on these potentials, each leaf response is a scalar evaluation, while the potentials themselves must be solved jointly by a convex optimization or tree recursion. Appendix~\ref{app:tree_forward_solver} gives an explicit bottom-up and top-down solution scheme.

Hierarchical organization of travel decisions has a long history in structural travel-demand and random-utility models \citep{benakiva1974structure,williams1977formation,daly1978improved}. The present construction is an optimization-based analogue of hierarchical and nested-logit trees \citep{wen2001generalized}. Suitable entropic leaf and node primitives recover members of the nested-logit family, whereas general convex primitives extend beyond the GEV restrictions and need not possess a random-utility representation.

Each leaf has a unique ancestor path, so the formulation excludes overlapping nests and cross-nested substitution patterns. The resulting sparse tree coupling supports the structured algorithms below.

\subsubsection{Specific Instances}

The PUM framework unifies various discrete choice specifications through the choice of $\Lambda(\mathbf{q})$. The following examples illustrate different probability responses and boundary behaviors. Throughout, \(\mu>0\) is a scale parameter.

\paragraph{Multinomial Logit (MNL):} The standard MNL model is recovered when the perturbation function is the negative Shannon entropy scaled by $\mu$:
\begin{equation}
    \Lambda_{\text{Logit}}(\mathbf{q}) = \mu \sum_{i \in \mathcal{C}} q_i \ln q_i.
\end{equation}
We use the convention \(0\ln0=0\), so this perturbation is finite on the entire simplex. The corresponding conjugate function is the familiar Log-Sum-Exp function:
\begin{equation}
    \Omega_{\text{Logit}}(\mathbf{V}) = \mu \ln \left( \sum_{i \in \mathcal{C}} \exp\left(\frac{V_i}{\mu}\right) \right).
\end{equation}

Taking its gradient gives the choice probabilities

$$ { p_i(\mathbf V) =\frac{\exp(V_i/\mu)} {\sum_{j\in\mathcal C}\exp(V_j/\mu)}. } $$

Every alternative receives strictly positive probability for every finite utility vector, even though the perturbation itself remains finite at the boundary.

\paragraph{The Sparsemax Model:} 
Sparsemax replaces entropy with the scaled squared Euclidean norm
\begin{equation}
    \Lambda_{\text{Sparse}}(\mathbf{q}) = \frac{\mu}{2} ||\mathbf{q}||_2^2.
\end{equation}

The resulting probability is the Euclidean projection of $\mathbf V_n/\mu$ onto $\Delta$:
\begin{equation}
    \mathbf{p}_n^{\text{Sparse}} = \arg \min_{\mathbf{q} \in \Delta} \left\| \mathbf{q} - \frac{\mathbf{V}_n}{\mu} \right\|_2^2.
\end{equation}

Unlike MNL, Sparsemax can assign exactly zero probability to some alternatives. It admits the explicit threshold representation:
\begin{equation}
    p_{ni} = \max \left\{ 0, \frac{V_{ni}}{\mu} - \lambda (\mathbf{V}_n, \mu) \right\},
\end{equation}
where $\lambda(\mathbf V_n,\mu)$ enforces $\sum_i p_{ni}=1$. The threshold can be computed by sorting the scaled utilities. 

The scale $\mu$ controls sparsity: smaller values amplify utility differences and concentrate probability on the utility-maximizing alternatives, whereas larger values produce denser probabilities. If the utility maximizer is unique, the probability vector approaches its simplex vertex as $\mu\downarrow0$; if several alternatives tie for the maximum, probability is shared equally among them.

\paragraph{The Cauchy Model:}
A different full-support specification is obtained by constructing a perturbation whose derivative is a scaled Cauchy quantile function. We refer to this specification as the Cauchy model; the name describes the perturbation construction and does not presume an independent-Cauchy random-utility representation. The specific additive regularizer is
\begin{equation}
    \Lambda_{\text{Cauchy}}(\mathbf{q}) = -\frac{\mu}{\pi} \sum_{i \in \mathcal{C}} \ln \left( \cos \left( \pi \left( q_i - \frac{1}{2} \right) \right) \right).
\end{equation}
The scalar terms are interpreted as $+\infty$ at $q_i\in\{0,1\}$. Their derivative is the scaled Cauchy quantile function, so the interior first-order condition $\nabla \Lambda(\mathbf{p}_n) = \mathbf{V}_n - \lambda \mathbf{1}$ gives
\begin{equation}
    \frac{V_{ni} - \lambda}{\mu} = \tan \left( \pi \left( p_{ni} - \frac{1}{2} \right) \right).
\end{equation}
Hence
\begin{equation}
    p_{ni} = \frac{1}{2} + \frac{1}{\pi} \arctan \left( \frac{V_{ni} - \lambda}{\mu} \right).
\end{equation}
where $\lambda$ enforces $\sum_i p_{ni}=1$. The normalization equation is a scalar root-finding problem, and all probabilities are strictly positive for finite utilities. As an alternative’s utility decreases while the other utilities remain fixed, its probability decays polynomially rather than exponentially as in MNL. This provides a specification with slower probability decay for increasingly unattractive alternatives.

These examples illustrate distinct boundary behaviors. Both Shannon entropy and the quadratic perturbation are finite at one-hot choice vectors, whereas the Cauchy perturbation is infinite there. Nevertheless, all three generate well-defined choice probabilities for finite utilities. This distinction becomes relevant in Section \ref{sec:FY} when separating the canonical Fenchel--Young gap from the operational estimation criterion.

\section{Fenchel--Young Estimation of Perturbed Utility Models} \label{sec:FY}

We now consider estimation of utility parameters from observed choices under a fixed perturbation \(\Lambda\). Maximum likelihood provides a natural estimation criterion, but the resulting objective need not be convex in utility parameters and may become infinite when the model assigns zero probability to an observed choice. We introduce an alternative criterion based on the convex surplus function that generates the PUM probabilities. For utilities linear in parameters, this criterion is globally convex and remains finite for every finite utility vector. We first examine the potential difficulties of likelihood estimation, then define the Fenchel--Young criterion and establish its probability-level Fisher consistency. We subsequently explain its geometric interpretation and its relationship to MNL maximum likelihood.

Let $\{(\mathbf X_n,\mathbf y_n)\}_{n=1}^N$ be independent observations, where $\mathbf y_n=\mathbf e_{c_n}\in\{0,1\}^K$ and $c_n\in\mathcal C$ is the chosen alternative. For fixed $\Lambda$, write $\mathbf p(\mathbf X;\boldsymbol\beta):=\mathbf p_\Lambda(\mathbf X\boldsymbol\beta)$ and $\mathbf p_n(\boldsymbol\beta):=\mathbf p(\mathbf X_n;\boldsymbol\beta)$.

\subsection{Limitations of Maximum Likelihood for General PUMs}

Maximum likelihood applies naturally to PUMs, but its computational properties depend on the perturbation. In particular, sparse probability mappings can produce an infinite negative log-likelihood at some parameter values, while even full-support mappings can yield non-convex likelihood objectives. These considerations motivate an alternative criterion that is finite at sparse predictions and convex for fixed perturbations with linear utilities.

\subsubsection{The Zero-Probability Singularity}

Unlike MNL, which assigns strictly positive probabilities, PUMs such as Sparsemax can yield solutions with zero probabilities for some alternatives. For a dataset of $N$ observations, the negative log-likelihood is
\begin{equation}
    \mathcal L_{\mathrm{NLL}}(\boldsymbol\beta)
    =-\sum_{n=1}^N\sum_{i\in\mathcal C}y_{ni}\log p_{ni}(\boldsymbol\beta).
\end{equation}
If $p_{n c_n}(\boldsymbol\beta)=0$, the corresponding term is $+\infty$ because $\lim_{q\downarrow0}-\log q=+\infty$, and its ordinary gradient is unavailable.

This does not invalidate likelihood estimation: such parameter values receive zero likelihood. It does, however, complicate numerical optimization, particularly when an initialization or an intermediate iterate assigns zero probability to an observed choice. The alternative criterion introduced below remains finite at these parameter values.

\subsubsection{Non-Convexity}

Convexity of the forward problem defining choice probabilities does not imply convexity of the negative log-likelihood in utility parameters. For each possible observed choice, convexity of its negative log-probability requires the corresponding probability to be log-concave along the utility directions induced by the parameterization. The following example shows that this property can fail even when every choice probability is strictly positive.

\begin{example}[Non-convexity of the Cauchy negative log-likelihood]
Consider a binary Cauchy PUM with scale $\mu>0$, defined by the strictly convex additive regularizer
\begin{equation}
    \Lambda_{\mathrm{Cauchy}}(\mathbf{q}) = -\frac{\mu}{\pi}\sum_{i=1}^2 \log\Big(\cos\big(\pi(q_i-\tfrac{1}{2})\big)\Big),
    \qquad \mathbf q\in\operatorname{ri}(\Delta),
\end{equation}
extended by $+\infty$ at the two simplex vertices. Let $z=V_1-V_2$ and write $p(z):=p_1(z)$, so $p_2(z)=1-p(z)$. Subtracting the two interior KKT equations gives
\[
    z=2\mu\tan\!\left(\pi\left(p(z)-\tfrac12\right)\right),
    \qquad
    p(z)=\frac12+\frac1\pi\arctan\!\left(\frac{z}{2\mu}\right).
\]
For an observation choosing alternative 1, set $x=z/(2\mu)$. Direct differentiation of
$\ell_1(z)=-\log p(z)$ gives
\begin{equation}
    \ell_1''(z)
    =\frac{1+2\pi x p(z)}{4\pi^2\mu^2(1+x^2)^2p(z)^2}.
\end{equation}
At $x=-1$ (equivalently, $z=-2\mu$), $p(z)=1/4$ and the numerator is
$1-\pi/2<0$. Hence the Cauchy negative log-likelihood is locally non-convex in $z$ and, whenever $z$ is a nonconstant affine function of $\boldsymbol\beta$, along the corresponding parameter direction.
\end{example}

This example establishes that global convexity of likelihood estimation is not a general consequence of the PUM formulation. We therefore consider an alternative criterion derived directly from the convex surplus function.

\subsection{The Loss Function}

For $\mathbf y\in\Delta$, define the canonical extended Fenchel--Young gap
\begin{equation}
    \Gamma_{\mathrm{FY}}(\mathbf y,\mathbf V;\Lambda)
    :=\Omega_\Lambda(\mathbf V)+\Lambda(\mathbf y)-\mathbf y^\top\mathbf V
    \in[0,+\infty].
    \label{eq:fy_gap_def}
\end{equation}

Its nonnegativity follows from the definition of the surplus: evaluating the supremum at $q=y$ gives the Fenchel--Young inequality whenever $\Lambda(y)$ is finite. When $\Lambda(y)=+\infty$, the gap is infinite.

For an observed one-hot choice, $y$ is a simplex vertex. The canonical gap is therefore finite for the Shannon and quadratic perturbations, but infinite for the Cauchy perturbation. To accommodate all these specifications, we define the operational loss directly as

\begin{equation}
    \label{eq:fyloss_def}
    \ell_{\mathrm{FY}}(\mathbf V,\mathbf y;\Lambda)
    :=\Omega_\Lambda(\mathbf V)-\mathbf y^\top\mathbf V,
\end{equation}
and its parameter-level composition
$\ell_{\mathrm{FY}}(\boldsymbol\beta;\mathbf X,\mathbf y)
:=\ell_{\mathrm{FY}}(\mathbf X\boldsymbol\beta,\mathbf y;\Lambda)$ when $\Lambda$ is fixed. Under (A1), this loss is finite for every finite utility vector, regardless of whether the predicted probability of the observed choice is zero.

Whenever $\Lambda(\mathbf y)<\infty$, the canonical gap and the operational loss satisfy
\[
    \Gamma_{\mathrm{FY}}(\mathbf y,\mathbf V;\Lambda)
    =\ell_{\mathrm{FY}}(\mathbf V,\mathbf y;\Lambda)+\Lambda(\mathbf y).
\]
Thus they have the same minimizers and derivatives with respect to $\boldsymbol\beta$. If $\Lambda(\mathbf y)=+\infty$, we do not obtain Eq.~\eqref{eq:fyloss_def} by subtracting an infinite constant. Instead, Eq.~\eqref{eq:fyloss_def} defines a finite criterion whose statistical justification follows directly from the convexity of the surplus. Unlike the canonical gap, the operational loss need not be nonnegative.

\begin{proposition}[Probability-Level Fisher Consistency]
\label{prop:consistency}
Let $\Omega_\Lambda$ be finite and differentiable, and let $\E[\mathbf Y]=\mathbf p_0\in\Delta$. If $\mathbf p_0=\nabla\Omega_\Lambda(\mathbf V_0)$ for some $\mathbf V_0$, then
\[
\argmin_{\mathbf V\in\mathbb R^K}
\E[\ell_{\mathrm{FY}}(\mathbf V,\mathbf Y;\Lambda)]
=\{\mathbf V:\nabla\Omega_\Lambda(\mathbf V)=\mathbf p_0\}.
\]
Thus the loss is strictly Fisher consistent for every reachable choice probability. Under (A1), strict convexity of $\Lambda$ gives the required differentiability.
\end{proposition}

\begin{proof}
See Appendix \ref{app:proofs}.
\end{proof}

The proposition states that every minimizer of the expected loss generates the true probability vector, provided that vector is attainable by the model. It does not assert uniqueness of the utility vector or identification of $\boldsymbol\beta$; parameter identification is addressed separately in Section 5.

This probability-level statement also clarifies the relationship to proper scoring rules. A proper probability score is indexed by a reported probability vector and has expected value minimized at the true distribution. Eq.~\eqref{eq:fyloss_def}, however, is indexed by a utility vector $\mathbf V$. In a full-support setting where the inverse link is unique modulo a common utility shift, it induces the usual proper probability score. At sparse predictions, distinct utility vectors may yield the same probabilities but different realized losses. The general result is therefore expressed as Fisher consistency of the utility-indexed loss at the probability level.

When $\E[\Lambda(\mathbf Y)]<\infty$, adding this expectation shifts the population risk by a finite constant and does not change Fisher consistency or the parameter estimator. No such comparison is needed for boundary-divergent perturbations.

The convex-analytic ingredients are classical \citep{rockafellar1970convex}. The generic Fenchel--Young construction, its proper-composite interpretation, the conjugate/Bregman relations, and the gradient identity follow the framework of \citet{blondel2019learning,blondel2020learning}. Here, we specialize these ideas to PUM estimation; Section 5 develops the associated identification and asymptotic theory, and Section 6 addresses perturbation learning.

The empirical FY estimator solves
\begin{equation}
    \label{eq:erm_formulation}
    \min_{\boldsymbol\beta}\quad
    \mathcal J_N(\boldsymbol\beta):=\frac1N\sum_{n=1}^N
    \ell_{\mathrm{FY}}(\boldsymbol\beta;\mathbf X_n,\mathbf y_n).
\end{equation}

\subsection{Geometric Interpretation of the Fenchel--Young Estimator}
A parameter $\boldsymbol\beta$ maps observed attributes to utilities $\mathbf V=\mathbf X\boldsymbol\beta$, and the surplus gradient maps these utilities to probabilities $\mathbf p=\nabla\Omega(\mathbf V)$. The Fenchel--Young criterion therefore evaluates a prediction using the same convex pair $(\Lambda,\Omega)$ that generates the probability map. Its geometry is intrinsic to the PUM rather than imposed by an unrelated distance on the simplex.

For a realized one-hot choice, the canonical Fenchel--Young gap is evaluated at a vertex of $\Delta$. When $\Lambda(\mathbf y)<\infty$, the operational loss differs from this gap only by the label-dependent constant $-\Lambda(\mathbf y)$, and the gap admits the Bregman representation below. Pointwise loss need not vanish for every realized choice: the statistical statement is instead that the conditional expected loss is minimized by scores inducing the true conditional probability, as established in Proposition~\ref{prop:consistency}.

\subsubsection{Bregman Divergence}

Recall that \(\bar\Lambda\) is the perturbation extended by \(+\infty\) outside the simplex. Since
\(\Omega=\bar\Lambda^*\), conjugacy gives
\[
\mathbf V\in\partial\bar\Lambda(\mathbf p),
\qquad
\Omega(\mathbf V)=\mathbf p^\top\mathbf V-\bar\Lambda(\mathbf p).
\]
Hence
\begin{equation}
\Gamma_{\mathrm{FY}}(\mathbf y,\mathbf V;\Lambda)
=\bar\Lambda(\mathbf y)-\bar\Lambda(\mathbf p)
-(\mathbf y-\mathbf p)^\top\mathbf V
=:D_{\bar\Lambda}^{\mathbf V}(\mathbf y\Vert\mathbf p)\ge0.
\label{eq:fy_generalized_bregman}
\end{equation}
This simplex-constrained Bregman divergence uses a supporting vector of the extended function \(\bar\Lambda\), so it incorporates the simplex constraint and remains valid at sparse predictions. It is finite precisely when \(\Lambda(\mathbf y)<\infty\). The use of \(\bar\Lambda\) is essential on a boundary face: a supporting vector can contain a normal-cone component that is absent from the ordinary gradient formula.

If additionally \(\mathbf p\in\operatorname{ri}(\Delta)\) and \(\Lambda\) is differentiable at
\(\mathbf p\), then \(\mathbf V=\nabla\Lambda(\mathbf p)+\lambda\mathbf1\), so
\begin{equation}
\Gamma_{\mathrm{FY}}(\mathbf y,\mathbf V;\Lambda)
=D_\Lambda(\mathbf y\Vert\mathbf p)
:=\Lambda(\mathbf y)-\Lambda(\mathbf p)
-(\mathbf y-\mathbf p)^\top\nabla\Lambda(\mathbf p).
\label{eq:fy_ordinary_bregman}
\end{equation}
The common-shift multiplier disappears because \((\mathbf y-\mathbf p)^\top\mathbf1=0\). This cancellation is the probability-space counterpart of the fact that adding the same constant to every utility does not change choice probabilities. At a sparse boundary, however, the normal-cone component of a supporting vector may contribute to the gap, so Eq.~\eqref{eq:fy_ordinary_bregman} need not apply. If \(\mathbf y\notin\operatorname{dom}\Lambda\), the probability-space divergence is infinite, but the operational loss and Proposition~\ref{prop:consistency} remain valid.

\begin{figure}[ht!]
    \centering
    \includegraphics[width=0.85\textwidth]{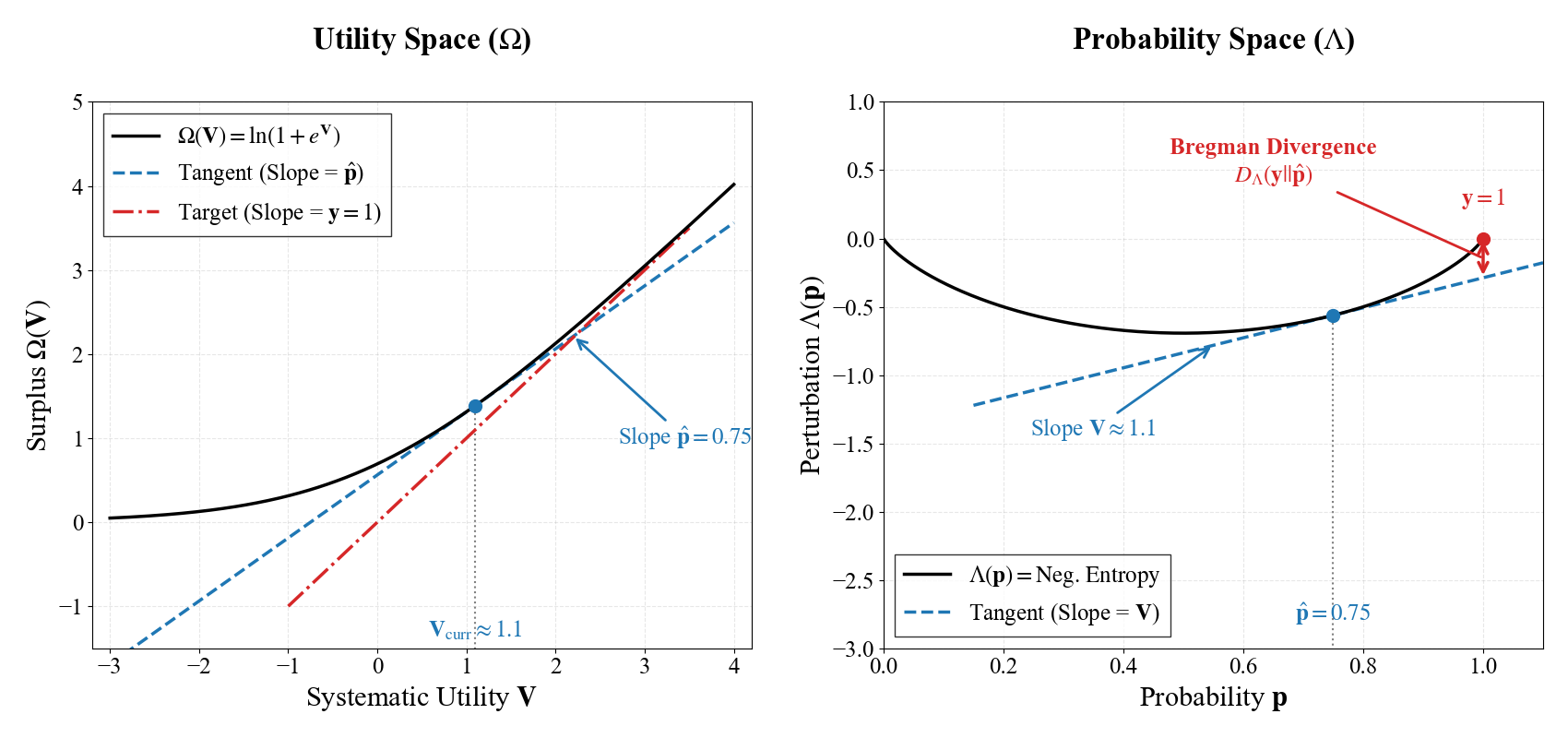}
    \caption{Fenchel--Young geometry. Left: the surplus $\Omega(\mathbf V)$ maps a utility vector to the tangent slope $\mathbf p=\nabla\Omega(\mathbf V)$. Right: at a differentiable interior prediction, the canonical gap is the vertical separation $D_\Lambda(\mathbf y\Vert\mathbf p)$ between $\Lambda(\mathbf y)$ and the supporting affine approximation at $\mathbf p$; on a sparse face, the corresponding object is the simplex-constrained gap $D_{\bar\Lambda}^{\mathbf V}(\mathbf y\Vert\mathbf p)$.}
    \label{fig:fy_geometric_duality}
\end{figure}

\subsubsection{Geometric Picture}
For an interior prediction, $D_\Lambda(\mathbf y\Vert\mathbf p)$ is the vertical gap between $\Lambda(\mathbf y)$ and the supporting affine approximation to $\Lambda$ at $\mathbf p$. At a sparse prediction, the relevant supporting hyperplane belongs to the simplex-constrained function $\bar\Lambda=\Lambda+I_\Delta$ and can include a normal-cone component. Convexity makes either finite gap nonnegative, and strict convexity makes it zero only when $\mathbf y=\mathbf p$. The divergence is oriented: in general $D_\Lambda(\mathbf y\Vert\mathbf p)\ne D_\Lambda(\mathbf p\Vert\mathbf y)$, and it need not satisfy a triangle inequality. Its role is therefore to encode the local geometry of the perturbation, not to define a Euclidean distance on the simplex.

A complementary score-space representation makes the population target explicit without requiring $\Lambda(\mathbf Y)$ to be finite. Let $\mathbf q_0(\mathbf X):=\E[\mathbf Y\mid\mathbf X]$ and suppose that $\mathbf q_0(\mathbf X)=\nabla\Omega(\mathbf V_0)$ for a reachable score $\mathbf V_0$. Conditional on $\mathbf X$, the excess operational risk at any score $\mathbf V$ is
\[
\{\Omega(\mathbf V)-\mathbf q_0^\top\mathbf V\}
-\{\Omega(\mathbf V_0)-\mathbf q_0^\top\mathbf V_0\}
=D_\Omega(\mathbf V\Vert\mathbf V_0)\ge0,
\]
where $D_\Omega(\mathbf V\Vert\mathbf V_0):=\Omega(\mathbf V)-\Omega(\mathbf V_0)-\nabla\Omega(\mathbf V_0)^\top(\mathbf V-\mathbf V_0)$. The zero-regret scores have the same surplus gradient and therefore induce the same probability $\mathbf q_0$; common utility shifts remain observationally irrelevant, and parameter identification is imposed separately.

Across observations, the parameter vector is restricted to the probability-map family $\mathcal M:=\{\mathbf p(\cdot;\boldsymbol\beta)=\nabla\Omega(\mathbf X(\cdot)\boldsymbol\beta):\boldsymbol\beta\in\mathbb R^d\}$. The FY estimator is thus a population-risk projection of the conditional choice distribution onto the reachable family $\mathcal M$, rather than a pointwise attempt to match each one-hot realization. When the canonical gap is finite, this projection can also be read in the Bregman geometry generated by $\Lambda$. When a boundary-divergent perturbation makes that gap infinite at one-hot labels, the finite score-space criterion retains the same probability-level target.

\subsubsection{Equivalence to Maximum Likelihood Estimation for Multinomial Logit}

The duality between maximum-entropy and MNL likelihood formulations has been studied by \citet{donoso2010microeconomic}. In the present Fenchel--Young notation, for the Shannon-entropy regularizer $\Lambda_{\text{Shannon}}(\mathbf{q}) = \mu \sum_i q_i \log q_i$ with a fixed temperature $\mu > 0$, the convex conjugate is the scaled log-sum-exp potential $\Omega_{\mu}(\mathbf{V}) = \mu \log \sum_i \exp(V_i / \mu)$. For the same fixed-temperature parameterization, the operational Fenchel--Young loss is exactly $\mu$ times the negative log-likelihood.

\begin{proposition}[FY vs.\ MLE for Multinomial Logit]
\label{prop:FY_MLE_equiv_logit}
Fix $\mu>0$. Consider the MNL model with utilities $V_{ni}(\boldsymbol{\beta}) = \boldsymbol{\beta}^\top \mathbf{x}_{ni}$ and temperature $\mu$, where the choice probabilities are given by
\[
p_{ni}(\boldsymbol\beta)
=
\frac{\exp\big(V_{ni}(\boldsymbol{\beta}) / \mu\big)}{\sum_{j} \exp\big(V_{nj}(\boldsymbol{\beta}) / \mu\big)}.
\]
For an observation $(\mathbf X_n,\mathbf y_n=\mathbf e_{c_n})$, the associated Fenchel--Young loss is
\[
\ell_{\mathrm{FY}}(\boldsymbol\beta;\mathbf X_n,\mathbf y_n)
=
\mu \log\!\sum_{j} \exp\big(\boldsymbol{\beta}^\top \mathbf{x}_{nj} / \mu\big)
-
\boldsymbol\beta^\top\mathbf x_{n c_n}.
\]
Because $\Lambda_{\mathrm{Shannon}}(\mathbf e_{c_n})=0$, this equals $\mu$ times the negative log-likelihood:
$\ell_{\mathrm{FY}}=\mu\{-\log p_{n c_n}(\boldsymbol\beta)\}$.
Since $\mu$ is fixed and strictly positive, the FY estimator and the maximum-likelihood estimator under the same temperature parameterization have the same argmin set:
\[
\argmin_{\boldsymbol\beta}\mathcal J_{\mathrm{FY}}(\boldsymbol\beta)
=\argmin_{\boldsymbol\beta}\mathcal L_{\mathrm{NLL},\mu}(\boldsymbol\beta).
\]
\end{proposition}

Equivalently, scaled negative Shannon entropy induces scaled Kullback--Leibler divergence:
\[
\ell_{\mathrm{FY}}(\boldsymbol\beta;\mathbf X_n,\mathbf y_n)
=
\mu\,\mathbb D_{\mathrm{KL}}\!\big(\mathbf e_{c_n}\Vert\mathbf p_n(\boldsymbol\beta)\big).
\]
with no additional constant under the convention $0\log0=0$. If the conventional unit-temperature MNL is parameterized by coefficients $\boldsymbol\theta$, then the same probabilities require $\boldsymbol\theta=\boldsymbol\beta/\mu$; consequently the numerical coefficient vectors differ by the scale factor $\mu$ unless $\mu=1$. If $\mu$ is estimated jointly with unrestricted $\boldsymbol\beta$, choice probabilities identify only the ratio $\boldsymbol\beta/\mu$, so a scale normalization is required and the fixed-$\mu$ argmin statement above cannot be applied directly.

\subsection{Convexity}

The central computational advantage of the Fenchel--Young estimator is that its optimization geometry follows directly from convex conjugacy. It does not require coordinate probabilities to be log-concave or a model-specific Hessian calculation. The result applies to dense and sparse PUM probability maps alike, provided the perturbation is fixed and utilities are affine in the estimated parameters.

\begin{proposition}[Global Convexity of PUM Estimation] \label{prop:convexity}
    For fixed $\Lambda$ and utilities affine in $\boldsymbol\beta$, $\mathcal J_N$ in Eq.~\eqref{eq:erm_formulation} is convex in $\boldsymbol\beta$.
\end{proposition}

\begin{proof}
    The conjugate $\Omega_\Lambda$ is convex. Hence
    $\Omega_\Lambda(\mathbf X_n\boldsymbol\beta)-\mathbf y_n^\top\mathbf X_n\boldsymbol\beta$
    is convex as the sum of a convex function composed with an affine map and a linear function. Averaging preserves convexity.
\end{proof}

Consequently, every local minimizer of $\mathcal J_N$ is global. Convexity alone does not guarantee that a finite minimizer exists or that it is unique. Strict convexity, and hence uniqueness when a minimizer exists, is a separate question governed by the design, the probability map, and any explicit regularization; the identification conditions in Section~\ref{sec:asymptotic} address the corresponding population problem. Evaluating the surplus or its gradient may still require solving the forward PUM problem. Section 4 develops computational methods for this task and for the resulting estimation problem.


\section{Solution Approach} \label{sec:FY_solution}

For any admissible PUM, the gradient with respect to $\boldsymbol\beta$ has the moment form
\begin{equation}
    \label{eq:general_gradient}
    \nabla_{\boldsymbol\beta}\mathcal J_N(\boldsymbol\beta)
    =
    \frac{1}{N} \sum_{n=1}^N
    \left(
        \underbrace{\sum_{i\in\mathcal C}p_{ni}(\boldsymbol\beta)\mathbf x_{ni}}_{\text{Predicted features}}
        \;-\;
        \underbrace{\mathbf x_{n c_n}}_{\text{Observed features}}
    \right).
\end{equation}
At a finite unconstrained minimizer, this gradient vanishes, so predicted and observed feature moments match. The PUM specification enters through $p_{ni}(\boldsymbol\beta)$.

The cost of evaluating $p_{ni}(\boldsymbol\beta)$ depends on $\Lambda$. MNL is closed form, whereas a general perturbation requires solving the PUM primal problem. This embedded solve does not affect convexity but determines the cost and accuracy of a gradient evaluation.

For additive separable PUMs, including Sparsemax and Cauchy, probabilities can be evaluated using the scalar response \(\psi_h\) introduced in Section 2.1.3. The optimal probabilities therefore satisfy
\begin{equation}
    p_{ni}(\boldsymbol\beta)=\psi_h\!\left(V_{ni}(\boldsymbol\beta)-\lambda_n\right),
\end{equation}
where a simplex multiplier $\lambda_n \in \mathbb{R}$ is obtained from the scalar root-finding equation
\begin{equation}
    \sum_{j\in\mathcal C}\psi_h\!\left(V_{nj}(\boldsymbol\beta)-\lambda_n\right)=1.
\end{equation}

The function $G_n(\lambda):=\sum_j\psi_h(V_{nj}-\lambda)-1$ is continuous and nonincreasing, and its endpoint limits bracket zero. Bracketing therefore converges to a root. Strict convexity makes the probability vector unique; the multiplier is unique if $G_n$ is strictly decreasing at the root. Newton's method is locally quadratic when $G_n$ is continuously differentiable near the root, $G_n'(\lambda_n^\star)\ne0$, and $G_n'$ is locally Lipschitz.

Thus the forward problem reduces to a scalar normalization search, provided the scalar responses can be evaluated. If $\mathcal J_N$ attains its minimum and has an $L$-Lipschitz gradient, Nesterov acceleration gives
$\mathcal J_N(\boldsymbol\beta^{(k)})-\mathcal J_N^\star=\mathcal O(k^{-2})$ \citep{nesterov1983method}. This rate assumes exact gradient evaluations. Approximate forward solves introduce gradient errors, and retaining the same rate requires additional control of those errors.

Tree-structured perturbations likewise admit the convex-cost-flow formulation described in Section~\ref{sec:PUM}. Their probabilities can be computed through the value-function recursion in Appendix~\ref{app:tree_forward_solver}, which constructs subtree costs in a bottom-up pass and recovers leaf probabilities in a top-down pass. These structured forward solvers are the default for the substantive additive and tree-structured models studied in this paper.

For a genuinely non-separable perturbation with no exploitable scalar or tree structure, the variational definition of $\Omega$ provides a saddle-point formulation and, under strong duality and attainment, an equivalent moment-constrained convex dual problem. We first derive this equivalence for unrestricted utility parameters and then describe the compact-parameter version used for computation.

\paragraph{Unrestricted saddle-point formulation.}
Lifting the probability vectors to explicit variables gives the unrestricted empirical problem
\begin{equation}
    \label{eq:saddle_point_formulation}
    \min_{\boldsymbol{\beta} \in \mathbb R^d} \max_{\mathbf{q}^{\circ} \in \Delta^{N}}
    \quad
    \mathcal{L}(\boldsymbol{\beta}, \mathbf{q}^{\circ})
=
    \frac{1}{N} \sum_{n=1}^N
    \left[
        \big(\mathbf{q}_n - \mathbf{y}_n\big)^\top \mathbf{V}_n(\boldsymbol{\beta})
        - \Lambda(\mathbf{q}_n)
    \right],
\end{equation}
where $\mathbf{q}^{\circ} = (\mathbf{q}_1, \dots, \mathbf{q}_N)$ denotes the stacked vector of generic simplex variables for all observations, and $\Delta^{N} = \Delta \times \dots \times \Delta$ is the Cartesian product of simplexes.

The objective is linear in $\boldsymbol\beta$ and concave in $\mathbf q^\circ$. Define
\begin{equation}
\begin{aligned}
    \mathbf g_N(\mathbf q^\circ)
    &:=\frac1N\sum_{n=1}^N\mathbf X_n^\top(\mathbf q_n-\mathbf y_n),\\
    \Lambda_N(\mathbf q^\circ)
    &:=\frac1N\sum_{n=1}^N\Lambda(\mathbf q_n),
\end{aligned}
\label{eq:fy_moment_map}
\end{equation}
so that $\mathcal L(\boldsymbol\beta,\mathbf q^\circ)
=\boldsymbol\beta^\top\mathbf g_N(\mathbf q^\circ)
-\Lambda_N(\mathbf q^\circ)$.

\paragraph{Equivalent constrained-convex formulation.}
Assume that the unrestricted empirical FY problem attains its minimum and that the equality-constrained problem below is feasible and satisfies a standard relative-interior qualification; for example, there exists $\mathbf q^\circ\in\operatorname{ri}(\operatorname{dom}\Lambda)^N$ with $\mathbf g_N(\mathbf q^\circ)=\mathbf0$. Strong duality and attainment then give the moment-constrained convex program
\begin{equation}
\begin{aligned}
    \min_{\mathbf q^\circ\in\Delta^N}\quad
    &\Lambda_N(\mathbf q^\circ)\\
    \text{s.t.}\quad
    &\mathbf g_N(\mathbf q^\circ)=\mathbf0.
\end{aligned}
\label{eq:fy_constrained_convex}
\end{equation}
Indeed, $\inf_{\boldsymbol\beta\in\mathbb R^d}
\boldsymbol\beta^\top\mathbf g$ equals zero when $\mathbf g=\mathbf0$ and $-\infty$ otherwise. The Lagrangian of Eq.~\eqref{eq:fy_constrained_convex}, using $\boldsymbol\beta$ as the multiplier with sign convention $-\boldsymbol\beta^\top\mathbf g_N$, is $-\mathcal L$. Thus its dual objective is the negative of the empirical FY objective, and strong duality identifies the same optimal probability vectors and utility parameter. The KKT conditions are
\begin{equation}
    \mathbf g_N(\mathbf q^{\circ,*})=\mathbf0,
    \qquad
    \mathbf0\in\partial\Lambda(\mathbf q_n^*)
    -\mathbf X_n\boldsymbol\beta^*+N_\Delta(\mathbf q_n^*),
    \quad n=1,\ldots,N.
\label{eq:fy_constrained_kkt}
\end{equation}
The second condition is precisely the observation-level PUM optimality condition. Hence Eq.~\eqref{eq:fy_constrained_convex} yields the same optimal probabilities, and its equality-constraint multiplier recovers the unrestricted FY parameter estimate; multiplier uniqueness follows under parameter identification.

For computation, let $\mathcal B\subset\mathbb R^d$ be nonempty, compact, and convex, and replace $\mathbb R^d$ by $\mathcal B$ in Eq.~\eqref{eq:saddle_point_formulation}. Under the same minimax and attainment conditions, eliminating $\boldsymbol\beta$ gives the exact compact-parameter reduction
\[
\min_{\mathbf q^\circ\in\Delta^N}
\left\{\Lambda_N(\mathbf q^\circ)
+\sigma_{\mathcal B}\bigl(-\mathbf g_N(\mathbf q^\circ)\bigr)\right\},
\qquad
\sigma_{\mathcal B}(\mathbf u):=\sup_{\boldsymbol\beta\in\mathcal B}
\boldsymbol\beta^\top\mathbf u,
\]
because $\inf_{\boldsymbol\beta\in\mathcal B}\boldsymbol\beta^\top\mathbf g=-\sigma_{\mathcal B}(-\mathbf g)$. If an unrestricted minimizer lies in $\operatorname{int}(\mathcal B)$, the compact formulation recovers that solution and its moment condition $\mathbf g_N(\mathbf q^{\circ,*})=\mathbf0$. If the boundary of $\mathcal B$ binds, the support-function formulation is the exact reduction, and the hard moment constraint need not hold.

A natural first-order implementation is common to the compact saddle formulation and a safeguarded primal--dual solution of Eq.~\eqref{eq:fy_constrained_convex}. Projected extragradient (PEG), with the multiplier restricted to a sufficiently large $\mathcal B$, gives the updates below. When the unrestricted multiplier lies in $\operatorname{int}(\mathcal B)$, this safeguard does not change the recovered solution. Set $\mathcal Z=\mathcal B\times\Delta^N$ and define the saddle operator
\begin{equation}
    \mathsf F(\boldsymbol\beta,\mathbf q^\circ)
    :=\left(
    \frac1N\sum_{n=1}^N\mathbf X_n^\top(\mathbf q_n-\mathbf y_n),
    \left\{\frac1N\big(\nabla\Lambda(\mathbf q_n)-\mathbf X_n\boldsymbol\beta\big)\right\}_{n=1}^N
    \right).
\end{equation}
Convexity of $\Lambda$ makes $\mathsf F$ monotone. For projected extragradient, assume additionally that $\Lambda$ is differentiable and $\mathsf F$ is $L_F$-Lipschitz on $\mathcal Z$; Lipschitz continuity of $\nabla\Lambda$ on $\Delta$ suffices. Boundary-divergent gradients can instead be handled on an invariant compact interior domain or by a proximal step for $\Lambda+I_\Delta$.

Let $0<\gamma<1/L_F$. At iteration $k$, the prediction step is
\begin{subequations}
\begin{align}
    \tilde{\boldsymbol\beta}^{(k)}
    &=\mathcal P_{\mathcal B}\!\left(\boldsymbol\beta^{(k)}-
      \frac{\gamma}{N}\sum_{n=1}^N\mathbf X_n^\top
      (\mathbf q_n^{(k)}-\mathbf y_n)\right),\\
    \tilde{\mathbf q}_n^{(k)}
    &=\mathcal P_\Delta\!\left(\mathbf q_n^{(k)}+
      \frac{\gamma}{N}\left[\mathbf X_n\boldsymbol\beta^{(k)}-
      \nabla\Lambda(\mathbf q_n^{(k)})\right]\right),
      \quad n=1,\ldots,N.
\end{align}
\end{subequations}
The primal component equals the gradient of the empirical FY objective when $\mathbf q_n=\mathbf p_n(\boldsymbol\beta)$.

The correction step evaluates the operator at the predicted point but projects from the original iterate:
\begin{subequations}
\label{eq:extragradient_update}
\begin{align}
    \boldsymbol\beta^{(k+1)}
    &=\mathcal P_{\mathcal B}\!\left(\boldsymbol\beta^{(k)}-
      \frac{\gamma}{N}\sum_{n=1}^N\mathbf X_n^\top
      (\tilde{\mathbf q}_n^{(k)}-\mathbf y_n)\right),\\
    \mathbf q_n^{(k+1)}
    &=\mathcal P_\Delta\!\left(\mathbf q_n^{(k)}+
      \frac{\gamma}{N}\left[\mathbf X_n\tilde{\boldsymbol\beta}^{(k)}-
      \nabla\Lambda(\tilde{\mathbf q}_n^{(k)})\right]\right),
      \quad n=1,\ldots,N.
\end{align}
\end{subequations}

Let $(\bar{\boldsymbol\beta}^{K},\bar{\mathbf q}^{\circ,K})=K^{-1}\sum_{k=0}^{K-1}(\tilde{\boldsymbol\beta}^{(k)},\tilde{\mathbf q}^{\circ,(k)})$ and let $D_{\mathcal Z}$ be the diameter of $\mathcal Z$. The projected extragradient iterates satisfy the ergodic saddle-gap bound
\begin{equation}
    \max_{\mathbf q^\circ\in\Delta^N}\mathcal L(\bar{\boldsymbol\beta}^{K},\mathbf q^\circ)
    -\min_{\boldsymbol\beta\in\mathcal B}\mathcal L(\boldsymbol\beta,\bar{\mathbf q}^{\circ,K})
    \le \frac{D_{\mathcal Z}^2}{2\gamma K}.
\end{equation}
To verify the bound, write $\mathbf z^k=(\boldsymbol\beta^{(k)},\mathbf q^{\circ,(k)})$, $\tilde{\mathbf z}^k=(\tilde{\boldsymbol\beta}^{(k)},\tilde{\mathbf q}^{\circ,(k)})$, and let $\mathbf z\in\mathcal Z$. The two projection inequalities and Lipschitz continuity give
\begin{equation}
\begin{aligned}
2\gamma\langle\mathsf F(\tilde{\mathbf z}^k),\tilde{\mathbf z}^k-\mathbf z\rangle
&\le \|\mathbf z^k-\mathbf z\|_2^2-\|\mathbf z^{k+1}-\mathbf z\|_2^2\\
&\quad-(1-\gamma^2L_F^2)\|\mathbf z^k-\tilde{\mathbf z}^k\|_2^2.
\end{aligned}
\label{eq:eg_one_step}
\end{equation}
Because $\gamma<1/L_F$, the last term is nonpositive and may be dropped. Summing Eq.~\eqref{eq:eg_one_step} over $k$ telescopes. Convexity in $\boldsymbol\beta$, concavity in $\mathbf q^\circ$, and Jensen's inequality imply, for every $(\boldsymbol\beta,\mathbf q^\circ)\in\mathcal Z$,
\[
\mathcal L(\bar{\boldsymbol\beta}^{K},\mathbf q^\circ)
-\mathcal L(\boldsymbol\beta,\bar{\mathbf q}^{\circ,K})
\le \frac1K\sum_{k=0}^{K-1}
\langle\mathsf F(\tilde{\mathbf z}^k),\tilde{\mathbf z}^k-\mathbf z\rangle
\le \frac{D_{\mathcal Z}^2}{2\gamma K}.
\]
Taking the maximum over $\mathbf q^\circ$ and the minimum over $\boldsymbol\beta$ proves the displayed saddle-gap bound.
Thus the ergodic primal--dual gap is $\mathcal O(K^{-1})$ under the stated compactness and Lipschitz conditions.


\section{Asymptotic Properties of Fenchel--Young Estimators} \label{sec:asymptotic}

This section establishes consistency and asymptotic normality of the FY estimator for utility parameters under a fixed perturbation $\Lambda$. The two results answer different questions. Consistency is a global extremum argument: the empirical FY risk must converge uniformly to a population risk whose probability target is uniquely represented by $\boldsymbol\beta_0$. Asymptotic normality is local: near that target, the empirical first-order condition must admit a stable linear expansion with nonsingular average curvature.

These results apply directly to fixed-perturbation PUM estimation and provide the unregularized fixed-weight benchmark for the PBE inner problem. For the ridge-regularized inner estimator used in PBE, consistency for the same target requires $\theta_N\to0$, while the same centered root-$N$ limit additionally requires a negligible first-order penalty, such as $\sqrt N\theta_N\to0$; otherwise the limit can contain a regularization bias. Learning the perturbation weights introduces a separate profile-estimation problem, whose identification and consistency are treated in Section~\ref{sec:Joint}; the fixed-$\Lambda$ limit distribution below is not asserted as the joint PBE limit distribution.

\begin{lemma}[Continuity and domination of the Fenchel--Young loss]
    \label{lem:fy_cont_dom}
    Let $(\mathbf X,\mathbf Y)\sim\mathbb P_0$, let $\mathcal B\subset\mathbb R^d$ be compact, and set
    $\ell_{\mathrm{FY}}(\boldsymbol\beta;\mathbf X,\mathbf Y)
    =\Omega_\Lambda(\mathbf X\boldsymbol\beta)-\mathbf Y^\top\mathbf X\boldsymbol\beta$.
    Assume that $\Lambda$ is proper, closed, convex, and bounded below on $\Delta$, and that
    $\operatorname{dom}\Lambda\cap\Delta\neq\varnothing$. If
    $\E_{\mathbb P_0}[\|\mathbf X\|]<\infty$, then:
    \begin{enumerate}
        \item $\boldsymbol{\beta}\mapsto
        \ell_{\mathrm{FY}}(\boldsymbol\beta;\mathbf X,\mathbf Y)$
        is continuous on $\mathcal B$ for every $(\mathbf X,\mathbf Y)$;
        \item there exists an integrable envelope $M$ such that
        \[
            \bigl|\ell_{\mathrm{FY}}(\boldsymbol\beta;\mathbf X,\mathbf Y)\bigr|
            \le M(\mathbf X,\mathbf Y),
            \qquad
            \forall \boldsymbol{\beta}\in\mathcal{B},
        \]
        for $\mathbb P_0$-almost every $(\mathbf X,\mathbf Y)$.
    \end{enumerate}
\end{lemma}

\begin{proof}
    See Appendix \ref{app:proofs}.
\end{proof}

\begin{lemma}[Full-support identifiability of the PUM parameter]
    \label{lem:pum_identifiable}
    Suppose that whenever $\mathbf p(\mathbf X;\boldsymbol\beta_1)=\mathbf p(\mathbf X;\boldsymbol\beta_2)$ almost surely, their common value lies in $\operatorname{ri}(\Delta)$ and $\Lambda$ is differentiable there. Assume also that the normalized utility index satisfies
    \[
        \mathbf{V}(\mathbf{X};\boldsymbol{\beta}_1)
        -
        \mathbf{V}(\mathbf{X};\boldsymbol{\beta}_2)
        \in \operatorname{span}\{\mathbf{1}\}
        \quad \text{a.s.}
        \;\Longrightarrow\;
        \boldsymbol{\beta}_1=\boldsymbol{\beta}_2.
    \]
    Then
    \[
        \mathbf p(\mathbf X;\boldsymbol\beta_1)
        =
        \mathbf p(\mathbf X;\boldsymbol\beta_2)
        \quad \text{a.s.}
        \;\Longrightarrow\;
        \boldsymbol{\beta}_1=\boldsymbol{\beta}_2.
    \]
\end{lemma}

\begin{proof}
    See Appendix \ref{app:proofs}.
\end{proof}

Lemma~\ref{lem:pum_identifiable} covers dense PUMs. Sparse PUMs require direct probability-map identification or active-face conditions because inactive alternatives may have nonzero normal-cone multipliers. The two lemmas play distinct roles in the consistency argument: Lemma~\ref{lem:fy_cont_dom} supplies a uniform law of large numbers on the compact parameter space, while probability-map identification converts the Fisher-consistent probability target into a unique utility parameter.

\begin{theorem}[Consistency of the Fenchel--Young Estimator]
    \label{thm:fy_consistency}
    Let $\{(\mathbf X_n,\mathbf Y_n)\}_{n=1}^N$ be i.i.d.\ from $\mathbb P_0$. Suppose $\boldsymbol\beta_0\in\mathcal B$ satisfies
    \[
        \E_{\mathbb P_0}[\mathbf Y\mid\mathbf X]
        =\mathbf p(\mathbf X;\boldsymbol\beta_0)\quad\text{a.s.},
    \]
    Assume the continuity and domination conditions in Lemma~\ref{lem:fy_cont_dom} and probability-map identification:
    \[
        \mathbf p(\mathbf X;\boldsymbol\beta)
        =\mathbf p(\mathbf X;\boldsymbol\beta_0)\quad\text{a.s.}
        \ \Longrightarrow\ 
        \boldsymbol\beta=\boldsymbol\beta_0.
    \]
    Lemma~\ref{lem:pum_identifiable} supplies this condition for full-support models; sparse models may use active-face identification.

    Let
    \[
        \hat{\boldsymbol{\beta}}_N
        \in \arg\min_{\boldsymbol{\beta}\in\mathcal{B}}
        \frac{1}{N}\sum_{n=1}^N
        \ell_{\mathrm{FY}}(\boldsymbol{\beta}; \mathbf{X}_n, \mathbf{Y}_n)
    \]
    be any empirical minimizer. Then $R(\boldsymbol\beta):=\E_{\mathbb P_0}[\ell_{\mathrm{FY}}(\boldsymbol\beta;\mathbf X,\mathbf Y)]$ is uniquely minimized at $\boldsymbol\beta_0$, and
    \[
        \hat{\boldsymbol{\beta}}_N
        \xrightarrow{p}
        \boldsymbol\beta_0
        \quad\text{as } N\to\infty.
    \]
\end{theorem}

\begin{proof}
    See Appendix \ref{app:proofs}.
\end{proof}

The theorem separates calibration from parameter recovery. Proposition~\ref{prop:consistency} gives the conditional probability minimizer, identification rules out observationally equivalent parameter values, and uniform convergence transfers the unique population minimizer to the empirical problem. Under misspecification, the same extremum-estimation argument yields convergence to the unique pseudo-true minimizer of $R(\boldsymbol\beta)$ whenever that minimizer exists and the continuity, domination, and compactness conditions continue to hold; it need not equal a data-generating structural parameter.

Consistency alone does not determine a convergence rate. For linear utilities, the score and its smooth-regime Jacobian have the transparent forms
\[
\mathbf g(\mathbf Z;\boldsymbol\beta)
=\mathbf X^\top\{\mathbf p(\mathbf X;\boldsymbol\beta)-\mathbf Y\},
\qquad
\mathbf H(\mathbf Z;\boldsymbol\beta)
=\mathbf X^\top\nabla^2\Omega_\Lambda(\mathbf X\boldsymbol\beta)\mathbf X.
\]
Thus $\mathbf H_0$ measures local probability sensitivity along identified feature directions, whereas $\mathbf J_0$ measures sampling variation in the feature-moment score. The next theorem allows these matrices to differ and accommodates piecewise-smooth probability maps through an integral linearization and a local Hessian uniform law.

\begin{theorem}[Asymptotic Normality of the Fenchel--Young Estimator]
    \label{thm:fy_asymptotics}
    Let $\mathbf Z=(\mathbf X,\mathbf Y)$,
    $\mathbf g(\mathbf Z;\boldsymbol\beta):=\nabla_{\boldsymbol\beta}\ell_{\mathrm{FY}}(\boldsymbol\beta;\mathbf Z)$, and let $\mathbf H(\mathbf Z;\boldsymbol\beta)$ be a measurable version of its Jacobian. Suppose $\boldsymbol\beta_0\in\operatorname{int}(\mathcal B)$, Theorem~\ref{thm:fy_consistency} holds, and:
    \begin{enumerate}
        \item There is an open neighborhood $\mathcal U$ of $\boldsymbol\beta_0$ such that, for almost every $\mathbf Z$, $\mathbf g(\mathbf Z;\cdot)$ is absolutely continuous along every line segment in $\mathcal U$, and its derivative along such segments is represented almost everywhere by $\mathbf H(\mathbf Z;\cdot)$.
        \item For some compact neighborhood $\mathcal U_0\subset\mathcal U$ of $\boldsymbol\beta_0$, the local Hessian uniform law of large numbers holds:
        \[
            \sup_{\boldsymbol\beta\in\mathcal U_0}
            \left\|\frac1N\sum_{n=1}^N\mathbf H(\mathbf Z_n;\boldsymbol\beta)
            -\E[\mathbf H(\mathbf Z;\boldsymbol\beta)]\right\|
            \xrightarrow{p}0,
        \]
        and $\E[\mathbf H(\mathbf Z;\boldsymbol\beta)]$ is continuous at $\boldsymbol\beta_0$. Almost-sure continuity with an integrable uniform envelope suffices in smooth models; piecewise-smooth models require the displayed ULLN across their active regimes.
        \item The matrix $\mathbf H_0:=\E[\mathbf H(\mathbf Z;\boldsymbol\beta_0)]$ exists and is positive definite.
        \item The score has finite second moment,
        $\E[\|\mathbf g(\mathbf Z;\boldsymbol\beta_0)\|^2]<\infty$; define
        $\mathbf J_0:=\E[\mathbf g(\mathbf Z;\boldsymbol\beta_0)\mathbf g(\mathbf Z;\boldsymbol\beta_0)^\top]$.
    \end{enumerate}

    Then
    \[
        \sqrt{N} (\hat{\boldsymbol{\beta}}_N - \boldsymbol\beta_0) \xrightarrow{\mathrm{d}} \mathcal{N}\left(\mathbf{0}, \mathbf{V}_{\mathrm{sand}}\right),
    \]
    where $\mathbf V_{\mathrm{sand}}:=\mathbf H_0^{-1}\mathbf J_0\mathbf H_0^{-1}$.
\end{theorem}

The sandwich form also gives a direct route to inference. With $\widehat{\mathbf g}_n:=\mathbf g(\mathbf Z_n;\hat{\boldsymbol\beta}_N)$ and $\widehat{\mathbf H}_n:=\mathbf H(\mathbf Z_n;\hat{\boldsymbol\beta}_N)$, define
\[
\widehat{\mathbf H}_N:=\frac1N\sum_{n=1}^N\widehat{\mathbf H}_n,
\qquad
\widehat{\mathbf J}_N:=\frac1N\sum_{n=1}^N\widehat{\mathbf g}_n\widehat{\mathbf g}_n^\top.
\]
When these sample moments are consistent and $\widehat{\mathbf H}_N$ is nonsingular, $\widehat{\mathbf V}_N:=\widehat{\mathbf H}_N^{-1}\widehat{\mathbf J}_N\widehat{\mathbf H}_N^{-1}$ consistently estimates the covariance of the $\sqrt N$ limit, so the covariance of $\hat{\boldsymbol\beta}_N$ is estimated by $\widehat{\mathbf V}_N/N$.

The same logic extends to independent clusters. We assume these clusters are identically distributed; for an independent non-identically distributed triangular array, the same conclusion instead requires a cluster-level Lindeberg condition and convergence of the average cluster Hessian and score covariance to finite limits, with a nonsingular Hessian limit. If observations are grouped into $G$ independent clusters of uniformly bounded size and the criterion is written as an average of cluster contributions, replace the observation-level score and Hessian by their cluster counterparts and form $\mathbf J_0$ from the cluster-score outer product. The rate is then $\sqrt G$ and the resulting covariance estimator is cluster robust. This is the relevant sampling interpretation for repeated-choice panels such as Swissmetro.

Two special cases clarify the covariance:
\begin{itemize}
    \item \textbf{Multinomial Logit:} With fixed temperature $\mu$, the operational FY loss is $\mu$ times the MNL negative log-likelihood. If $\mathcal I_\beta$ denotes Fisher information for $\boldsymbol\beta$ under that parameterization, then $\mathbf H_0=\mu\mathcal I_\beta$ and $\mathbf J_0=\mu^2\mathcal I_\beta$. Thus the Bartlett equality for the scaled FY score holds only when $\mu=1$, but the sandwich variance still simplifies to $\mathcal I_\beta^{-1}$. The FY estimator therefore retains the asymptotic efficiency of MLE for fixed $\mu$.
    
    \item \textbf{Sparsemax:} The probability map is piecewise smooth. Theorem~\ref{thm:fy_asymptotics} applies when the index
    $\mathbf V(\mathbf X;\boldsymbol\beta_0)$ hits a support-switching boundary with probability zero,
    together with the local Hessian ULLN and positive definiteness. Continuous covariates and nondegenerate boundary equations suffice for the zero-boundary-probability condition. In general $\mathbf H_0\ne\mathbf J_0$, so the sandwich form remains.
\end{itemize}

These cases show that sparsity by itself is not an obstacle to root-$N$ inference. What matters is whether support-switching boundaries carry positive probability or otherwise prevent the required local Hessian law, and whether the expected curvature remains nonsingular on the identified parameter space. Away from such boundaries, zero predicted probabilities can coexist with an ordinary sandwich limit.


\section{Learning Perturbation Geometry within a Basis Family} \label{sec:Joint}

Previously, we established the Fenchel--Young estimation framework for a fixed \textit{a priori} perturbation function $\Lambda$. For linear-in-parameter utilities, it gives a globally convex loss in $\boldsymbol\beta$ under any valid perturbation. Uniqueness requires identification or explicit strong convexification, and statistical efficiency is model dependent.

Parametric Basis Estimation (PBE) learns $\Lambda$ in the convex hull of a finite perturbation dictionary while retaining a convex inner utility-estimation problem. Rather than selecting one perturbation from a set of rigid candidates, PBE learns a continuous convex combination of the dictionary elements, producing a data-driven hybrid within the maintained family while preserving the tractability of fixed-perturbation utility estimation.

We use the tree structure in Definition~\ref{defi:hierarchical}, with shared leaf-level penalties and node-specific nest-level penalties at active internal nodes. The shared leaf-level composite governs choice curvature across leaves, while the node-specific nest-level composites represent interactions within nests. This separation retains the conditional structure of the tree while allowing local and aggregate substitution geometry to be learned jointly.

\subsection{Formulation}  \label{subsec:PBE_formulation}

Let $\mathcal H_{\mathrm{leaf}}:=\{h_1,\ldots,h_M\}$ and $\mathcal H_{\mathrm{nest}}:=\{\phi_1,\ldots,\phi_L\}$ be the normalized dictionaries defined below. The leaf-level weights $\boldsymbol\omega_h$ are shared across leaves; the nest-level weights $\boldsymbol\omega_{\phi,s}$ vary by node. Let $\mathcal A_{\mathcal T}:=\mathcal T_{\mathrm{int}}\setminus\{r\}$, where $r$ is the root. Since $y_r\equiv1$, its penalty is constant and omitted. The endpoint anchors $h_m(0)=h_m(1)=\phi_l(0)=\phi_l(1)=0$ remove constant and linear gauges. Define
\begin{equation}
    h_{\boldsymbol\omega_h}(p_i)=\sum_{m=1}^M\omega^h_m h_m(p_i),
    \qquad
    \Phi_{s,\boldsymbol\omega_{\phi,s}}(y_s)=\sum_{l=1}^L\omega^\phi_{s,l}\phi_l(y_s),
    \label{eq:basis_combination}
\end{equation}
and hence the complete PBE perturbation is
\begin{equation}
\begin{aligned}
    \Lambda_{\boldsymbol\omega}(\mathbf p)
    &=\sum_{i\in\mathcal C}h_{\boldsymbol\omega_h}(p_i)
    +\sum_{s\in\mathcal A_{\mathcal T}}
    \Phi_{s,\boldsymbol\omega_{\phi,s}}(y_s(\mathbf p))\\
    &=\sum_{m=1}^M\omega_m^h\sum_{i\in\mathcal C}h_m(p_i)
    +\sum_{s\in\mathcal A_{\mathcal T}}\sum_{l=1}^L
    \omega_{s,l}^\phi\phi_l(y_s(\mathbf p)).
\end{aligned}
    \label{eq:pbe_complete_perturbation}
\end{equation}
Let $J:=M+L|\mathcal A_{\mathcal T}|$ and concatenate the nonnegative weights as $\boldsymbol\omega=((\boldsymbol\omega_h)^\top,(\boldsymbol\omega_\phi)^\top)^\top\in\mathbb R^J$. The admissible set $\mathcal W\subseteq\Delta_J:=\{\boldsymbol\omega\in\mathbb R_+^J:\mathbf1^\top\boldsymbol\omega=1\}$ is nonempty, compact, and convex. Every $\Lambda_{\boldsymbol\omega}$ is assumed proper, closed, and strictly convex on $\Delta$. It suffices to impose $\sum_m\omega_m^h\ge\underline\omega>0$ when each leaf-level composite is strictly convex and the nest-level bases are convex.

The simplex normalization removes the common positive-scale gauge between utilities and perturbations and allocates the normalized perturbation scale across leaf-level and nest-level components.

Because the scale of an FY loss varies with its perturbation, PBE estimates $\boldsymbol\beta$ by FY loss at the inner level and selects $\boldsymbol\omega$ by the Brier score at the outer level.
For a one-hot $\mathbf y$, every leaf-level argument and nest-level mass equals either zero or one.
The endpoint normalizations imposed above therefore give
$\Lambda_{\boldsymbol\omega}(\mathbf y)=0$, so this canonical FY loss coincides with the
operational loss in Eq.~\eqref{eq:fyloss_def}. We retain the notation $\mathbf V_n(\boldsymbol\beta)=\mathbf X_n\boldsymbol\beta$ introduced in Section~\ref{sec:PUM}.
The same observations enter both levels. For $N$ observations, the estimator is
\begin{subequations}
\begin{align}
\min_{\boldsymbol\omega\in\mathcal W}\quad
&\mathcal L_{\mathrm{outer},N}(\boldsymbol\omega)
=\frac1N\sum_{n=1}^N
\left\|\mathbf y_n-\mathbf p(\mathbf X_n;
\hat{\boldsymbol\beta}_N(\boldsymbol\omega),\boldsymbol\omega)\right\|_2^2
\label{eq:pbe_outer}\\
\text{s.t.}\quad
&\hat{\boldsymbol\beta}_N(\boldsymbol\omega)
\in\argmin_{\boldsymbol\beta\in\mathcal B}
\left\{\frac1N\sum_{n=1}^N
\Gamma_{\mathrm{FY}}(\mathbf y_n,\mathbf V_n(\boldsymbol\beta);
\Lambda_{\boldsymbol\omega})+\theta_N\|\boldsymbol\beta\|_2^2\right\}.
\label{eq:pbe_inner}
\end{align}
\end{subequations}
Here $\mathcal B$ is the normalized utility-parameter space. The ridge makes the finite-sample inner solution unique; consistency for the unpenalized target requires $\theta_N\to0$. We fix and report the basis scale $C$ together with $\theta_N$: rescaling $C$ and $\boldsymbol\beta$ by $a>0$ leaves probabilities unchanged, while the equivalent ridge coefficient becomes $\theta_N/a$. Section~\ref{subsec:normalization_sensitivity} verifies this equivariance numerically.

Primitive bases require a common scale for their weights to be comparable. Without normalization, the structural weights can partly reflect arbitrary differences in the magnitudes or ranges of the primitive functions rather than differences in their curvature. The area normalization below therefore standardizes the scale of both leaf-level and nest-level candidates before the simplex weights allocate the perturbation geometry across them.

For each nondegenerate basis, define the area scale
\begin{equation}
    Z(h_m):=-\int_0^1h_m(p)\,dp,
    \qquad Z(\phi_l):=-\int_0^1\phi_l(y)\,dy.
    \label{eq:area_def}
\end{equation}
With $h_m(0)=h_m(1)=0$ and $\phi_l(0)=\phi_l(1)=0$, convexity makes each retained basis non-positive on $(0,1)$, while the imposed nondegeneracy makes its negative integral strictly positive.

Scale $\bar h_m=\alpha_mh_m$ and $\bar\phi_l=\gamma_l\phi_l$ so that
\begin{equation}
    \alpha_mZ(h_m)=C,\quad m=1,\ldots,M,
    \qquad
    \gamma_lZ(\phi_l)=C,\quad l=1,\ldots,L,
    \label{eq:normalization_condition}
\end{equation}
where $C>0$ is fixed. Henceforth $h_m$ and $\phi_l$ denote the normalized primitives. Area matching makes weights comparable, and $\Delta_J$ fixes their common conic scale; statistical identification is addressed below.

\begin{proposition}[Normalized Dictionary Representation]
\label{prop:universal_representability}
Define the shared leaf-level composite bases
\[
    H_m(\mathbf p):=\sum_{i\in\mathcal C}h_m(p_i),
    \qquad m=1,\ldots,M,
\]
and the nest-level bases $\Phi_{s,l}(\mathbf p):=\phi_l(y_s(\mathbf p))$ for $s\in\mathcal A_{\mathcal T}$ and $l=1,\ldots,L$. Let
\[
    \mathcal F_{\mathcal T}
    :=\{H_m:m=1,\ldots,M\}
    \cup\{\Phi_{s,l}:s\in\mathcal A_{\mathcal T},\ l=1,\ldots,L\}.
\]
Suppose the classes $[F]$, $F\in\mathcal F_{\mathcal T}$, are linearly independent in $C(\Delta)/\operatorname{Aff}(\Delta)$; equivalently, $\sum_j a_jF_j(\mathbf p)$ can be affine on $\Delta$ only if every $a_j=0$. Then every nonzero $\Lambda\in\operatorname{cone}(\mathcal F_{\mathcal T})$ has the unique representation, modulo an affine function,
\[
    \Lambda(\mathbf p)=c\Lambda_{\boldsymbol\omega}(\mathbf p),
    \qquad c>0,\quad\boldsymbol\omega\in\Delta_J.
\]
Thus the functional representation uniquely determines its conic scale and normalized weights.
\end{proposition}

\begin{proof}
    See Appendix \ref{app:proofs}.
\end{proof}

\paragraph{Approximation scope of a finite dictionary.}
Proposition~\ref{prop:universal_representability} concerns a fixed dictionary. For approximation, let $\Lambda_0$ and $\widetilde\Lambda$ be finite and continuous on $\Delta$, with $\Lambda_0$ $m$-strongly convex, and suppose
\[
    \sup_{\mathbf p\in\Delta}|\widetilde\Lambda(\mathbf p)-\Lambda_0(\mathbf p)|
    \le\varepsilon.
\]
For any utility vector $\mathbf V$, let $\mathbf p_0(\mathbf V)$ and $\widetilde{\mathbf p}(\mathbf V)$ be their respective PUM optimizers. Optimality and the uniform approximation bound give
\[
    \Lambda_0(\widetilde{\mathbf p})-\widetilde{\mathbf p}^{\top}\mathbf V
    \le \Lambda_0(\mathbf p_0)-\mathbf p_0^{\top}\mathbf V+2\varepsilon.
\]
Strong convexity of the left-hand objective around its constrained minimizer therefore implies
\begin{equation}
    \|\widetilde{\mathbf p}(\mathbf V)-\mathbf p_0(\mathbf V)\|_2
    \le 2\sqrt{\varepsilon/m}.
    \label{eq:pbe_approximation_bound}
\end{equation}
For a sequence $\{\mathfrak L_M\}_{M\ge1}$ of admissible normalized PBE perturbation classes indexed by dictionary richness, define
\[
    \varepsilon_M
    :=\inf_{\widetilde\Lambda\in\mathfrak L_M}
    \sup_{\mathbf p\in\Delta}
    |\widetilde\Lambda(\mathbf p)-\Lambda_0(\mathbf p)|.
\]
Hence Eq.~\eqref{eq:pbe_approximation_bound} implies that probability approximation vanishes whenever $\varepsilon_M\to0$. With a fixed misspecified dictionary, the extremum estimator instead targets its unique pseudo-true profile minimizer.

The ambient shape-constrained spline family is dense in the normalized smooth class used in the basis construction. To see this, let $h$ be a non-affine convex $C^1$ function on $[0,1]$ with $h(0)=h(1)=0$ and $Z(h)=C$. On knot sequences whose mesh tends to zero, the monotone B-spline quasi-interpolants $s_R$ of the continuous nondecreasing derivative $h'$ converge uniformly to $h'$. Subtracting the constant $\int_0^1s_R(u)\,du$ preserves monotonicity and makes the derivative integrate to zero. Hence
\[
    h_R(x):=\int_0^x\left\{s_R(t)-\int_0^1s_R(u)\,du\right\}dt
\]
is convex, anchored, and converges uniformly to $h$. Rescaling by $C/Z(h_R)\to1$ imposes the area normalization. Strict coefficient inequalities can be obtained by mixing in a strictly convex reference spline with weight $\eta_R>0$, where $\eta_R\downarrow0$. The construction applies componentwise to nest-level primitives and hence to the finite tree sum. A PBE sieve is dense when its convex hulls contain these approximants, for example through finite $\varepsilon_R$-nets of a compact bounded-regularity subclass, where $\varepsilon_R>0$ is the covering radius in the uniform norm and $\varepsilon_R\downarrow0$. The experimental maximin dictionary is fixed and finite.

\subsection{Identification of the PBE Model}\label{subsec:pbe_identification}

Proposition~\ref{prop:universal_representability} identifies weights from a known perturbation. Statistical identification also requires injectivity of the conditional probability map on the data support. The identification and consistency results below concern a fixed finite dictionary and establish recovery of the true utility parameters and perturbation weights under correct specification and the stated identification and regularity conditions. Under misspecification, the consistency argument instead targets a unique pseudo-true profile minimizer, provided the corresponding uniqueness and regularity conditions hold. The approximation results above describe the potential of richer dictionaries but do not establish statistical consistency when the dictionary expands with sample size.

Let $\mathbf p_0(\mathbf X):=\E[\mathbf Y\mid\mathbf X]$ denote the conditional choice probability identified from the distribution of observables. Assume correct specification: for some $(\boldsymbol\beta_0,\boldsymbol\omega_0)\in\mathcal B\times\mathcal W$,
\begin{equation}
    \mathbf p_0(\mathbf X)
    =\mathbf p(\mathbf X;\boldsymbol\beta_0,\boldsymbol\omega_0)
    \quad\text{a.s.}
    \label{eq:pbe_correct_specification}
\end{equation}
Utilities are identified only up to a common additive term. Let $\mathbf R\in\mathbb R^{(K-1)\times K}$ be any full-row-rank contrast matrix satisfying $\mathbf R\mathbf1=\mathbf0$, and write $\mathbf V(\mathbf X;\boldsymbol\beta)=\mathbf X\boldsymbol\beta$. Define the feasible utility-direction space $\mathcal T_{\mathcal B}:=\operatorname{span}(\mathcal B-\mathcal B)$ and the structural-gradient matrix
\begin{equation}
    \mathbf G(\mathbf p)
    :=\left[\nabla H_1(\mathbf p),\ldots,\nabla H_M(\mathbf p),
    \{\nabla\Phi_{s,l}(\mathbf p)\}_{s\in\mathcal A_{\mathcal T},\,l=1}^L\right],
    \label{eq:pbe_structural_gradient}
\end{equation}
where $\nabla H_m(\mathbf p)=(h_m'(p_1),\ldots,h_m'(p_K))^\top$ and $\nabla\Phi_{s,l}(\mathbf p)=\phi_l'(y_s)\mathbf a_s$, with $\mathbf a_s$ the indicator vector of leaves descending from node $s$. Let $\mathcal T_{\mathcal W}:=\operatorname{span}(\mathcal W-\mathcal W)$ be the linear space of feasible weight differences; in particular, $\mathbf1^\top\mathbf v=0$ for every $\mathbf v\in\mathcal T_{\mathcal W}$.

\begin{assumption}[Joint Excitation]\label{assump:pbe_joint_excitation}
Suppose $\mathbf p_0(\mathbf X)\in\operatorname{ri}(\Delta)$ almost surely and the bases are differentiable there. Define
\[
    \boldsymbol\Xi_0(\mathbf X)
    :=\left[\mathbf R\mathbf X,
    -\mathbf R\mathbf G(\mathbf p_0(\mathbf X))\right].
\]
For every nonzero $(\mathbf u,\mathbf v)\in\mathcal T_{\mathcal B}\times\mathcal T_{\mathcal W}$,
\begin{equation}
    \E\left[\left\|\boldsymbol\Xi_0(\mathbf X)
    (\mathbf u^\top,\mathbf v^\top)^\top\right\|_2^2\right]>0.
    \label{eq:pbe_joint_rank}
\end{equation}
\end{assumption}

This joint rank condition prevents utility changes from being offset by perturbation changes on the probability region visited by the data.

\begin{theorem}[Joint Identification of Utility and Perturbation Parameters]
\label{thm:pbe_joint_identification}
Suppose Eq.~\eqref{eq:pbe_correct_specification} and Assumption~\ref{assump:pbe_joint_excitation} hold under the normalized PBE parameterization. If $(\boldsymbol\beta,\boldsymbol\omega)\in\mathcal B\times\mathcal W$ satisfies
\[
    \mathbf p(\mathbf X;\boldsymbol\beta,\boldsymbol\omega)
    =\mathbf p_0(\mathbf X)\quad\text{a.s.},
\]
then $\boldsymbol\beta=\boldsymbol\beta_0$ and $\boldsymbol\omega=\boldsymbol\omega_0$.
\end{theorem}

For sparse PUMs, suppose the support lies almost surely in finitely many locally stable patterns and each restricted basis is differentiable on the relative interior of its face. For a pattern $S$, let $\mathbf R_S$ contrast active alternatives and $\mathbf G_S$ collect active-coordinate gradients. With $S_0(\mathbf X):=\{i:p_{0i}(\mathbf X)>0\}$, define
\[
    \boldsymbol\Xi_S(\mathbf X)
    :=\left[\mathbf R_S\mathbf X_S,
    -\mathbf R_S\mathbf G_S(\mathbf p_0(\mathbf X))\right].
\]
If
\[
    \E\left[\left\|\boldsymbol\Xi_{S_0(\mathbf X)}(\mathbf X)
    (\mathbf u^\top,\mathbf v^\top)^\top\right\|_2^2\right]>0
\]
for every nonzero $(\mathbf u,\mathbf v)\in\mathcal T_{\mathcal B}\times\mathcal T_{\mathcal W}$, the joint-identification conclusion also holds. This condition is stated away from support-switching boundaries; alternatives that are inactive almost surely do not identify alternative-specific directions.

Define the unpenalized population inner risk and its argmin correspondence by
\begin{align}
    L(\boldsymbol\beta,\boldsymbol\omega)
    &:=\E\left[\Gamma_{\mathrm{FY}}
    (\mathbf Y,\mathbf V(\mathbf X;\boldsymbol\beta);\Lambda_{\boldsymbol\omega})\right],
    \label{eq:pbe_population_inner}\\
    \mathcal B^\dagger(\boldsymbol\omega)
    &:=\argmin_{\boldsymbol\beta\in\mathcal B}
    L(\boldsymbol\beta,\boldsymbol\omega).
    \label{eq:pbe_population_profile}
\end{align}

\begin{corollary}[Recovery of the Inner Population Target]
\label{cor:pbe_inner_recovery}
Under the probability-level Fisher-consistency conditions of Proposition~\ref{prop:consistency} and the identification conditions of Theorem~\ref{thm:pbe_joint_identification},
\[
    \mathcal B^\dagger(\boldsymbol\omega_0)=\{\boldsymbol\beta_0\}.
\]
\end{corollary}

At $\boldsymbol\omega_0$, the inner problem is the fixed-perturbation FY problem; the outer profile identifies the structural weights.

\begin{assumption}[Profile Well-Posedness]\label{assump:pbe_profile_wellposed}
For every $\boldsymbol\omega\in\mathcal W$, the correspondence
$\mathcal B^\dagger(\boldsymbol\omega)$ is nonempty and singleton. Denote its unique element by
$\boldsymbol\beta^\dagger(\boldsymbol\omega)$.
\end{assumption}

A sufficient primitive condition for Assumption~\ref{assump:pbe_profile_wellposed} is that $\mathcal B$ is compact and convex and, for every $\boldsymbol\omega\in\mathcal W$, $L(\cdot,\boldsymbol\omega)$ is continuous and strictly convex on $\mathcal B$ relative to its affine hull.

Under Assumption~\ref{assump:pbe_profile_wellposed}, define
\begin{equation}
    Q(\boldsymbol\omega)
    :=\E\left[\left\|\mathbf Y-
    \mathbf p(\mathbf X;\boldsymbol\beta^\dagger(\boldsymbol\omega),
    \boldsymbol\omega)\right\|_2^2\right].
    \label{eq:pbe_population_outer}
\end{equation}

\begin{theorem}[Population Identification of PBE]
\label{thm:pbe_population_identification}
Suppose the conditions of Corollary~\ref{cor:pbe_inner_recovery} and Assumption~\ref{assump:pbe_profile_wellposed} hold. Then $\argmin_{\boldsymbol\omega\in\mathcal W}Q(\boldsymbol\omega)=\{\boldsymbol\omega_0\}$, and the unique population bi-level solution is $(\boldsymbol\beta_0,\boldsymbol\omega_0)$.
\end{theorem}

\subsection{Consistency of the PBE Estimator}\label{subsec:pbe_consistency}

Identification is a population property; consistency additionally requires the empirical inner solution and the resulting outer profile to converge uniformly to their population counterparts. Because both levels range over the same compact product class, this can be established using the full sample at both levels. The argument therefore does not require an artificial training--validation split in the PBE theory.

Let $\mathbf Z:=(\mathbf X,\mathbf Y)\sim\mathbb P$ and define the inner FY and outer Brier losses
\begin{align}
    \ell(\mathbf Z;\boldsymbol\beta,\boldsymbol\omega)
    &:=\Gamma_{\mathrm{FY}}(\mathbf Y,
    \mathbf V(\mathbf X;\boldsymbol\beta);\Lambda_{\boldsymbol\omega}),
    \label{eq:pbe_inner_loss_generic}\\
    b(\mathbf Z;\boldsymbol\beta,\boldsymbol\omega)
    &:=\|\mathbf Y-\mathbf p(\mathbf X;\boldsymbol\beta,
    \boldsymbol\omega)\|_2^2.
    \label{eq:pbe_brier_loss_generic}
\end{align}
Write $P_N f:=N^{-1}\sum_{n=1}^N f(\mathbf Z_n)$ and $P f:=\E[f(\mathbf Z)]$. The penalized inner criterion is
\begin{equation}
    L_N^{\theta}(\boldsymbol\beta,\boldsymbol\omega)
    :=P_N\ell(\boldsymbol\beta,\boldsymbol\omega)
    +\theta_N\|\boldsymbol\beta\|_2^2.
    \label{eq:pbe_penalized_inner_empirical}
\end{equation}

\begin{assumption}[Uniform Regularity for PBE]\label{assump:pbe_uniform_regularity}
The following conditions hold:
\begin{enumerate}
    \item $\{\mathbf Z_n\}_{n=1}^N$ are i.i.d. copies of $\mathbf Z$.
    \item $\mathcal B$ is nonempty, compact, and convex, $\mathcal W$ is nonempty and compact, and $\theta_N>0$ with $\theta_N\to0$.
    \item The map $(\mathbf Z,\boldsymbol\beta,\boldsymbol\omega)\mapsto\ell(\mathbf Z;\boldsymbol\beta,\boldsymbol\omega)$ is jointly measurable and is continuous in $(\boldsymbol\beta,\boldsymbol\omega)$ for $\mathbb P$-almost every $\mathbf Z$. There exists an integrable envelope $M_\ell(\mathbf Z)$ such that
    \[
        \sup_{(\boldsymbol\beta,\boldsymbol\omega)\in\mathcal B\times\mathcal W}
        |\ell(\mathbf Z;\boldsymbol\beta,\boldsymbol\omega)|
        \le M_\ell(\mathbf Z)
        \quad\text{a.s.}
    \]
    \item The map $(\mathbf X,\boldsymbol\beta,\boldsymbol\omega)\mapsto
    \mathbf p(\mathbf X;\boldsymbol\beta,\boldsymbol\omega)$ is jointly measurable and is continuous in $(\boldsymbol\beta,\boldsymbol\omega)$ on $\mathcal B\times\mathcal W$ for $\mathbb P_X$-almost every $\mathbf X$, where $\mathbb P_X$ is the marginal law of $\mathbf X$.
\end{enumerate}
\end{assumption}

Since both arguments of the Brier loss lie in $\Delta$, $0\le b\le2$; compactness and continuity therefore yield its uniform law of large numbers. The FY loss uses the stated integrable envelope.

\begin{theorem}[Consistency of the PBE Estimator]\label{thm:pbe_consistency}
Suppose Assumptions~\ref{assump:pbe_profile_wellposed} and \ref{assump:pbe_uniform_regularity} hold and
\begin{equation}
    \mathcal B^\dagger(\boldsymbol\omega_0)=\{\boldsymbol\beta_0\},
    \qquad
    \argmin_{\boldsymbol\omega\in\mathcal W}Q(\boldsymbol\omega)
    =\{\boldsymbol\omega_0\}.
    \label{eq:pbe_population_separation}
\end{equation}
For each $N$, let $\hat{\boldsymbol\beta}_N(\boldsymbol\omega)$ be a measurable minimizer of Eq.~\eqref{eq:pbe_inner}, and let $\hat{\boldsymbol\omega}_N$ be a measurable global minimizer of Eq.~\eqref{eq:pbe_outer}. Then
\[
    \bigl(\hat{\boldsymbol\beta}_N(\hat{\boldsymbol\omega}_N),
    \hat{\boldsymbol\omega}_N\bigr)
    \xrightarrow{p}(\boldsymbol\beta_0,\boldsymbol\omega_0).
\]
\end{theorem}

\begin{proof}
    See Appendix \ref{app:proofs}.
\end{proof}

Corollary~\ref{cor:pbe_inner_recovery} and Theorem~\ref{thm:pbe_population_identification} imply Eq.~\eqref{eq:pbe_population_separation}; the fixed-active-face argument gives the sparse analogue.

The same observations may be used at both levels because the proof controls the empirical process uniformly on $\mathcal B\times\mathcal W$; no sample split is required.

Theorem~\ref{thm:pbe_consistency} concerns a global empirical profile minimizer, whereas the algorithm below guarantees stationarity. Let
$\widehat Q_N(\boldsymbol\omega):=P_N b(\hat{\boldsymbol\beta}_N(\boldsymbol\omega),\boldsymbol\omega)
=\mathcal L_{\mathrm{outer},N}(\boldsymbol\omega)$. More generally, if a computed sequence $\widetilde{\boldsymbol\omega}_N$ satisfies
\[
    \widehat Q_N(\widetilde{\boldsymbol\omega}_N)
    \le \inf_{\boldsymbol\omega\in\mathcal W}\widehat Q_N(\boldsymbol\omega)+o_p(1),
\]
then the same uniform argmin argument yields
$\widetilde{\boldsymbol\omega}_N\xrightarrow{p}\boldsymbol\omega_0$ and
$\hat{\boldsymbol\beta}_N(\widetilde{\boldsymbol\omega}_N)\xrightarrow{p}\boldsymbol\beta_0$.

\subsection{Solution Approach}  \label{subsec:PBE_solution}

The PBE inner problem remains convex for every fixed structural weight vector, but profiling it through the generally non-convex outer Brier criterion creates a different computational problem. The strategy is to solve the inner problem exactly, differentiate its solution map within a stable active regime, and use the resulting hypergradient in projected gradient descent (PGD). We first establish the continuity and piecewise smoothness needed for this construction, then state the corresponding stationary-convergence result.

\begin{proposition}[Continuity and Piecewise Smoothness of the PBE Objective]
    \label{prop:continuity}
    Let $\mathcal W$ be nonempty, compact, and convex and $\mathcal B$ a nonempty closed convex polyhedron. Suppose each basis is continuous, strictly convex, and piecewise $C^2$ on a finite partition. Index regimes by choice supports, the active face of $\mathcal B$, and the primitive pieces. Within each regime, assume the perturbation Hessian is positive definite on every active-face tangent space and the inner solution remains in the relative interior of a fixed face of $\mathcal B$. Assume the regime boundaries have relative Lebesgue measure zero in $\mathcal W$.

    If $\theta_N>0$, the inner solution $\hat{\boldsymbol\beta}_N(\boldsymbol\omega)$ is unique. This solution map and $\mathcal L_{\mathrm{outer},N}$ are continuous on $\mathcal W$, $C^1$ within each regime, and differentiable almost everywhere relative to $\operatorname{aff}(\mathcal W)$.
\end{proposition}

\begin{proof}
    See Appendix \ref{app:proofs}.
\end{proof}

For a fixed finite knot sequence, the spline bases satisfy the finite-piece condition. The remaining assumptions hold, for example, if active-face leaf-level curvature is bounded below by $\kappa>0$ and every switching residual is transverse to zero in the free coordinates of $\boldsymbol\omega$. The switching sets are then finitely many codimension-one $C^1$ surfaces; identical adjacent descriptions are treated as one regime.

To express the constrained inner problem in ordinary coordinates, fix a reference point
$\boldsymbol\beta_{\mathrm{ref}}$ on the relevant face of $\mathcal B$ and let
$\mathbf S_{\beta}$ have orthonormal columns spanning the linear space parallel to that face. Write
\begin{equation}
    \boldsymbol\beta=\boldsymbol\beta_{\mathrm{ref}}+\mathbf S_{\beta}\boldsymbol\gamma,
    \qquad
    \widetilde{\mathbf X}_n:=\mathbf S_{\beta}^{\top}\mathbf X_n^\top.
    \label{eq:pbe_free_coordinates}
\end{equation}
For an interior solution one takes the face to be $\mathcal B$ itself; for a stable boundary solution,
the same construction uses its active face. All derivatives below are taken with respect to the free
coordinate $\boldsymbol\gamma$.

Changes in $\boldsymbol\omega$ affect the outer criterion through two channels. The direct channel changes the probability map while holding the fitted utility coordinate fixed; the indirect channel changes the inner optimum $\hat{\boldsymbol\gamma}_N(\boldsymbol\omega)$. Because this optimum is defined implicitly rather than in closed form, the second channel is obtained by differentiating the inner first-order condition. PGD on $\boldsymbol\omega$ therefore uses the profile hypergradient
\begin{equation}
    \nabla_{\boldsymbol{\omega}} \mathcal{L}_{\mathrm{outer},N}
    = \frac{\partial \mathcal{L}_{\mathrm{outer},N}}{\partial \boldsymbol{\omega}}
    + \left( \frac{d \hat{\boldsymbol{\gamma}}_N}{d \boldsymbol{\omega}} \right)^{\top}
    \nabla_{\boldsymbol\gamma}\mathcal L_{\mathrm{outer},N}
    \big|_{\boldsymbol\gamma=\hat{\boldsymbol\gamma}_N}.
    \label{eq:chain_rule}
\end{equation}
The first term holds the fitted utility coordinate fixed. Define the sum-scaled inner criterion
\[
    \mathcal J_{\mathrm{in},N}(\boldsymbol\gamma,\boldsymbol\omega)
    :=\sum_{n=1}^N\Gamma_{\mathrm{FY}}
    (\mathbf y_n,\mathbf X_n(\boldsymbol\beta_{\mathrm{ref}}
    +\mathbf S_\beta\boldsymbol\gamma);\Lambda_{\boldsymbol\omega})
    +N\theta_N\|\boldsymbol\beta_{\mathrm{ref}}+\mathbf S_\beta\boldsymbol\gamma\|_2^2.
\]
Within a smooth fixed-face regime, the inner optimum satisfies
$\nabla_{\boldsymbol\gamma}\mathcal J_{\mathrm{in},N}=\mathbf0$.
Differentiating this equation with respect to $\boldsymbol\omega$ yields
\begin{equation}
    \nabla_{\boldsymbol\gamma\boldsymbol\gamma}^2\mathcal J_{\mathrm{in},N}
    \frac{d \hat{\boldsymbol\gamma}_N}{d \boldsymbol\omega}
    + \nabla_{\boldsymbol\gamma\boldsymbol\omega}^2
    \mathcal J_{\mathrm{in},N}=\mathbf0.
\end{equation}
Let $\mathbf H_{\mathrm{inner}}:=\nabla_{\boldsymbol\gamma\boldsymbol\gamma}^2
\mathcal J_{\mathrm{in},N}$ and
$\mathbf B_{\mathrm{inner}}:=\nabla_{\boldsymbol\gamma\boldsymbol\omega}^2
\mathcal J_{\mathrm{in},N}$. The ridge term makes $\mathbf H_{\mathrm{inner}}$ positive
definite in the free coordinates. The implicit function theorem therefore gives
$d\hat{\boldsymbol\gamma}_N/d\boldsymbol\omega
=-\mathbf H_{\mathrm{inner}}^{-1}\mathbf B_{\mathrm{inner}}$.

Substitution into Eq.~\eqref{eq:chain_rule} gives
\begin{equation}
    \nabla_{\boldsymbol{\omega}} \mathcal{L}_{\mathrm{outer},N}
    = \frac{\partial \mathcal{L}_{\mathrm{outer},N}}{\partial \boldsymbol{\omega}}
    - \mathbf{B}_{\mathrm{inner}}^{\top} \mathbf{H}_{\mathrm{inner}}^{-1}
    \nabla_{\boldsymbol\gamma}\mathcal L_{\mathrm{outer},N}
    \big|_{\boldsymbol\gamma=\hat{\boldsymbol\gamma}_N}.
    \label{eq:hypergradient_final}
\end{equation}

The detailed components of Eq.~\eqref{eq:hypergradient_final} and the tree-structured
matrix--vector products used by the linear solves are given in Appendix~\ref{app:hypergradient_derivation}.

For $\eta>0$, define the projected-gradient mapping of the profiled empirical objective
$\widehat Q_N:=\mathcal L_{\mathrm{outer},N}$ by
\begin{equation}
    \mathcal G_\eta(\boldsymbol\omega)
    :=\frac1\eta\left[\boldsymbol\omega-
    \operatorname{Proj}_{\mathcal W}
    \{\boldsymbol\omega-\eta\nabla\widehat Q_N(\boldsymbol\omega)\}\right].
    \label{eq:pbe_gradient_mapping}
\end{equation}

\begin{theorem}[Stationarity of the PBE Outer Iteration]
\label{thm:pbe_pgd_stationarity}
Suppose $\mathcal W$ is nonempty, compact, and convex and $\widehat Q_N$ is continuously
differentiable relative to $\operatorname{aff}(\mathcal W)$ with an $L_N$-Lipschitz gradient on
$\mathcal W$.
With exact inner solutions and the hypergradient in Eq.~\eqref{eq:hypergradient_final}, let $0<\underline\eta\le\eta_t\le\overline\eta<\infty$ and
\[
    \boldsymbol\omega^{(t+1)}
    =\operatorname{Proj}_{\mathcal W}
    \{\boldsymbol\omega^{(t)}-\eta_t\nabla\widehat Q_N(\boldsymbol\omega^{(t)})\},
\]
with steps chosen to satisfy
\begin{equation}
    \widehat Q_N(\boldsymbol\omega^{(t+1)})
    \le \widehat Q_N(\boldsymbol\omega^{(t)})
    -\frac{\eta_t}{2}\|\mathcal G_{\eta_t}(\boldsymbol\omega^{(t)})\|_2^2.
    \label{eq:pbe_pgd_descent}
\end{equation}
Then $\mathcal G_{\eta_t}(\boldsymbol\omega^{(t)})\to\mathbf0$, every accumulation point $\bar{\boldsymbol\omega}$ satisfies
\begin{equation}
    \mathbf0\in\nabla\widehat Q_N(\bar{\boldsymbol\omega})
    +N_{\mathcal W}(\bar{\boldsymbol\omega}),
    \label{eq:pbe_stationarity_condition}
\end{equation}
where $N_{\mathcal W}$ is the Euclidean normal cone, and for every $T\ge1$,
\[
    \min_{0\le t<T}\|\mathcal G_{\eta_t}(\boldsymbol\omega^{(t)})\|_2^2
    \le \frac{2\{\widehat Q_N(\boldsymbol\omega^{(0)})-
    \inf_{\boldsymbol\omega\in\mathcal W}\widehat Q_N(\boldsymbol\omega)\}}{\underline\eta T}.
\]
\end{theorem}

\begin{proof}
    See Appendix \ref{app:proofs}.
\end{proof}

The smoothness condition holds for dense smooth specifications and for sparse specifications with stable active faces and primitive pieces. Support-switching points, including those created by activating a boundary-divergent basis from zero weight, fall outside this smooth regime.
Every step $\eta_t\le L_N^{-1}$ satisfies Eq.~\eqref{eq:pbe_pgd_descent}. Thus, starting from a fixed
$\eta_0>0$ and reducing the trial step by a factor $\rho\in(0,1)$ until that inequality holds yields
$\eta_t\ge\min\{\eta_0,\rho/L_N\}$, so the theorem applies.

\begin{algorithm}[ht!]
\small
\SetAlgoLined
\DontPrintSemicolon
\caption{Bi-Level PGD for Tree-Structured PBE}
\label{alg:bilevel_pbe}

\KwIn{Data $\mathcal D=\{(\mathbf X_n,\mathbf y_n)\}_{n=1}^N$, dictionaries $\mathcal H_{\mathrm{leaf}},\mathcal H_{\mathrm{nest}}$, tree $\mathcal T$, ridge $\theta_N>0$, initial step $\eta_0$, tolerance $\varepsilon_{\mathrm{tol}}$}
\KwOut{Inner solution $\hat{\boldsymbol\beta}$ and stationary structural weights $\hat{\boldsymbol\omega}$}

\tcc{1. Initialization}
Initialize $\boldsymbol{\omega}^{(0)}$ in the relative interior of $\mathcal W$ and set $t \leftarrow 0$\;

\While{not converged}{
    \tcc{2. Inner Utility Estimation}
    Solve the convex inner problem using the scalar or tree solver when the perturbation is structured (Appendix~\ref{app:tree_forward_solver}), or PEG for a genuinely non-separable perturbation when the conditions in Section~\ref{sec:FY_solution} hold, under fixed $\Lambda_{\boldsymbol{\omega}^{(t)}}$:\;
    $\displaystyle \hat{\boldsymbol\beta}^{(t)} \leftarrow \argmin_{\boldsymbol\beta\in\mathcal B} \left\{ \sum_{n=1}^N \Gamma_{\mathrm{FY}}(\mathbf y_n,\mathbf X_n\boldsymbol\beta;\Lambda_{\boldsymbol\omega^{(t)}}) + N\theta_N\|\boldsymbol\beta\|_2^2 \right\}$\;
    
    \tcc{3. Implicit Differentiation via Active-Face Tree VJPs}
    Compute $\hat{\mathbf p}_n\leftarrow\nabla\Omega_{\Lambda_{\boldsymbol\omega^{(t)}}}(\mathbf X_n\hat{\boldsymbol\beta}^{(t)})$, $n=1,\ldots,N$\;
    $\mathbf{g}_{\mathrm{outer}} \leftarrow \frac{2}{N} \sum_{n=1}^N \widetilde{\mathbf X}_n \mathbf{H}_{\Omega, n} (\hat{\mathbf{p}}_n - \mathbf{y}_n)$\;
    Solve for $\mathbf{z}$ via Conjugate Gradient (applying $\mathbf{H}_{\Omega, n}$ VJPs via tree $\mathcal{T}$):\;
    \Indp $\left[ \sum_{n=1}^N \left( \widetilde{\mathbf X}_n \mathbf{H}_{\Omega, n} \widetilde{\mathbf X}_n^{\top} \right) + 2N\theta_N \mathbf{I} \right] \mathbf{z} = \mathbf{g}_{\mathrm{outer}}$\; \Indm
    
    \tcc{4. Hypergradient Assembly \& Projection}
    $\mathbf{g}_{\mathrm{direct}} \leftarrow -\frac{2}{N} \sum_{n=1}^N \mathbf G(\hat{\mathbf{p}}_n)^\top \mathbf{H}_{\Omega, n} (\hat{\mathbf{p}}_n - \mathbf{y}_n)$\;
    $\mathbf g_{\mathrm{prof}}\leftarrow\mathbf g_{\mathrm{direct}}-\mathbf B_{\mathrm{inner}}^\top\mathbf z$\;
    Choose $\eta_t\le\eta_0$ by backtracking, evaluating each trial profile at its inner optimum,
    until Eq.~\eqref{eq:pbe_pgd_descent} holds\;
    \If{$\|\mathcal G_{\eta_t}(\boldsymbol\omega^{(t)})\|_2\le\varepsilon_{\mathrm{tol}}$}{\textbf{break}\;}
    $\boldsymbol\omega^{(t+1)}\leftarrow\operatorname{Proj}_{\mathcal W}(\boldsymbol\omega^{(t)}-\eta_t\mathbf g_{\mathrm{prof}})$\;
    $t \leftarrow t + 1$\;
}
\Return{$\hat{\boldsymbol{\beta}}^{(t)}, \boldsymbol{\omega}^{(t)}$}
\end{algorithm}

Algorithm~\ref{alg:bilevel_pbe} alternates between the unique convex inner solution and a projected
hypergradient step for the structural weights. Its stopping rule is the projected-gradient mapping,
which is the appropriate first-order residual under the simplex-type constraints on $\boldsymbol\omega$.
Under Theorem~\ref{thm:pbe_pgd_stationarity}, the outer iterates approach the stationary set of the
finite-sample profile criterion. The tree-structured products in Appendix~\ref{app:hypergradient_derivation}
avoid forming dense choice-Jacobian matrices and make each implicit linear solve scalable to broad or
deep taxonomies.

The finite PBE dictionaries are constructed before estimation by a maximin design over convex B-spline primitives. For each scalar primitive, its derivative is parameterized as a cubic B-spline $f_m(x)=\mathbf b(x)^\top\mathbf c_m$. The linear constraints $\mathbf v^\top\mathbf c_m=0$ and $\mathbf u^\top\mathbf c_m=C$ impose the endpoint anchor and common area scale, while $\mathbf D_1\mathbf c_m\ge\varepsilon_{\mathrm c}\mathbf1$ enforces strict convexity. We then choose the leaf-level bases, and analogously the nest-level bases, to maximize their minimum pairwise weighted cosine distance under the interior weight $x(1-x)$. Canonical primitives such as Shannon entropy may be included as fixed anchors in a hybrid dictionary. Appendix~\ref{app:Choice_Basis} gives the resulting finite-dimensional maximin program and its trust-region sequential convex solution.


\section{Numerical Experiments} \label{sec:Numerical}

This section evaluates the proposed methods from three complementary perspectives. First, a non-separable quadratic PUM compares projected extragradient with a nested forward-solve and accelerated-gradient baseline. Second, controlled synthetic designs examine whether PBE can recover utility parameters and perturbation geometry as the sample size grows. Third, the Swissmetro application compares four discrete-choice specifications using repeated respondent-grouped nested cross-validation.

\subsection{PEG and Nested-Solve Comparison}

We compare two solution routes on the non-separable quadratic perturbation $\Lambda(\mathbf q)=(\mu/2)\mathbf q^\top\mathbf Q\mathbf q$, where $\mu=1$, $\mathbf Q:=\mathbf A^\top\mathbf A+\mathbf I_K\succ\mathbf0$, and $\mathbf A\in\mathbb R^{K\times K}$ has independent standard-normal entries. This dense perturbation couples all alternatives and has no scalar or tree forward reduction. PEG applies the saddle-point updates in Section~\ref{sec:FY_solution}, treating the observation-level probabilities as explicit variables. For this quadratic specification, Eq.~\eqref{eq:fy_constrained_convex} is a convex quadratic program with linear moment and simplex constraints. Projected primal--dual extragradient applied to its Lagrangian with the same compact multiplier safeguard produces the PEG updates after identifying $\boldsymbol\beta$ with the moment-constraint multiplier. Accordingly, the curve labeled PEG also represents this safeguarded first-order implementation of the constrained-convex option. The baseline instead solves the $N$ simplex-constrained quadratic forward problems at each extrapolated utility parameter and then performs a projected Nesterov step on the resulting FY objective. The forward QPs are solved in a vectorized batch by accelerated projected gradient, warm-started from the preceding outer iteration, to a projected-gradient tolerance of $10^{-9}$.

The experiment uses $N=5{,}000$ observations, $K=d=10$, the common initialization $\boldsymbol\beta^{(0)}=\mathbf0$, and the ball $\mathcal B=\{\boldsymbol\beta:\|\boldsymbol\beta\|_2\le10\}$ for both methods. The high-accuracy unrestricted reference solution lies strictly inside this ball, so the compact safeguard does not change the target solution. PEG uses the Lipschitz-based step rule stated above, while the nested method uses the corresponding smoothness bound for its projected Nesterov step. Both methods run for 500 outer iterations. For a common primal diagnostic, the empirical FY objective at each recorded $\boldsymbol\beta$ is re-evaluated with a high-accuracy forward solve and compared with a high-accuracy nested reference solution; the KKT curve uses the same step-normalized projected mapping for both methods. The reported wall-clock times exclude data generation, construction of the reference solution, diagnostic re-evaluation, and plotting.

\begin{figure}[ht!]
    \centering
    \includegraphics[width=\textwidth]{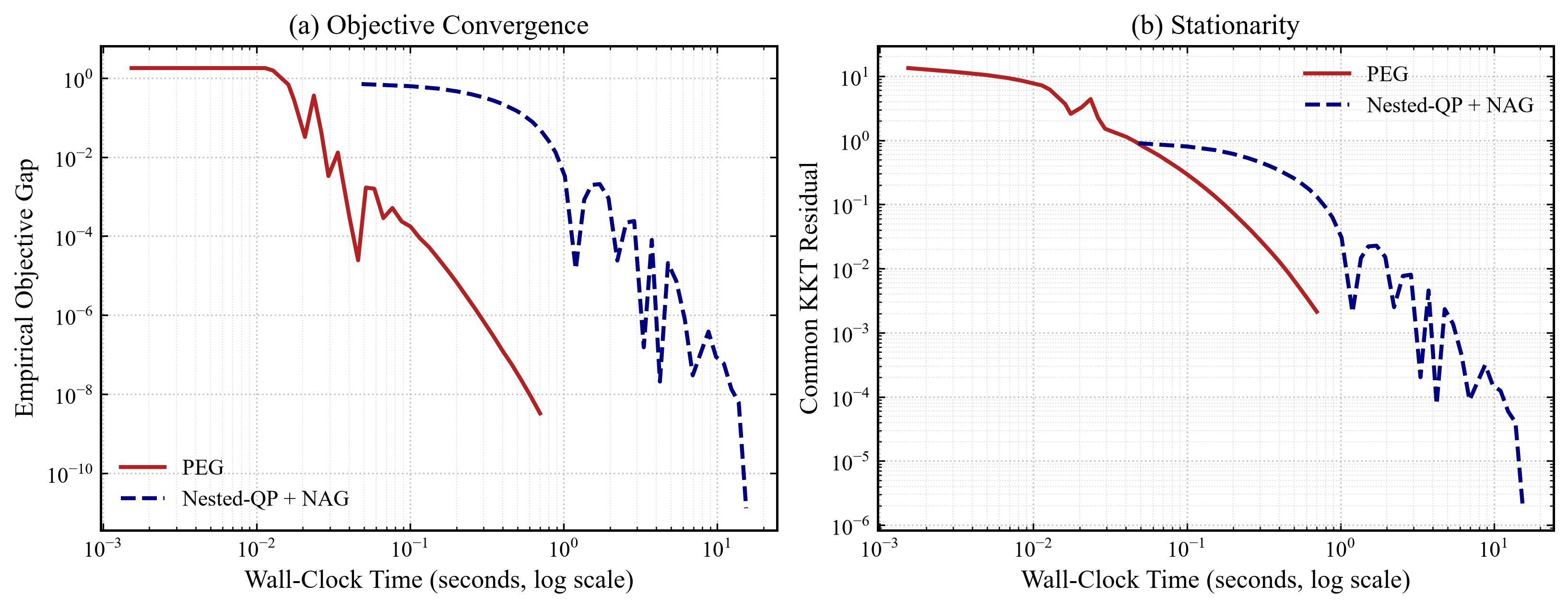}
    \caption{PEG and nested-solve convergence under a dense non-separable quadratic perturbation. The PEG curve also represents safeguarded projected primal--dual extragradient applied to the constrained-convex formulation. Panel (a) reports the empirical FY objective gap relative to a high-accuracy reference solution; Panel (b) reports the common projected KKT residual. Both axes use logarithmic scales, and both methods use the same data, parameter domain, initialization, and 500 outer iterations.}
    \label{fig:fy_convergence}
\end{figure}

Figure~\ref{fig:fy_convergence} shows complementary computational behavior. PEG, equivalently the safeguarded projected primal--dual implementation of Eq.~\eqref{eq:fy_constrained_convex}, reaches an objective gap of $3.26\times10^{-9}$ in 0.70 seconds, while its common KKT residual after 500 iterations is $2.13\times10^{-3}$. The accelerated nested-solve baseline takes 15.42 seconds and attains an objective gap of $1.29\times10^{-11}$ and a KKT residual of $1.84\times10^{-6}$. Thus PEG obtains a near-optimal empirical objective quickly without fully solving every forward problem at each iteration, whereas the nested method spends more computation to achieve substantially tighter final stationarity. The comparison supports the role of PEG and the constrained-convex primal--dual implementation as practical fallbacks for dense non-separable perturbations, while also showing that nested accelerated optimization remains effective when highly accurate forward solutions are desired.

\subsection{Testing the PBE with Synthetic Datasets}

The synthetic experiments cover both additive and tree-structured PUMs. In each design, attributes and utility parameters $\boldsymbol\beta_0$ are generated first and choices are sampled from the induced probability map. The first two designs form the true perturbation from the normalized dictionary used for estimation, making it possible to evaluate recovery of the utility and perturbation blocks directly. The final design independently misspecifies the utility index and the perturbation dictionary to examine held-out predictive and structural robustness.

We report two relative discrepancies. The utility-parameter gap is
\[
\operatorname{RelGap}_\beta:=\frac{\lVert\hat{\boldsymbol\beta}-\boldsymbol\beta_0\rVert_2}{\lVert\boldsymbol\beta_0\rVert_2}.
\]

The function gap is
\[
\operatorname{RelGap}_\Lambda:=
\frac{\int_0^1|\hat h(p)-h_0(p)|\,dp+
\sum_{s\in\mathcal A_{\mathcal T}}\int_0^1|\hat\Phi_s(y)-\Phi_{s,0}(y)|\,dy}
{\int_0^1|h_0(p)|\,dp+
\sum_{s\in\mathcal A_{\mathcal T}}\int_0^1|\Phi_{s,0}(y)|\,dy}.
\]

The first metric records recovery of systematic utility, while the second aggregates deviations in the learned leaf-level and nest-level primitives over their common domain. Because the perturbation scale is normalized in the PBE parameterization, the two quantities provide meaningful finite-sample recovery measures under correct specification.

\subsubsection{Additive Separable PUMs}

Equipped with these scale-invariant metrics, we first investigate the impact of sample size ($N$) on the estimation accuracy of the PBE framework under additive separable PUMs. To ensure a robust evaluation, the experiment corresponding to Table \ref{tab:pbe_results} is conducted over 100 randomly generated problem scenarios that are consistent across all variations of $N$. In this experimental setup, the predefined dictionary consists of three candidate basis perturbation functions ($M = 3$), and the true composite perturbation function for each scenario is explicitly generated as a random convex combination of these bases. Table \ref{tab:pbe_results} summarizes the average performance of the joint estimator across these 100 instances, illustrating how the recovery quality evolves as the sample size increases.

Both gaps generally decrease with sample size. As $N$ rises from 150 to 2,400, the average utility-parameter gap falls from $0.3647$ to $0.1244$, while the perturbation gap falls from $0.0767$ to $0.0475$. The function gap is unchanged between $N=300$ and $N=600$ in these repetitions, so the finite-sample pattern is not strictly monotone at every step, but the overall improvement is clear.

\begin{table}[ht!]
\centering
\caption{Estimation Accuracy of the PBE Framework across Different Sample Sizes}
\label{tab:pbe_results}
\begin{tabular}{ccc}
\toprule
\textbf{Sample Size ($N$)} & \textbf{Relative $\beta$ Gap} & \textbf{Relative Function Gap} \\
\midrule
150  & 0.3647 & 0.0767 \\
300  & 0.2491 & 0.0583 \\
600  & 0.2015 & 0.0583 \\
1,200 & 0.1506 & 0.0498 \\
2,400 & 0.1244 & 0.0475 \\
\bottomrule
\end{tabular}
\end{table}

\begin{figure}[ht!]
    \centering
    \begin{subfigure}{0.48\textwidth}
        \centering
        \includegraphics[width=\linewidth]{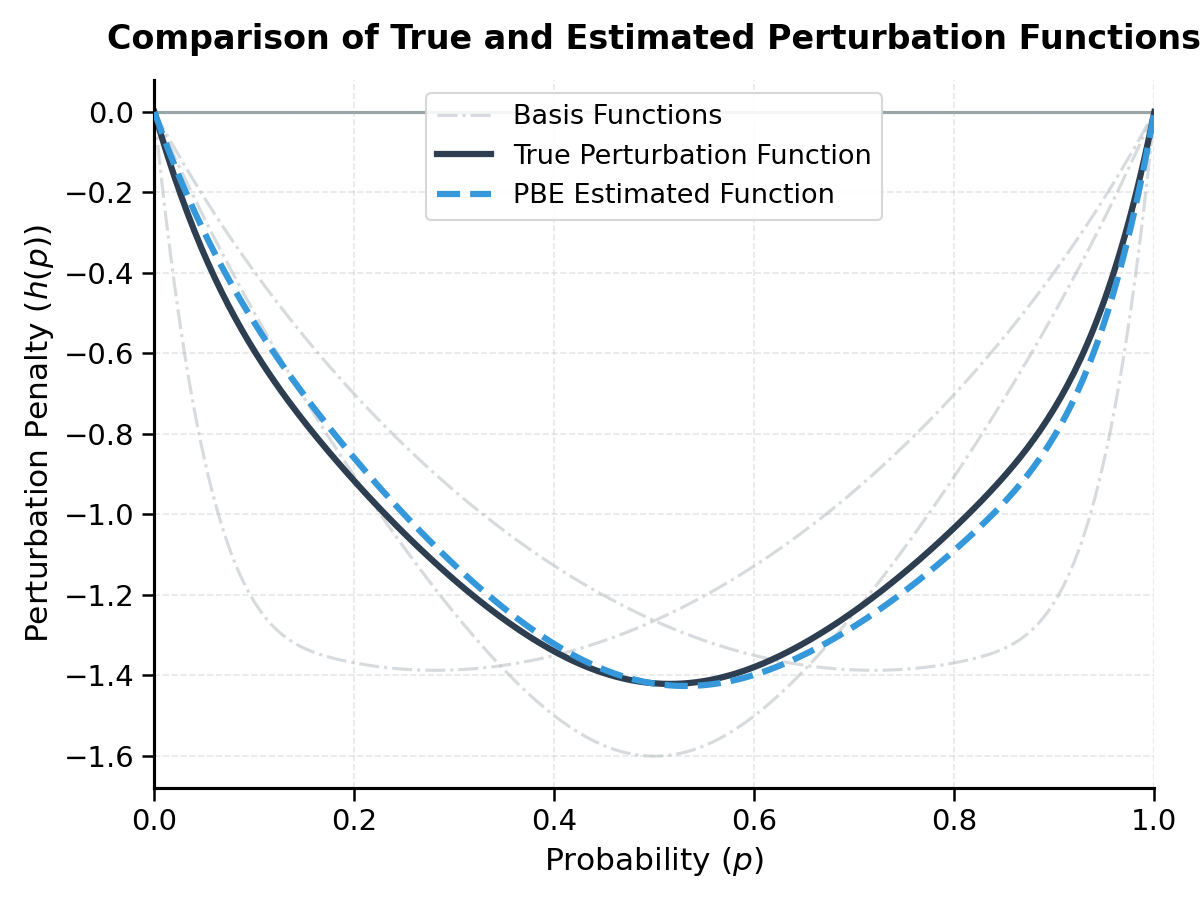} 
        \caption{Sample Size $N=300$}
        \label{fig:pert_N300}
    \end{subfigure}
    \hfill 
    \begin{subfigure}{0.48\textwidth}
        \centering
        \includegraphics[width=\linewidth]{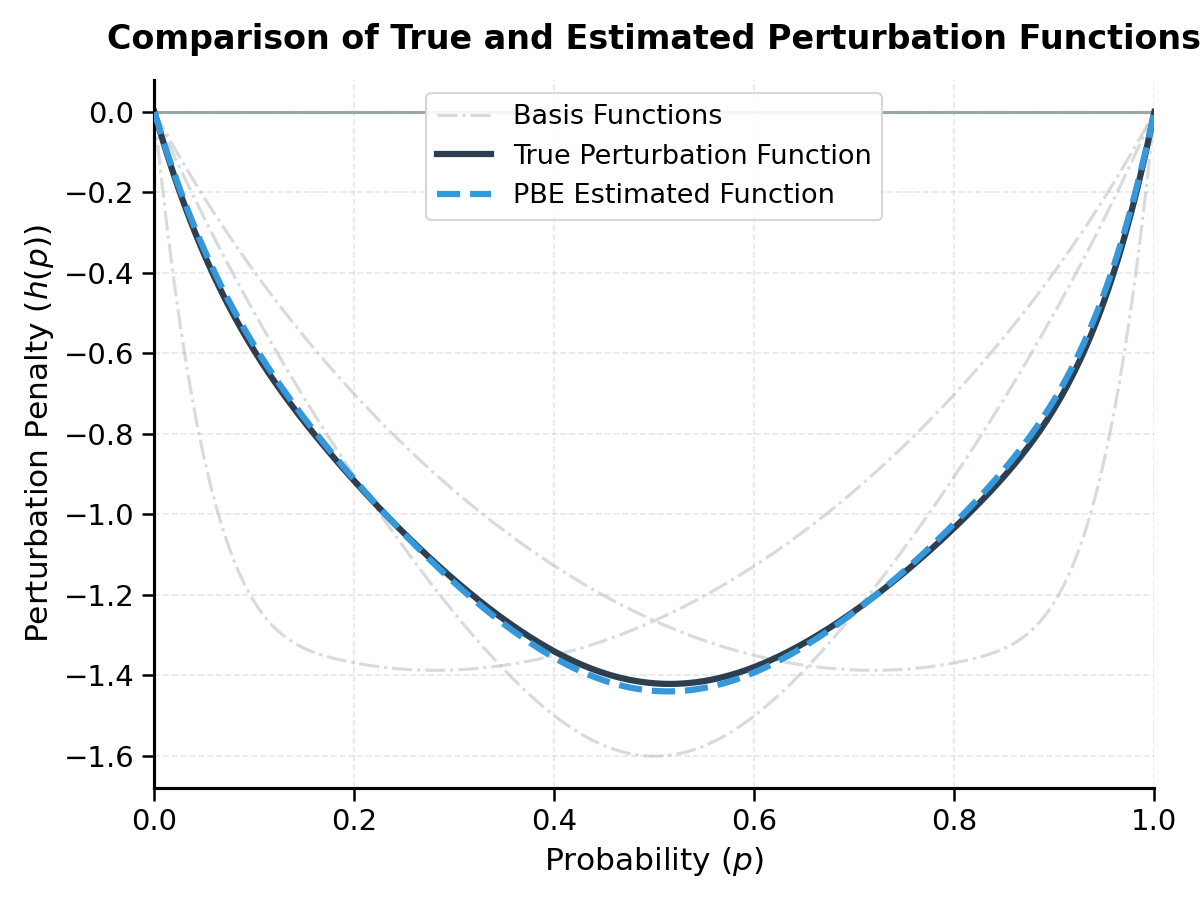}
        \caption{Sample Size $N=2,400$}
        \label{fig:pert_N2400}
    \end{subfigure}
    
    \caption{True and estimated perturbation functions for one instance at $N=300$ and $N=2{,}400$.}
    \label{fig:perturbation_comparison_N}
\end{figure}

Figure~\ref{fig:perturbation_comparison_N} gives the corresponding functional comparison for one representative instance. At $N=300$, the estimated curve captures the broad shape but retains visible deviations from the true perturbation. At $N=2{,}400$, the estimate tracks the true curve more closely across the interior of the simplex. The figure complements the averaged gaps in Table~\ref{tab:pbe_results} by showing what the reported recovery improvement means at the level of the perturbation function itself.

\subsubsection{Tree-Structured PUMs}

For the tree experiment, $K=5$ alternatives form two nontrivial nests, $\mathcal L_1=\{1,2\}$ and $\mathcal L_2=\{3,4\}$, while alternative 5 attaches directly to the root; the utility dimension is $d=4$. We sample attributes and choices from the tree-structured PUM with parameters $(\boldsymbol\beta_0,\boldsymbol\omega_0)$, where $\boldsymbol\omega_0\in\Delta_J$, and estimate them by the bi-level PBE algorithm. This design requires the estimator to separate the common leaf-level curvature from distinct nest-level penalties while recovering the utility coefficients from the same choice observations.

\begin{figure}[ht!]
    \centering
    \begin{subfigure}{0.95\textwidth}
        \centering
        \includegraphics[width=\linewidth]{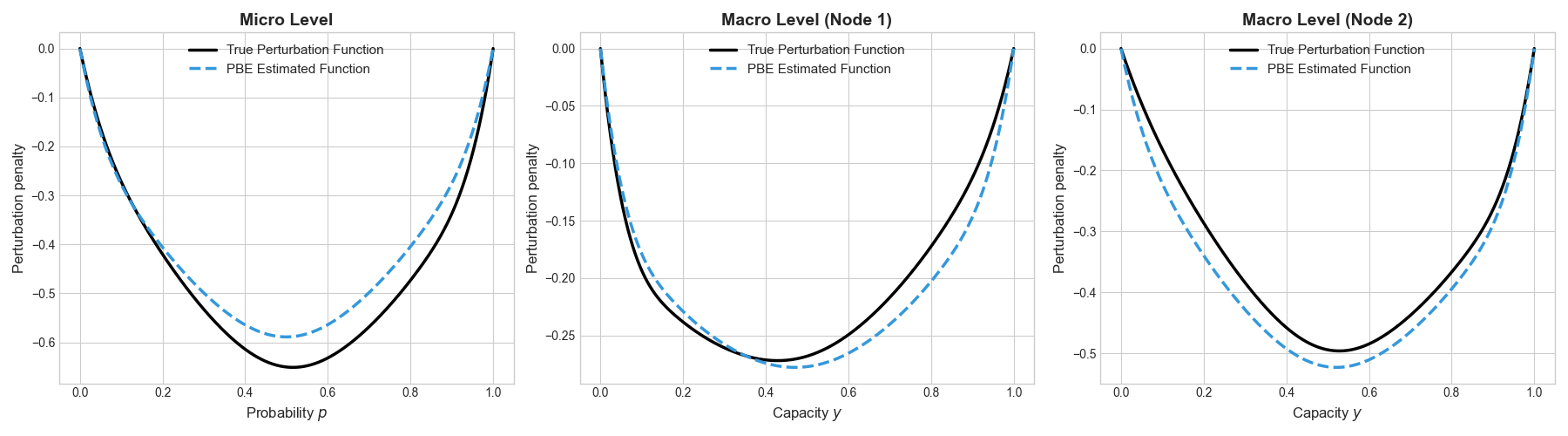} 
        \caption{Sample Size $N=1,000$}
        \label{fig:pert_H1000}
    \end{subfigure}
    \\
    \begin{subfigure}{0.95\textwidth}
        \centering
        \includegraphics[width=\linewidth]{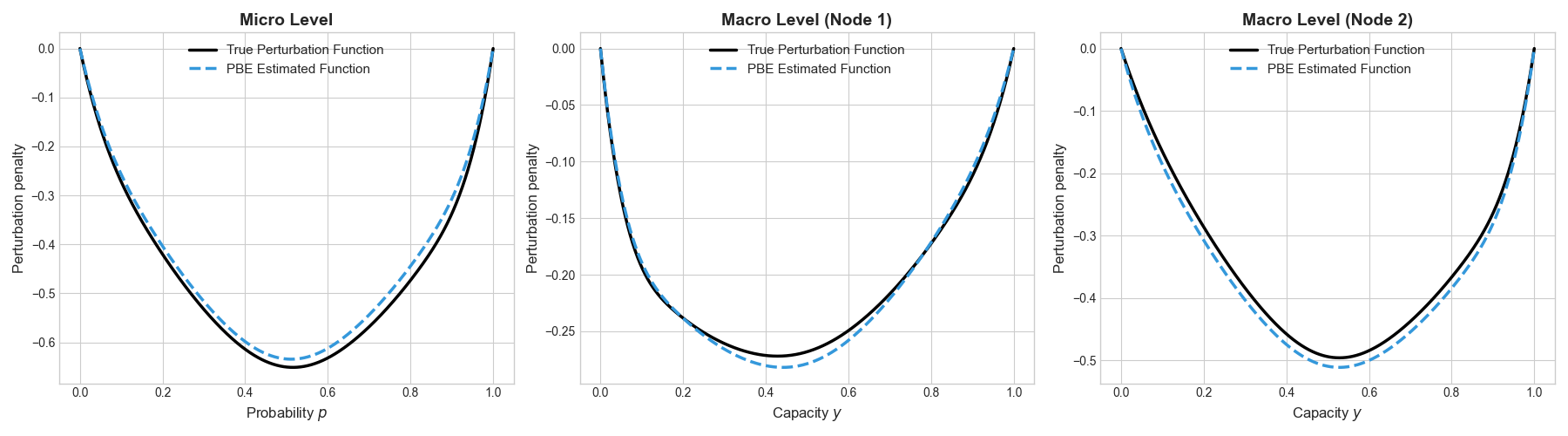}
        \caption{Sample Size $N=10,000$}
        \label{fig:pert_H10000}
    \end{subfigure}
    
    \caption{Comparison of the true and estimated leaf-level and nest-level perturbation functions in a tree-structured PUM for one problem instance at two sample sizes ($N=1{,}000$ and $N=10{,}000$).}
    \label{fig:perturbation_hierarchical}
\end{figure}

Figure~\ref{fig:perturbation_hierarchical} compares recovery at $N=1{,}000$ and $N=10{,}000$. The larger sample produces visibly closer estimates for the shared leaf-level primitive and both nest-level primitives. Across the randomized instances, the average utility-parameter gap decreases from 0.06 to 0.03 and the function gap from 0.16 to 0.07. These results are consistent with the recovery pattern predicted by the identification and consistency analysis, while remaining finite-sample numerical evidence rather than a separate proof of those results.

\subsubsection{Robustness to Utility and Perturbation Misspecification}

We next use a $2\times2$ additive design to separate utility-index misspecification from perturbation-dictionary misspecification. The DGP has $K=3$ alternatives and utilities
\[
V_{ni,0}=\alpha_i+\beta_1x_{ni1}+\beta_2x_{ni2}+\gamma z_{ni},
\qquad
z_{ni}:=\frac{x_{ni1}^2-1}{\sqrt{2}},
\]
with independent standard-normal attributes. The correctly specified utility index includes $z_{ni}$, whereas the misspecified index omits it. The true perturbation is an interior convex combination of four normalized B-spline primitives and the Shannon primitive. The correct dictionary contains these five bases. The misspecified dictionary retains the Shannon primitive and four alternative normalized B-splines, so the two dictionaries have the same dimension and normalization; its best relative $L_1$ approximation error for the true perturbation is $0.02814$.

Each of 50 Monte Carlo repetitions uses $N_{\mathrm{train}}=2{,}000$ observations for estimation and an independent sample of $N_{\mathrm{test}}=10{,}000$ observations for evaluation. The four cells share the same data within each repetition. PBE uses ridge $10^{-4}$, two data-independent starts, and at most 100 outer updates. In addition to held-out Brier score, we report
\[
\operatorname{ProbErr}:=\frac{1}{N_{\mathrm{test}}}
\sum_{n=1}^{N_{\mathrm{test}}}
\left\|\widehat{\mathbf p}_n-\mathbf p_{0n}\right\|_2^2,
\]
which equals the expected excess Brier risk conditional on the test covariates. The relative function gap is defined as above; under misspecification it measures structural distortion from the DGP rather than distance to a consistently estimable parameter.

\begin{table}[ht!]
\centering
\caption{$2\times2$ misspecification experiment. Entries are Monte Carlo means over 50 repetitions. UC and UM denote correct and misspecified utility indices; HC and HM denote correct and misspecified perturbation dictionaries.}
\label{tab:pbe_misspecification}
\begin{tabular}{llccc}
\toprule
\textbf{Utility} & \textbf{Perturbation} & \textbf{Brier Score} & \textbf{ProbErr} & \textbf{Relative Function Gap} \\
\midrule
UC & HC & 0.63502 & 0.00087 & 0.08181 \\
UC & HM & 0.63513 & 0.00098 & 0.15121 \\
UM & HC & 0.64050 & 0.00613 & 0.11692 \\
UM & HM & 0.64057 & 0.00617 & 0.17555 \\
\bottomrule
\end{tabular}
\end{table}

Table~\ref{tab:pbe_misspecification} shows that the two misspecifications have different consequences. With the utility index correctly specified, perturbation misspecification raises Brier score by $0.00011$ (paired 95\% Monte Carlo interval $[0.00008,0.00015]$) and ProbErr by $0.00011$ ($[0.00009,0.00013]$). Its structural effect is more visible: the relative function gap rises from $0.08181$ to $0.15121$. Thus the selected moderate dictionary misspecification has little effect on held-out prediction while materially changing recovery of the perturbation geometry.

Utility misspecification is more consequential. Under the correct perturbation dictionary, it raises Brier score by $0.00548$ ($[0.00519,0.00578]$) and ProbErr by $0.00526$ ($[0.00516,0.00536]$). The doubly misspecified cell has ProbErr $0.00617$, close to the utility-only value $0.00613$; the factorial interaction in ProbErr is $-0.00007$ ($[-0.00010,-0.00004]$). Hence predictive deterioration in this design is driven primarily by the omitted utility nonlinearity. The small perturbation effect reflects robustness to the moderate approximation error considered here, while the function gaps show that predictive robustness does not imply equally accurate structural recovery.

\subsubsection{Sensitivity to Perturbation Normalization}\label{subsec:normalization_sensitivity}

We finally examine whether the numerical PBE fit is invariant to an equivalent change of perturbation scale. Using the correctly specified additive design above, we repeat estimation ten times for $C\in\{0.5,1,2\}$. At scale $C$, every normalized basis is multiplied by $C$, the utility coefficients are interpreted in the coordinates $C\boldsymbol\beta$, and the ridge is set to $\theta_C=\theta_1/C$ with $\theta_1=10^{-4}$. Each scale uses the same training and test samples, two data-independent starts, and at most 100 outer updates. Table~\ref{tab:pbe_normalization_sensitivity} reports differences from the $C=1$ fit after returning the utility and perturbation estimates to the $C=1$ coordinates.

\begin{table}[ht!]
\centering
\small
\setlength{\tabcolsep}{4pt}
\caption{Sensitivity to equivalent perturbation normalizations. Entries are means over ten repetitions. Probability, utility, and function differences are relative to the $C=1$ fit; the last two columns first return the estimates to the $C=1$ coordinates.}
\label{tab:pbe_normalization_sensitivity}
\begin{tabular}{cccccc}
\toprule
$C$ & $\theta_C$ & \textbf{Brier Score} & \textbf{Mean $\max|\Delta p|$} & \textbf{RelGap $\boldsymbol\beta/C$} & \textbf{RelGap $h/C$} \\
\midrule
0.5 & $2\times10^{-4}$ & 0.63505052 & $3.70\times10^{-6}$ & $6.06\times10^{-6}$ & 0 \\
1.0 & $1\times10^{-4}$ & 0.63505051 & 0 & 0 & 0 \\
2.0 & $5\times10^{-5}$ & 0.63505051 & $8.75\times10^{-7}$ & $1.53\times10^{-6}$ & 0 \\
\bottomrule
\end{tabular}
\end{table}

The three normalizations produce the same held-out fit up to numerical tolerance. Across all repetitions and scales, the largest absolute Brier difference from $C=1$ is $5.03\times10^{-8}$ and the largest test-probability difference is $1.68\times10^{-5}$. The structural weights are identical, the restored perturbation functions coincide on the evaluation grid, and the mean restored utility-parameter gap is at most $6.06\times10^{-6}$. Thus changing $C$ together with the corresponding utility and ridge rescaling changes coordinates but not the estimated choice model. Holding the ridge fixed instead would define a different penalized problem.

\subsection{Experiments on the Swissmetro Dataset}

We illustrate the estimators on the Swissmetro dataset, a canonical benchmark in discrete choice analysis \citep{bierlaire2001acceptance}. Collected in 1998, this stated-preference survey records repeated hypothetical choices among private car, existing train service, and Swissmetro.

After filtering the samples to respondents facing a complete choice set and removing incomplete choice observations, we obtain 9,036 choice tasks from 1,004 respondents, each evaluating nine stated-preference scenarios. Repeated tasks from the same respondent need not be independent. For statistical interpretation, one respondent and all nine tasks are therefore treated as one independent cluster; with independent respondents and uniformly bounded cluster size, the uniform-law and extremum arguments above apply to cluster-averaged losses. Let $V_{in}$ denote the systematic utility of alternative $i \in \{\text{Train}, \text{Swissmetro}, \text{Car}\}$ for task $n$. We specify
\begin{equation}
    V_{in} = \beta_{i}^{\mathrm{ASC}} + \beta_{\mathrm{TT}} \mathrm{TT}_{in} + \beta_{\mathrm{CO}} \mathrm{CO}_{in} + \beta_{\mathrm{HE}} \mathrm{HE}_{in} + \boldsymbol{\gamma}_{i}^{\top} \mathbf{Z}_{n}
\end{equation}
where $\mathrm{TT}_{in}$, $\mathrm{CO}_{in}$, and $\mathrm{HE}_{in}$ represent the trip-specific travel time, travel cost, and headway, respectively. The term $\beta_{i}^{\mathrm{ASC}}$ denotes the alternative-specific constant (with the traditional train service normalized as the base alternative), while $\mathbf{Z}_{n}$ captures individual-specific characteristics and contextual modifiers, such as age, luggage requirements, and the ownership of a General Abonnement transit pass.

Leveraging this empirical setup, our primary objective is to test and systematically compare the structural expressivity across four distinct estimation paradigms: the standard MNL model, the Nested Logit model, the additive PBE, and the tree-structured PBE. Specifically, we establish the standard MNL model as our baseline. To evaluate the capacity for capturing unobserved correlations, both the Nested Logit model and the tree-structured PBE are formulated using an identical nesting topology, where the Train and Swissmetro alternatives are grouped into a single public transit nest. For our PBE implementations, we construct the hybrid dictionary using a total of five basis functions ($M=5$). Four of these are data-driven B-splines derived by directly solving the basis choice optimization model, while the remaining basis is explicitly specified as the classic Shannon entropy to anchor the framework with traditional economic interpretability.

We summarize probabilistic fit using the Brier score and the Brier Skill Score (BSS). The BSS rescales a model's Brier score relative to a climatology baseline that predicts constant market shares. This baseline is a reference forecast, not the irreducible conditional choice variance.

The BSS is $\mathrm{BSS}:=1-\mathrm{BS}_{\mathrm{model}}/\mathrm{BS}_{\mathrm{null}}$. In each outer fold, the constant-share reference probabilities are estimated from the corresponding training respondents and then evaluated on the held-out respondents.

Because the 9,036 observations are repeated tasks from 1,004 respondents, random row-wise splitting would leak respondent information. We therefore use repeated respondent-grouped nested cross-validation. The outer evaluation repeats a two-fold respondent partition three times, so every respondent receives one out-of-fold prediction in each repetition. Within each outer training sample, a two-fold respondent-grouped split selects the PBE structural weights and the ridge level from $\{10^{-4},10^{-3}\}$; the same inner folds select the ridge level for the competing models. PBE uses a prespecified interior initialization, at most 100 projected updates, and a projected-gradient tolerance of $10^{-3}$. Each selected model is refitted on the full outer training sample, scored on that training sample for an in-sample diagnostic, and then scored once on the untouched outer test respondents. The reported in-sample Brier score pools these outer-training scores with weights proportional to the number of choice tasks, and the generalization gap is held-out minus in-sample Brier score. MNL and Nested Logit retain likelihood-based parameter fitting but are evaluated using the same outer splits. Paired held-out score differences are averaged within respondent and across repetitions, and 95\% intervals are obtained from 5,000 respondent-cluster bootstrap draws. Model comparisons and inferential conclusions continue to use the held-out scores.

Table~\ref{tab:performance_comparison} reports the repeated respondent-grouped nested cross-validation results. Figure~\ref{fig:PBE_swissmetro} separately visualizes a representative Additive PBE perturbation estimated by respondent-grouped cross-fitting on the full sample; this visualization fit does not enter the outer-test scores.

Table~\ref{tab:performance_comparison} summarizes the held-out empirical results. The evaluation demonstrates the following:

\begin{figure}[ht!]
    \centering
    \includegraphics[width=0.65\textwidth]{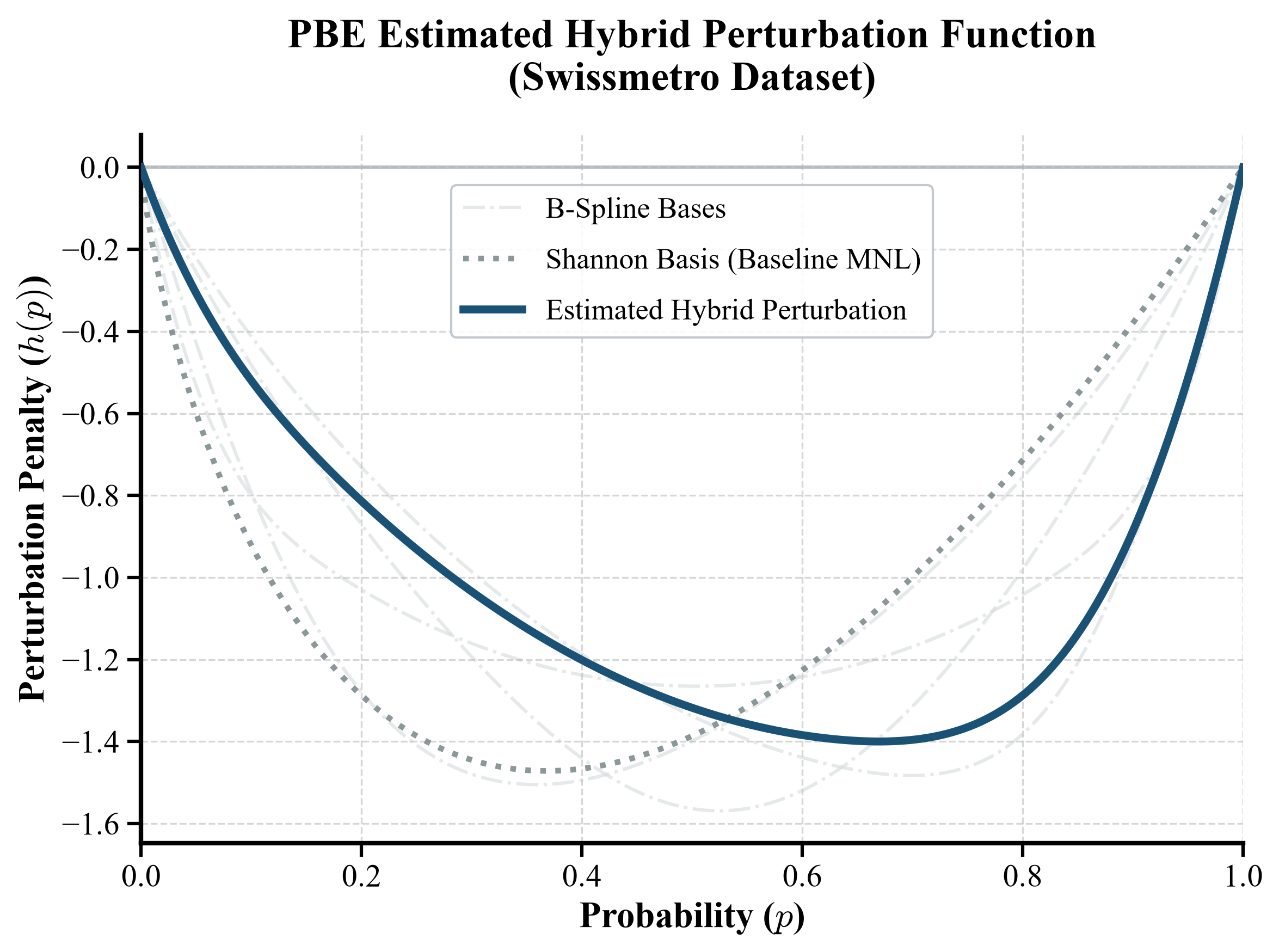}
    \caption{Representative Additive PBE perturbation estimated on all Swissmetro respondents using respondent-grouped cross-fitting. The selected ridge is $10^{-4}$ and the five dictionary weights are $(0.2772,0,0,0.7228,0)$, so the Shannon-entropy weight is zero. This full-sample visualization fit is separate from the outer-test predictions used in Table~\ref{tab:performance_comparison}.}
    \label{fig:PBE_swissmetro}
\end{figure}

\begin{itemize}
    \item Both PBE specifications improve held-out Brier score relative to MNL. The respondent-level paired differences are $-0.00537$ for Additive PBE (95\% CI $[-0.00697,-0.00376]$) and $-0.00539$ for Tree-structured PBE (95\% CI $[-0.00698,-0.00383]$).
    \item The generalization gaps are $0.00263$ for MNL, $0.00308$ for Nested Logit, and $0.00317$ for both PBE specifications. Additive PBE and Tree-structured PBE improve on MNL by about $0.0059$ in-sample and $0.0054$ held-out, so nearly all of the fitted-sample advantage persists on untouched respondents; the PBE gains are not driven by an in-sample comparison.
    \item Nested Logit is slightly worse than MNL in this design: its paired Brier difference is $0.00195$ with 95\% CI $[0.00125,0.00268]$.
    \item Relative to MNL, Additive PBE and Tree-structured PBE increase BSS by 7.08\% and 7.15\%, respectively. Their held-out Brier scores differ by only $0.00002$, providing no material predictive distinction between the two PBE structures in this application.
    \item In the representative full-sample cross-fitted Additive PBE shown in Figure~\ref{fig:PBE_swissmetro}, the first and fourth B-spline bases receive weights 0.2772 and 0.7228, while the Shannon basis receives zero weight. Thus, within the normalized hybrid dictionary and the cross-fitted Brier criterion, the selected additive perturbation places all weight on shapes whose curvature departs from the Shannon benchmark underlying MNL. The figure therefore provides a concrete description of the direction in which PBE relaxes the MNL geometry, while the separate held-out Brier and BSS gains show that this added flexibility improves prediction rather than only in-sample fit. The zero Shannon weight remains specific to this dictionary, normalization, and sample; it does not imply that Shannon entropy is universally invalid or identify a unique behavioral mechanism.
\end{itemize}

\begin{table}[ht!]
\centering
\caption{Repeated respondent-grouped nested cross-validation performance on Swissmetro (three repetitions, two outer folds, and two inner folds). In-sample BS averages the outer-training scores after model selection and refitting; the gap is held-out minus in-sample BS. BSS is computed from held-out predictions using the constant-share reference estimated within each outer training fold.}
\label{tab:performance_comparison}
\begin{tabular}{l c c c c}
\toprule
\textbf{Estimation Paradigm} & \textbf{In-sample BS} & \textbf{Held-out BS} & \textbf{Gap} & \textbf{Held-out BSS} \\
\midrule
MNL (Baseline)               & 0.47035 & 0.47298 & 0.00263 & 0.1384 \\
Nested Logit                 & 0.47186 & 0.47493 & 0.00308 & 0.1349 \\
Additive PBE                 & 0.46444 & 0.46761 & 0.00317 & 0.1482 \\
Tree-structured PBE          & \textbf{0.46442} & \textbf{0.46759} & 0.00317 & \textbf{0.1483} \\
\bottomrule
\end{tabular}
\end{table}

\section{Concluding Remarks} \label{sec:concluding}

This paper develops a Fenchel--Young estimation framework for Perturbed Utility Models. The estimator is built from the same convex conjugacy that generates PUM choice probabilities, so its loss geometry remains aligned with the underlying perturbation. For a fixed perturbation and utilities affine in the parameters, the empirical objective is globally convex and remains finite when a sparse choice map assigns zero probability to the realized alternative. The framework includes MNL maximum likelihood as the Shannon-entropy special case up to a fixed scale, while extending estimation to non-entropic and sparse PUMs. We establish probability-level Fisher consistency, parameter identification, consistency, and asymptotic normality under explicit regularity conditions, and provide structured forward solvers together with a projected-extragradient method for non-separable specifications.

Building on this fixed-perturbation foundation, Parametric Basis Estimation jointly learns utility parameters and a tree-structured perturbation within normalized leaf-level and nest-level dictionaries. The normalization removes affine and common-scale gauges, while the joint excitation condition separates utility variation from changes in perturbation geometry. At the population level, these conditions identify both parameter blocks; under uniform regularity, the global empirical profile estimator is consistent even when the same observations are used at both levels. Computationally, the regularized inner problem has a unique solution, and active-regime implicit differentiation yields scalable hypergradients for the outer problem. The projected outer iteration converges to stationarity under the stated smoothness conditions. The synthetic experiments examine recovery of both utility and perturbation parameters. In Swissmetro, repeated respondent-grouped nested cross-validation shows that both PBE specifications improve held-out Brier score and BSS relative to MNL, while additive and tree-structured PBE perform nearly identically.

Several extensions follow naturally. The current tree structure is specified in advance, so a next step is to learn the topology jointly with its perturbation geometry or to enrich the dictionary with non-tree, non-separable bases. A second direction is to develop scale-comparable criteria for evaluating models generated by different perturbations, since their native Fenchel--Young losses need not be directly comparable in magnitude. Finally, adapting joint utility--perturbation estimation to aggregate market shares would extend the framework beyond the individual-level one-hot observations studied here.



\printbibliography

\newpage 

\setcounter{section}{0}
\renewcommand{\thesection}{EC.\arabic{section}}
\setcounter{figure}{0}
\renewcommand{\thefigure}{EC-\arabic{figure}}

\setcounter{equation}{0}
\renewcommand{\theequation}{EC-\arabic{equation}}
\setcounter{table}{0}
\renewcommand{\thetable}{EC-\arabic{table}}

\Large
\begin{center}
    Electronic Companions for \\
    \textbf{Fenchel--Young Estimators of Perturbed Utility Models}
\end{center}

\normalsize

\section{Proofs} \label{app:proofs}

\subsection{Proof of Proposition \ref{prop:consistency}}

Linearity of expectation gives
\[
    \mathcal R(\mathbf V)=\Omega(\mathbf V)-(\mathbf p_0)^\top\mathbf V.
\]
This is a differentiable convex function, so its minimizers are exactly the solutions of
\[
    \mathbf0=\nabla\mathcal R(\mathbf V)
    =\nabla\Omega(\mathbf V)-\mathbf p_0.
\]
Existence follows from the reachable representative $\mathbf V_0$. Thus the minimizers are precisely the scores inducing $\mathbf p_0$. The utility representative need not be unique because
$\nabla\Omega(\mathbf V+c\mathbf1)=\nabla\Omega(\mathbf V)$; the strict conclusion therefore refers to the induced probability rather than to a unique utility representative.


\subsection{Proof of Lemma \ref{lem:fy_cont_dom}}

Since $\mathcal{B}$ is compact, there exists $R<\infty$ such that
$\|\boldsymbol{\beta}\|\leq R$ for all $\boldsymbol{\beta}\in\mathcal{B}$.
Let
\[
    m_{\Lambda}:=\inf_{\mathbf{q}\in\Delta}\Lambda(\mathbf{q})>-\infty,
\]
and choose $\mathbf{q}_0\in\Delta$ such that $\Lambda(\mathbf{q}_0)<\infty$.

For any $\mathbf{V}\in\mathbb{R}^K$,
\[
    \Omega(\mathbf{V})
    =
    \sup_{\mathbf{q}\in\Delta}
    \{\mathbf{q}^{\top}\mathbf{V}-\Lambda(\mathbf{q})\}
    \leq
    \|\mathbf{V}\|_{\infty}-m_{\Lambda},
\]
while evaluating the supremum at $\mathbf{q}_0$ gives
\[
    \Omega(\mathbf{V})
    \geq
    \mathbf{q}_0^{\top}\mathbf{V}-\Lambda(\mathbf{q}_0)
    \geq
    -\|\mathbf{V}\|_{\infty}-\Lambda(\mathbf{q}_0).
\]
Hence, for some constant $C_{\Lambda}<\infty$,
\[
    |\Omega(\mathbf{V})|
    \leq
    \|\mathbf{V}\|_{\infty}+C_{\Lambda}.
\]
Taking $\mathbf{V}=\mathbf{X}\boldsymbol{\beta}$ and using
$\mathbf{y}\in\{\mathbf{e}_1,\dots,\mathbf{e}_K\}$,
\[
\begin{aligned}
    \bigl|
    \ell_{\mathrm{FY}}(\boldsymbol{\beta};\mathbf{X},\mathbf{y})
    \bigr|
    &\leq
    |\Omega(\mathbf{X}\boldsymbol{\beta})|
    +
    |\mathbf{y}^{\top}\mathbf{X}\boldsymbol{\beta}|  \\
    &\leq
    2\|\mathbf{X}\boldsymbol{\beta}\|_{\infty}
    +
    C_{\Lambda}  \\
    &\leq
    2R\|\mathbf{X}\|
    +
    C_{\Lambda}.
\end{aligned}
\]
Thus
\[
    M(\mathbf{X},\mathbf{y})
    :=
    2R\|\mathbf{X}\|+C_{\Lambda}
\]
is an integrable envelope.

Finally, the above bounds show that $\Omega$ is finite on $\mathbb{R}^K$.
Since $\Omega$ is convex and finite, it is continuous. Because
$\boldsymbol{\beta}\mapsto\mathbf{X}\boldsymbol{\beta}$ is linear and
$-\mathbf{y}^{\top}\mathbf{X}\boldsymbol{\beta}$ is continuous, the map
$\boldsymbol{\beta}\mapsto
\ell_{\mathrm{FY}}(\boldsymbol{\beta};\mathbf{X},\mathbf{y})$
is continuous on $\mathcal{B}$.


\subsection{Proof of Lemma \ref{lem:pum_identifiable}}

Let
    \[
        \mathbf{V}_j(\mathbf{X})
        :=
        \mathbf{V}(\mathbf{X};\boldsymbol{\beta}_j),
        \qquad j=1,2.
    \]
    If
    \[
        \mathbf p(\mathbf X;\boldsymbol\beta_1)
        =
        \mathbf p(\mathbf X;\boldsymbol\beta_2)
        =
        \mathbf{p}(\mathbf{X})
        \quad \text{a.s.},
    \]
    then $\mathbf p(\mathbf X)\in\operatorname{ri}(\Delta)$ by assumption, so the nonnegativity constraints are inactive. The KKT conditions for the PUM primal problem imply that, for some scalars
    $\lambda_1(\mathbf{X})$ and $\lambda_2(\mathbf{X})$,
    \[
        \mathbf{V}_1(\mathbf{X})-\lambda_1(\mathbf{X})\mathbf{1}
        =
        \nabla\Lambda(\mathbf{p}(\mathbf{X})),
    \]
    and
    \[
        \mathbf{V}_2(\mathbf{X})-\lambda_2(\mathbf{X})\mathbf{1}
        =
        \nabla\Lambda(\mathbf{p}(\mathbf{X})).
    \]
    Hence
    \[
        \mathbf{V}_1(\mathbf{X})-\mathbf{V}_2(\mathbf{X})
        =
        \bigl(\lambda_1(\mathbf{X})-\lambda_2(\mathbf{X})\bigr)\mathbf{1}
        \quad \text{a.s.}
    \]
    The normalized utility identifiability condition then gives
    \[
        \boldsymbol{\beta}_1=\boldsymbol{\beta}_2.
    \]


\subsection{Proof of Theorem \ref{thm:fy_consistency}}
Let \(\mathbf p_0(\mathbf X):=\mathbf p(\mathbf X;\boldsymbol\beta_0)\). Correct specification gives
\[
\begin{aligned}
L_{\mathbf X}(\boldsymbol\beta)
&:=\E[\ell_{\mathrm{FY}}(\boldsymbol\beta;\mathbf X,\mathbf Y)\mid\mathbf X]\\
&=\Omega_\Lambda(\mathbf X\boldsymbol\beta)
-\mathbf p_0(\mathbf X)^\top\mathbf X\boldsymbol\beta .
\end{aligned}
\]
By Proposition~\ref{prop:consistency}, this conditional risk is minimized exactly when
\(\mathbf p(\mathbf X;\boldsymbol\beta)=\mathbf p_0(\mathbf X)\). Hence the population risk
\(R(\boldsymbol\beta):=\E[L_{\mathbf X}(\boldsymbol\beta)]\) is minimized at
\(\boldsymbol\beta_0\). Equality of population risks implies equality of the conditional risks
almost surely and therefore equality of the probability maps. The identification assumption gives
\(\boldsymbol\beta=\boldsymbol\beta_0\), so the population minimizer is unique.

Let \(R_N(\boldsymbol\beta):=N^{-1}\sum_{n=1}^N
\ell_{\mathrm{FY}}(\boldsymbol\beta;\mathbf X_n,\mathbf Y_n)\).
Lemma~\ref{lem:fy_cont_dom} and compactness yield
\[
\sup_{\boldsymbol\beta\in\mathcal B}
|R_N(\boldsymbol\beta)-R(\boldsymbol\beta)|\xrightarrow{p}0.
\]
The extremum-estimator consistency theorem (Newey and McFadden, 1994, Theorem~2.1) then gives
\(\hat{\boldsymbol\beta}_N\xrightarrow{p}\boldsymbol\beta_0\).


\subsection{Proof of Theorem \ref{thm:fy_asymptotics}}

Let $\mathbf g_n(\boldsymbol\beta)=\mathbf g(\mathbf Z_n;\boldsymbol\beta)$. Correct specification and Eq.~\eqref{eq:general_gradient} give
\[
    \E[\mathbf g_n(\boldsymbol\beta_0)\mid\mathbf X_n]
    =\mathbf X_n^\top\left(\mathbf p(\mathbf X_n;\boldsymbol\beta_0)
      -\E[\mathbf Y_n\mid\mathbf X_n]\right)=\mathbf0.
\]
Thus the score has mean zero. Consistency, interiority, absolute continuity of the score along local line segments, and the local Hessian ULLN are sufficient for the standard M-estimator linearization: integrating the sample Hessian along the segment from $\boldsymbol\beta_0$ to $\hat{\boldsymbol\beta}_N$ and applying the stated ULLN gives the same expansion as the classical smooth result (compare Newey and McFadden, 1994, Theorem 3.1),
\[
    \sqrt N(\hat{\boldsymbol\beta}_N-\boldsymbol\beta_0)
    =-\mathbf H_0^{-1}\frac1{\sqrt N}
      \sum_{n=1}^N\mathbf g_n(\boldsymbol\beta_0)+o_p(1).
\]
The finite second-moment assumption and the multivariate central limit theorem yield
\[
    \frac1{\sqrt N}\sum_{n=1}^N\mathbf g_n(\boldsymbol\beta_0)
    \xrightarrow{d}\mathcal N(\mathbf0,\mathbf J_0).
\]
Slutsky's theorem therefore gives
\[
    \sqrt N(\hat{\boldsymbol\beta}_N-\boldsymbol\beta_0)
    \xrightarrow{d}
    \mathcal N\!\left(\mathbf0,\mathbf H_0^{-1}\mathbf J_0\mathbf H_0^{-1}\right).
\]


\subsection{Proof of Proposition \ref{prop:universal_representability}}

Index $\mathcal F_{\mathcal T}$ by $j=1,\ldots,J$. Write any nonzero $\Lambda\in\operatorname{cone}(\mathcal F_{\mathcal T})$ as $\Lambda=\sum_j a_jF_j$ with $a_j\ge0$. Setting $c:=\sum_j a_j$ and $\omega_j:=a_j/c$ gives $c>0$, $\boldsymbol\omega\in\Delta_J$, and $\Lambda=c\Lambda_{\boldsymbol\omega}$.

For uniqueness modulo the affine gauge, suppose
\[
    c\sum_{j=1}^J\omega_jF_j(\mathbf p)
    -\widetilde c\sum_{j=1}^J\widetilde\omega_jF_j(\mathbf p)
    =\mathbf a^\top\mathbf p+b
    \quad\text{for all }\mathbf p\in\Delta.
\]
The resulting linear combination is affine, hence the zero equivalence class in $C(\Delta)/\operatorname{Aff}(\Delta)$. Independence of the classes $[F_j]$ implies $c\omega_j=\widetilde c\widetilde\omega_j$ for every $j$. Summing over $j$ and using $\sum_j\omega_j=\sum_j\widetilde\omega_j=1$ gives $c=\widetilde c$, and therefore $\boldsymbol\omega=\widetilde{\boldsymbol\omega}$.

\subsection{Proof of Theorem \ref{thm:pbe_joint_identification}}

Let $(\boldsymbol\beta,\boldsymbol\omega)$ produce $\mathbf p_0(\mathbf X)$ almost surely. Because $\mathbf p_0(\mathbf X)\in\operatorname{ri}(\Delta)$, the nonnegativity constraints are locally inactive. The KKT conditions for the candidate and true parameters yield measurable scalars $\lambda(\mathbf X)$ and $\lambda_0(\mathbf X)$ such that
\begin{align*}
    \mathbf X\boldsymbol\beta
    -\mathbf G(\mathbf p_0(\mathbf X))\boldsymbol\omega
    &=\lambda(\mathbf X)\mathbf1,\\
    \mathbf X\boldsymbol\beta_0
    -\mathbf G(\mathbf p_0(\mathbf X))\boldsymbol\omega_0
    &=\lambda_0(\mathbf X)\mathbf1.
\end{align*}
Subtracting and premultiplying by $\mathbf R$ eliminates the unidentified common-utility component:
\[
    \mathbf R\mathbf X(\boldsymbol\beta-\boldsymbol\beta_0)
    -\mathbf R\mathbf G(\mathbf p_0(\mathbf X))
    (\boldsymbol\omega-\boldsymbol\omega_0)=\mathbf0
    \quad\text{a.s.}
\]
Since $\boldsymbol\beta,\boldsymbol\beta_0\in\mathcal B$ and $\boldsymbol\omega,\boldsymbol\omega_0\in\mathcal W$, their differences belong to $\mathcal T_{\mathcal B}$ and $\mathcal T_{\mathcal W}$, respectively. Consequently,
\[
    \E\left[\left\|\boldsymbol\Xi_0(\mathbf X)
    \begin{pmatrix}\boldsymbol\beta-\boldsymbol\beta_0\\
    \boldsymbol\omega-\boldsymbol\omega_0\end{pmatrix}\right\|_2^2\right]=0.
\]
Assumption~\ref{assump:pbe_joint_excitation} implies that the stacked difference must be zero. Hence $\boldsymbol\beta=\boldsymbol\beta_0$ and $\boldsymbol\omega=\boldsymbol\omega_0$.

For the sparse extension, condition on a stable support $S_0(\mathbf X)=S$ and restrict both KKT systems to $S$. Complementary slackness removes the active-coordinate nonnegativity multipliers, and premultiplication by $\mathbf R_S$ gives the same contrast equation with $\boldsymbol\Xi_S$. The cross-regime rank condition excludes every nonzero feasible direction.

\subsection{Proof of Corollary \ref{cor:pbe_inner_recovery}}

At $\boldsymbol\omega_0$, the inner problem is a fixed-perturbation FY population problem. Applying Proposition~\ref{prop:consistency} conditionally on $\mathbf X$ and using correct specification shows that $\boldsymbol\beta_0$ is a minimizer and that every minimizer $\widetilde{\boldsymbol\beta}$ induces the true conditional probability:
\[
    \mathbf p(\mathbf X;\widetilde{\boldsymbol\beta},\boldsymbol\omega_0)
    =\mathbf p_0(\mathbf X)\quad\text{a.s.}
\]
Applying Theorem~\ref{thm:pbe_joint_identification} with the structural weight fixed at $\boldsymbol\omega_0$ gives $\widetilde{\boldsymbol\beta}=\boldsymbol\beta_0$. Thus $\mathcal B^\dagger(\boldsymbol\omega_0)=\{\boldsymbol\beta_0\}$.

\subsection{Proof of Theorem \ref{thm:pbe_population_identification}}

For any $\mathbf X$-measurable probability vector $\mathbf q(\mathbf X)$, the conditional mean property $\E[\mathbf Y-\mathbf p_0(\mathbf X)\mid\mathbf X]=\mathbf0$ gives the orthogonal decomposition
\begin{align*}
    \E[\|\mathbf Y-\mathbf q(\mathbf X)\|_2^2\mid\mathbf X]
    &=\E[\|\mathbf Y-\mathbf p_0(\mathbf X)\|_2^2\mid\mathbf X]\\
    &\quad+\|\mathbf q(\mathbf X)-\mathbf p_0(\mathbf X)\|_2^2.
\end{align*}
Set $\mathbf q(\mathbf X)=\mathbf p(\mathbf X;\boldsymbol\beta^\dagger(\boldsymbol\omega),\boldsymbol\omega)$. Corollary~\ref{cor:pbe_inner_recovery} implies that at $\boldsymbol\omega_0$ this probability equals $\mathbf p_0$. Therefore
\[
    Q(\boldsymbol\omega)-Q(\boldsymbol\omega_0)
    =\E\left[\left\|\mathbf p(\mathbf X;
    \boldsymbol\beta^\dagger(\boldsymbol\omega),\boldsymbol\omega)
    -\mathbf p_0(\mathbf X)\right\|_2^2\right]\ge0.
\]
Equality holds only if the profiled candidate reproduces $\mathbf p_0$ almost surely. Theorem~\ref{thm:pbe_joint_identification} then implies $\boldsymbol\omega=\boldsymbol\omega_0$ and $\boldsymbol\beta^\dagger(\boldsymbol\omega)=\boldsymbol\beta_0$. Hence the population bi-level solution is unique.

\subsection{Proof of Theorem \ref{thm:pbe_consistency}}

Under Assumption~\ref{assump:pbe_uniform_regularity}, the standard uniform law of large numbers for a compact, continuously indexed, dominated class gives
\begin{align}
    \Delta_{\ell,N}
    &:=
    \sup_{(\boldsymbol\beta,\boldsymbol\omega)\in\mathcal B\times\mathcal W}
    |P_N\ell(\boldsymbol\beta,\boldsymbol\omega)
    -P\ell(\boldsymbol\beta,\boldsymbol\omega)|
    \xrightarrow{p}0,
    \label{eq:pbe_inner_ulln_proof}\\
    \Delta_{b,N}
    &:=
    \sup_{(\boldsymbol\beta,\boldsymbol\omega)\in\mathcal B\times\mathcal W}
    |P_N b(\boldsymbol\beta,\boldsymbol\omega)
    -P b(\boldsymbol\beta,\boldsymbol\omega)|
    \xrightarrow{p}0.
    \label{eq:pbe_outer_ulln_proof}
\end{align}
For the second display, the envelope may be taken to be the constant $2$. Compactness of $\mathcal B$ also gives $M_{\mathcal B}:=\sup_{\boldsymbol\beta\in\mathcal B}\|\boldsymbol\beta\|_2^2<\infty$. Consequently,
\begin{equation}
    \rho_N
    :=\sup_{(\boldsymbol\beta,\boldsymbol\omega)\in\mathcal B\times\mathcal W}
    |L_N^\theta(\boldsymbol\beta,\boldsymbol\omega)
    -L(\boldsymbol\beta,\boldsymbol\omega)|
    \le\Delta_{\ell,N}+\theta_NM_{\mathcal B}
    \xrightarrow{p}0.
    \label{eq:pbe_penalized_inner_ulln}
\end{equation}

Dominated convergence makes $L$ jointly continuous on $\mathcal B\times\mathcal W$. Assumption~\ref{assump:pbe_profile_wellposed}, compactness, and the maximum theorem then imply continuity of $\boldsymbol\beta^\dagger(\boldsymbol\omega)$ on $\mathcal W$.

For $\varepsilon>0$, consider the population separation
\begin{equation}
    \xi_\beta(\varepsilon)
    :=\inf_{\substack{\boldsymbol\omega\in\mathcal W,\ \boldsymbol\beta\in\mathcal B:\\
    \|\boldsymbol\beta-\boldsymbol\beta^\dagger(\boldsymbol\omega)\|_2\ge\varepsilon}}
    \left\{L(\boldsymbol\beta,\boldsymbol\omega)
    -L(\boldsymbol\beta^\dagger(\boldsymbol\omega),\boldsymbol\omega)\right\}.
    \label{eq:pbe_uniform_inner_separation}
\end{equation}
The set is compact, so uniqueness of the population inner minimizer gives $\xi_\beta(\varepsilon)>0$ whenever the set is nonempty. On $2\rho_N<\xi_\beta(\varepsilon)$, every empirical inner minimizer lies within $\varepsilon$ of its population counterpart. Therefore
\begin{equation}
    \sup_{\boldsymbol\omega\in\mathcal W}
    \|\hat{\boldsymbol\beta}_N(\boldsymbol\omega)
    -\boldsymbol\beta^\dagger(\boldsymbol\omega)\|_2
    \xrightarrow{p}0.
    \label{eq:pbe_uniform_profile_consistency}
\end{equation}

Let $B(\boldsymbol\beta,\boldsymbol\omega):=Pb(\boldsymbol\beta,\boldsymbol\omega)$. Bounded convergence and the assumed almost-sure continuity of the choice map imply that $B$ is jointly continuous, hence uniformly continuous on the compact product space. Writing
\[
    \widehat Q_N(\boldsymbol\omega)
    :=P_N b(\hat{\boldsymbol\beta}_N(\boldsymbol\omega),\boldsymbol\omega),
\]
we have
\begin{align}
    \sup_{\boldsymbol\omega\in\mathcal W}
    |\widehat Q_N(\boldsymbol\omega)-Q(\boldsymbol\omega)|
    &\le \Delta_{b,N}\\
    &\quad+\sup_{\boldsymbol\omega\in\mathcal W}
    |B(\hat{\boldsymbol\beta}_N(\boldsymbol\omega),\boldsymbol\omega)
    -B(\boldsymbol\beta^\dagger(\boldsymbol\omega),\boldsymbol\omega)|
    \xrightarrow{p}0.
    \label{eq:pbe_profile_outer_ulln}
\end{align}
The first term is a supremum over the entire deterministic parameter class and therefore also controls evaluation at the data-dependent inner solution. This is the step that permits the same sample to be used at both levels.

Continuity of $B$ and $\boldsymbol\beta^\dagger$ implies continuity of $Q$. By
Eq.~\eqref{eq:pbe_population_separation}, $Q$ has the unique minimizer $\boldsymbol\omega_0$.
Thus, for every $\varepsilon>0$,
\[
    \kappa(\varepsilon)
    :=\inf_{\substack{\boldsymbol\omega\in\mathcal W:\\
    \|\boldsymbol\omega-\boldsymbol\omega_0\|_2\ge\varepsilon}}
    \{Q(\boldsymbol\omega)-Q(\boldsymbol\omega_0)\}>0
\]
whenever the constraint set is nonempty. Uniform convergence in Eq.~\eqref{eq:pbe_profile_outer_ulln} and the defining optimality of $\hat{\boldsymbol\omega}_N$ then give
$\hat{\boldsymbol\omega}_N\xrightarrow{p}\boldsymbol\omega_0$ by the argmin theorem.

Finally,
\begin{align*}
    \|\hat{\boldsymbol\beta}_N(\hat{\boldsymbol\omega}_N)
    -\boldsymbol\beta_0\|_2
    &\le
    \sup_{\boldsymbol\omega\in\mathcal W}
    \|\hat{\boldsymbol\beta}_N(\boldsymbol\omega)
    -\boldsymbol\beta^\dagger(\boldsymbol\omega)\|_2\\
    &\quad+
    \|\boldsymbol\beta^\dagger(\hat{\boldsymbol\omega}_N)
    -\boldsymbol\beta^\dagger(\boldsymbol\omega_0)\|_2.
\end{align*}
The first term converges to zero by Eq.~\eqref{eq:pbe_uniform_profile_consistency}; the second does so by continuity of $\boldsymbol\beta^\dagger$ and consistency of $\hat{\boldsymbol\omega}_N$. Equation~\eqref{eq:pbe_population_separation} gives $\boldsymbol\beta^\dagger(\boldsymbol\omega_0)=\boldsymbol\beta_0$, completing the proof.


\subsection{Proof of Proposition \ref{prop:continuity}}

In the original parameter coordinates, the inner objective is
\[
    \mathcal J_{\mathrm{in},N}(\boldsymbol\beta,\boldsymbol\omega)
    =
    \sum_{n=1}^N
    \Gamma_{\mathrm{FY}}
    \bigl(
        \mathbf{y}_n,\mathbf{X}_n\boldsymbol{\beta};
        \Lambda_{\boldsymbol{\omega}}
    \bigr)
    +
    N\theta_N\|\boldsymbol{\beta}\|_2^2 .
\]
The FY terms are convex in $\boldsymbol\beta$, and the ridge is strongly convex. Thus $\mathcal J_{\mathrm{in},N}(\cdot,\boldsymbol\omega)$ has at most one minimizer.
Existence follows from continuity and coercivity; if $\mathcal B$ is bounded, compactness gives the
same conclusion directly. In PBE the endpoint-normalized canonical gaps are finite and nonnegative, so the ridge term also places all minimizers
in a common bounded set as $\boldsymbol\omega$ ranges over the compact set $\mathcal W$.

The perturbation maximization defining the choice map has a compact feasible simplex and a unique
optimizer. Joint continuity of its objective and the maximum theorem imply continuity of the choice
map in utilities and structural weights. It follows that $\mathcal J_{\mathrm{in},N}$ is jointly continuous. Applying the
maximum theorem on the common bounded set containing all inner minimizers, together with uniqueness,
gives
\[
    \hat{\boldsymbol{\beta}}_N(\boldsymbol{\omega}_k)
    \to
    \hat{\boldsymbol{\beta}}_N(\boldsymbol{\omega}).
\]

Now fix a regime. Restrict each sample's primal KKT system to its active simplex face. The assumed
positive definiteness of the perturbation Hessian on the face tangent space makes the corresponding
bordered KKT Jacobian nonsingular. The implicit function theorem therefore makes the choice map
continuously differentiable within the regime.

Use the free coordinates in Eq.~\eqref{eq:pbe_free_coordinates} for the fixed face of $\mathcal B$.
The inner first-order condition is
\[
    \nabla_{\boldsymbol\gamma}
    \mathcal J_{\mathrm{in},N}(\boldsymbol\beta_{\mathrm{ref}}+
    \mathbf S_\beta\hat{\boldsymbol\gamma}_N(\boldsymbol\omega),
    \boldsymbol\omega)
    =
    \mathbf{0}.
\]
Its Jacobian with respect to $\boldsymbol\gamma$ is
\[
    \sum_{n=1}^N
    \widetilde{\mathbf X}_n\mathbf H_{\Omega,n}
    \widetilde{\mathbf X}_n^\top+2N\theta_N\mathbf I
    \succ\mathbf0,
\]
where $\mathbf H_{\Omega,n}$ is the active-face choice Jacobian. A second application of the
implicit function theorem shows that $\hat{\boldsymbol\gamma}_N$, and hence
$\hat{\boldsymbol\beta}_N$, is continuously differentiable on the relative interior of the regime.

By assumption, the union of regime boundaries has Lebesgue measure zero. Therefore
$\hat{\boldsymbol{\beta}}_N(\boldsymbol{\omega})$ is differentiable almost everywhere on
$\mathcal W$ relative to its affine hull.

Finally, define $\hat{\mathbf p}_n(\boldsymbol\omega):=\mathbf p(\mathbf X_n;\hat{\boldsymbol\beta}_N(\boldsymbol\omega),\boldsymbol\omega)$. Then
\[
    \mathcal{L}_{\mathrm{outer},N}(\boldsymbol{\omega})
    =
    \frac1N\sum_{n=1}^N
    \|\mathbf y_n-\hat{\mathbf p}_n(\boldsymbol\omega)\|_2^2
\]
is a continuous composition of continuous mappings and is continuously differentiable within each
regime. Since the finite union of regime boundaries has relative Lebesgue measure zero, both asserted
almost-everywhere differentiability conclusions follow.

\subsection{Proof of Theorem \ref{thm:pbe_pgd_stationarity}}

Let $\mathbf d_t:=\boldsymbol\omega^{(t+1)}-\boldsymbol\omega^{(t)}$.
The variational characterization of Euclidean projection, applied with
$\boldsymbol\omega^{(t)}\in\mathcal W$, gives
\[
    \left\langle\nabla\widehat Q_N(\boldsymbol\omega^{(t)}),\mathbf d_t\right\rangle
    \le-\frac1{\eta_t}\|\mathbf d_t\|_2^2.
\]
Whenever $\eta_t\le L_N^{-1}$, the descent lemma implies
\begin{align*}
    \widehat Q_N(\boldsymbol\omega^{(t+1)})
    &\le \widehat Q_N(\boldsymbol\omega^{(t)})
    +\left\langle\nabla\widehat Q_N(\boldsymbol\omega^{(t)}),\mathbf d_t\right\rangle
    +\frac{L_N}{2}\|\mathbf d_t\|_2^2\\
    &\le \widehat Q_N(\boldsymbol\omega^{(t)})
    -\frac{\eta_t}{2}\|\mathcal G_{\eta_t}(\boldsymbol\omega^{(t)})\|_2^2,
\end{align*}
because $\mathbf d_t=-\eta_t\mathcal G_{\eta_t}(\boldsymbol\omega^{(t)})$. Thus the
sufficient-decrease test is met after finitely many backtracking reductions. For the accepted steps,
sum Eq.~\eqref{eq:pbe_pgd_descent} from $t=0$ to $T-1$. The lower boundedness
of the continuous function $\widehat Q_N$ on compact $\mathcal W$ proves both the displayed
finite-iteration bound and $\mathcal G_{\eta_t}(\boldsymbol\omega^{(t)})\to\mathbf0$.

Compactness provides an accumulation point. Along any convergent subsequence
$\boldsymbol\omega^{(t_j)}\to\overline{\boldsymbol\omega}$, pass to a further subsequence such that
$\eta_{t_j}\to\overline\eta_\ast\in[\underline\eta,\overline\eta]$. Continuity of the gradient and projection
gives $\mathcal G_{\overline\eta_\ast}(\overline{\boldsymbol\omega})=\mathbf0$. The projection optimality condition
is equivalent to
$-\nabla\widehat Q_N(\overline{\boldsymbol\omega})\in
N_{\mathcal W}(\overline{\boldsymbol\omega})$, which is Eq.~\eqref{eq:pbe_stationarity_condition}.


\newpage
\section{Forward Solution of a Tree-Structured PUM}
\label{app:tree_forward_solver}

This appendix gives an exact dynamic-programming formulation for evaluating the probability map of
the tree-structured model. It applies directly to the finite-valued PBE dictionaries. For an
extended-valued boundary barrier, the scalar costs remain extended-valued functions on $[0,1]$, with
$+\infty$ assigned outside their effective domains; this preserves the closed-domain formulation below.
Let $r$ be the root and let $\operatorname{ch}(s)$ denote the children of node $s$. Set
$\Phi_r\equiv0$. For a leaf $i$, define the conditional subtree cost
\begin{equation}
    C_i(y;\mathbf V):=h(y)-V_i y,
    \qquad 0\le y\le1.
    \label{eq:tree_leaf_value}
\end{equation}
For an internal node $s$, recursively define
\begin{equation}
    C_s(y;\mathbf V)
    :=\Phi_s(y)+
    \inf_{\substack{y_c\ge0, c\in\operatorname{ch}(s):\\
    \sum_{c\in\operatorname{ch}(s)}y_c=y}}
    \sum_{c\in\operatorname{ch}(s)}C_c(y_c;\mathbf V),
    \qquad 0\le y\le1.
    \label{eq:tree_value_recursion}
\end{equation}
Thus $C_s(y;\mathbf V)$ is the minimum contribution of the entire subtree rooted at $s$,
conditional on sending mass $y$ into that subtree. A bottom-up pass constructs these scalar convex
value functions by infimal convolution and addition of the node cost. Each $C_s$ is strictly convex:
conditional optimizers associated with two distinct incoming masses are distinct leaf-flow vectors,
strict convexity of the summed leaf costs gives a strict inequality at their convex combination, and
the convex node costs preserve it. At the root,
\begin{equation}
    C_r(1;\mathbf V)
    =\min_{\mathbf p\in\Delta}
    \{\Lambda(\mathbf p)-\mathbf p^\top\mathbf V\}
    =-\Omega(\mathbf V).
    \label{eq:tree_root_value}
\end{equation}

Each allocation in Eq.~\eqref{eq:tree_value_recursion} is a one-constraint convex resource-allocation
problem. For a child $c$, define its scalar response
\begin{equation}
    \rho_c(\tau)
    :=\argmax_{0\le u\le1}\{\tau u-C_c(u;\mathbf V)\}
    =\nabla C_c^*(\tau),
    \label{eq:tree_child_response}
\end{equation}
where strict convexity of the aggregate leaf cost makes the maximizer unique. Given an incoming
mass $y_s$, the KKT conditions for the split at node $s$ are
\begin{equation}
    y_c=\rho_c(\tau_s),\quad c\in\operatorname{ch}(s),
    \qquad
    \sum_{c\in\operatorname{ch}(s)}\rho_c(\tau_s)=y_s.
    \label{eq:tree_node_root}
\end{equation}
The aggregate response in Eq.~\eqref{eq:tree_node_root} is continuous and nondecreasing. Its limiting
values bracket every feasible $y_s\in[0,1]$, so a bracketing method converges to a node multiplier
$\tau_s$. Flat portions may make the multiplier nonunique, but the optimal child masses remain
unique. For smooth primitives Newton or safeguarded Newton can be used under the usual nonzero-slope
condition; for spline or truncated responses, bisection is sufficient.

The forward probability evaluation is therefore:
\begin{enumerate}
    \item construct the subtree value functions $C_s$ from the leaves to the root using
    Eq.~\eqref{eq:tree_value_recursion};
    \item set $y_r=1$ and, starting at the root, solve Eq.~\eqref{eq:tree_node_root} for each visited
    internal node;
    \item pass the recovered child masses downward and return $p_i=y_i$ at every leaf.
\end{enumerate}
For analytic primitives the scalar responses may be evaluated directly; for spline primitives the
one-dimensional value functions may instead be represented piecewise or approximated on an adaptively
refined grid. The exactness statement below concerns exact scalar-response evaluations; its numerical
counterpart follows as the node-solve and function-approximation tolerances vanish.
The cost of constructing or approximating the one-dimensional value-function representations depends
on the chosen analytic, piecewise, or grid implementation. Once the child-response oracles have been
stored, if each oracle evaluation has constant cost and each node root is solved in at most $J$
response evaluations, the top-down allocation pass uses $\mathcal O(J|\mathcal T|)$ scalar-response
evaluations because the sum of the node degrees is $|\mathcal T|-1$.

To verify exactness, proceed by induction from the leaves. Equation~\eqref{eq:tree_leaf_value} is the
conditional objective for a one-leaf subtree. If the statement holds for all children of $s$, then
minimizing their conditional objectives subject to flow conservation and adding $\Phi_s(y)$ gives
Eq.~\eqref{eq:tree_value_recursion}; hence it holds for $s$. At the root this yields
Eq.~\eqref{eq:tree_root_value}. Strict convexity of the sum of leaf costs makes the terminal flow
vector unique. Finally, Eq.~\eqref{eq:tree_node_root} is exactly the KKT system of each conditional
allocation. Consequently, exact node solves recover the unique PUM probability vector, and the
output of the bracketing implementation converges to that vector as the node tolerances vanish.

\newpage
\section{Detailed Derivations of Section \ref{subsec:PBE_solution}} \label{app:hypergradient_derivation}

This appendix gives the components of Eq.~\eqref{eq:hypergradient_final} and a tree-structured
implementation of the required linear operators. The derivation is made within one of the smooth
fixed-face regimes of Proposition~\ref{prop:continuity}. The dense case is obtained by taking the
active face to be the full simplex.

Let $\mathbf V_n=\mathbf X_n\hat{\boldsymbol\beta}$ and
$\hat{\mathbf p}_n=\nabla\Omega_{\Lambda_{\boldsymbol\omega}}(\mathbf V_n)$. Write
$S_n:=\{i:\hat p_{n,i}>0\}$ for its locally stable support. The active-face choice Jacobian
$\mathbf H_{\Omega,n}$ is the linear map from utility perturbations to probability perturbations,
embedded in $\mathbb R^{|\mathcal C|}$ by zeros outside $S_n$. It is symmetric positive semidefinite
and satisfies $\mathbf H_{\Omega,n}\mathbf1=\mathbf0$. For a generic structural coordinate $\omega_k$,
define
\[
    \mathbf g_{\omega_k}(\mathbf p)
    :=\frac{\partial}{\partial\omega_k}
    \nabla_{\mathbf p}\Lambda_{\boldsymbol\omega}(\mathbf p).
\]

The analytical components are derived as follows:

\begin{enumerate}
    \item \textbf{Outer gradient in free utility coordinates.}
    For $\mathcal{L}_{\mathrm{outer},N}=N^{-1}\sum_{n=1}^N
    \|\mathbf y_n-\hat{\mathbf p}_n\|_2^2$ and
    $\widetilde{\mathbf X}_n=\mathbf S_\beta^\top\mathbf X_n^\top$,
    \begin{equation}
        \nabla_{\boldsymbol\gamma}\mathcal L_{\mathrm{outer},N}
        \big|_{\boldsymbol\gamma=\hat{\boldsymbol\gamma}_N}
        =\frac2N\sum_{n=1}^N\widetilde{\mathbf X}_n
        \mathbf H_{\Omega,n}(\hat{\mathbf p}_n-\mathbf y_n).
        \label{eq:pbe_outer_gamma_gradient}
    \end{equation}

    \item \textbf{Inner Hessian.}
    Because $\mathbf S_\beta$ has orthonormal columns,
    \begin{equation} 
        \mathbf H_{\mathrm{inner}}
        =\sum_{n=1}^N\widetilde{\mathbf X}_n\mathbf H_{\Omega,n}
        \widetilde{\mathbf X}_n^\top+2N\theta_N\mathbf I
        \succ\mathbf0.
        \label{eq:pbe_inner_hessian_explicit}
    \end{equation}

    \item \textbf{Mixed partials.}
    Differentiating the active-face primal KKT system while holding utilities fixed gives
    $\partial\hat{\mathbf p}_n/\partial\omega_k
    =-\mathbf H_{\Omega,n}\mathbf g_{\omega_k}(\hat{\mathbf p}_n)$. Hence column $k$ of
    $\mathbf B_{\mathrm{inner}}$ is
    \begin{equation}
        [\mathbf B_{\mathrm{inner}}]_{:,k}
        =-\sum_{n=1}^N\widetilde{\mathbf X}_n\mathbf H_{\Omega,n}
        \mathbf g_{\omega_k}(\hat{\mathbf p}_n).
        \label{eq:pbe_mixed_partial_explicit}
    \end{equation}
    For a shared leaf-level weight $\omega_m^h$, this vector has entries
    $h_m'(\hat p_{n,i})$. For a nest-level weight $\omega_{s,l}^\phi$, it equals
    $\phi_l'(\hat y_{n,s})\mathbf a_s$.

    \item \textbf{Direct structural gradient.}
    Holding $\hat{\boldsymbol\gamma}_N$ fixed,
    \begin{equation}
        \frac{\partial\mathcal L_{\mathrm{outer},N}}{\partial\omega_k}
        =-\frac2N\sum_{n=1}^N(\hat{\mathbf p}_n-\mathbf y_n)^\top
        \mathbf H_{\Omega,n}\mathbf g_{\omega_k}(\hat{\mathbf p}_n).
        \label{eq:pbe_direct_gradient_explicit}
    \end{equation}
\end{enumerate}

It remains to apply $\mathbf H_{\Omega,n}$ without forming a dense matrix. On support $S_n$, let
$\mathbf A_{S_n}$ be the leaf--ancestor incidence matrix restricted to active leaves and define
\begin{align}
    \mathbf W_{n,S_n}
    &:=\operatorname{diag}\left(
    \left\{h_{\boldsymbol\omega_h}''(\hat p_{n,i})\right\}_{i\in S_n}\right),\\
    \boldsymbol\Sigma_n
    &:=\operatorname{diag}\left(
    \left\{\Phi_{s,\boldsymbol\omega_{\phi,s}}''(\hat y_{n,s})
    \right\}_{s\in\mathcal A_{\mathcal T}}\right),\\
    \mathbf D_{n,S_n}
    &:=\mathbf W_{n,S_n}+\mathbf A_{S_n}\boldsymbol\Sigma_n\mathbf A_{S_n}^\top.
    \label{eq:pbe_tree_primal_hessian}
\end{align}
Positive definiteness on the active tangent space makes the following bordered system nonsingular.
For any $\mathbf v\in\mathbb R^{|\mathcal C|}$, the product
$\mathbf z=\mathbf H_{\Omega,n}\mathbf v$ is determined by
\begin{equation}
    \begin{pmatrix}
        \mathbf D_{n,S_n}&\mathbf1_{S_n}\\
        \mathbf1_{S_n}^\top&0
    \end{pmatrix}
    \begin{pmatrix}\mathbf z_{S_n}\\\Delta c\end{pmatrix}
    =
    \begin{pmatrix}\mathbf v_{S_n}\\0\end{pmatrix},
    \qquad
    \mathbf z_{S_n^c}=\mathbf0.
    \label{eq:pbe_active_face_kkt}
\end{equation}
This system also shows directly that probability perturbations remain on the active-face tangent space.

The tree makes multiplication by $\mathbf D_{n,S_n}$ linear in the representation size. Given a
vector $\mathbf r$, compute $\mathbf u=\mathbf A_{S_n}^\top\mathbf r$ by a bottom-up pass: $u_s$ is
the sum of $r_i$ over active leaves descending from node $s$. Multiply each message by the corresponding
diagonal entry of $\boldsymbol\Sigma_n$, then compute
$\mathbf A_{S_n}(\boldsymbol\Sigma_n\mathbf u)$ by a top-down pass that accumulates ancestor messages
at each active leaf. Adding $\mathbf W_{n,S_n}\mathbf r$ gives
$\mathbf D_{n,S_n}\mathbf r$ in
$\mathcal O(|S_n|+|\mathcal T|)$ operations.

Equation~\eqref{eq:pbe_active_face_kkt} can therefore be solved to a prescribed tolerance by MINRES on
the symmetric bordered system, or by conjugate gradients after restricting to a basis of
$\{\mathbf z:\mathbf1_{S_n}^\top\mathbf z=0\}$. If $J_{\mathrm{lin}}$ iterations are used, one
choice-Jacobian product costs
$\mathcal O(J_{\mathrm{lin}}(|S_n|+|\mathcal T|))$ operations and requires no dense
$|\mathcal C|\times|\mathcal C|$ matrix. The same products enter
Eq.~\eqref{eq:pbe_inner_hessian_explicit}, so the solve with $\mathbf H_{\mathrm{inner}}$ can likewise
be carried out by conjugate gradients using Hessian--vector products. This operator formulation is
valid for both the dense simplex and any locally stable sparse face.


\newpage
\section{Choice of Finite Basis Perturbation Functions} \label{app:Choice_Basis}

This appendix constructs finite dictionaries $\mathcal H_{\mathrm{leaf}}$ and $\mathcal H_{\mathrm{nest}}$ whose normalized convex primitives are well separated under a weighted functional distance. Given a fixed basis budget, the purpose is to select a concise collection with broad geometric variation rather than a set of nearly redundant perturbations. This leads to a maximin design over pairwise functional distances.

The construction seeks finite sets of convex basis functions whose pairwise distance is large. Within the anchored, finite-valued spline class used here, it is convenient to parameterize derivatives. Monotonicity of the derivative enforces convexity, its integral over $[0,1]$ enforces the endpoint normalization $h_m(1)-h_m(0)=0$, and a second linear integral constraint fixes the common area scale. These conditions characterize the chosen spline design class; they are not a characterization of all admissible extended-valued perturbations, such as entropy or Cauchy primitives with divergent boundary derivatives.

Direct optimization over continuous convex primitives is infinite dimensional. Parameterizing their derivatives with B-splines reduces the functional design to a finite collection of control vectors, so convexity and normalization become linear restrictions and functional dissimilarity becomes a finite-dimensional quadratic-form calculation.

We describe a dictionary of $M$ leaf-level bases; the nest-level construction replaces $M$ by $L$. For $x\in[0,1]$, parameterize the $m$th derivative by a cubic B-spline with $R$ control points $\mathbf c_m=(c_{m1},\ldots,c_{mR})^\top\in\mathbb R^R$:
\begin{equation}
    f_m(x) = \sum_{r=1}^R c_{m,r} B_{r,3}(x) = \mathbf{b}(x)^\top \mathbf{c}_m,
\end{equation}
where $\mathbf b(x):=(B_{1,3}(x),\ldots,B_{R,3}(x))^\top$.

Define the weighted Gram matrix $\mathbf Q_{\mathrm{sp}}\in\mathbb R^{R\times R}$ by
$[\mathbf Q_{\mathrm{sp}}]_{ij}:=\int_0^1B_{i,3}(x)B_{j,3}(x)x(1-x)\,dx$. The weight $x(1-x)$ emphasizes interior differences and attenuates boundary behavior.

Let $v_r:=\int_0^1B_{r,3}(x)\,dx$ and $u_r:=\int_0^1xB_{r,3}(x)\,dx$. Then $\mathbf v^\top\mathbf c_m=0$ enforces $h_m(1)-h_m(0)=0$, and integration by parts gives
\[
    \int_0^1 x h'_m(x) \, dx = \Big[x h_m(x)\Big]_0^1 - \int_0^1 h_m(x) \, dx.
\]
With $h_m(0)=h_m(1)=0$, the right side equals $Z(h_m)$; hence $\mathbf u^\top\mathbf c_m=C$ imposes the common area scale. Let $\mathbf D_1\in\mathbb R^{(R-1)\times R}$ be the first-difference matrix, so $[\mathbf D_1\mathbf c_m]_r=c_{m,r+1}-c_{mr}$.

The maximin design is
\begin{align}
    \max_{\mathbf c_1,\ldots,\mathbf c_M\in\mathbb R^R}\quad
    &\min_{1\le i<j\le M}\left(1-
    \frac{\mathbf c_i^\top\mathbf Q_{\mathrm{sp}}\mathbf c_j}
    {\sqrt{\mathbf c_i^\top\mathbf Q_{\mathrm{sp}}\mathbf c_i}
    \sqrt{\mathbf c_j^\top\mathbf Q_{\mathrm{sp}}\mathbf c_j}}\right)
    \label{eq:bspline_obj}\\
    \text{subject to:} \quad & \mathbf{v}^\top \mathbf{c}_m = 0, \quad \forall m \in \{1, \dots, M\}, \label{eq:bspline_zero_int} \\
    & \mathbf{u}^\top \mathbf{c}_m = C, \quad \forall m \in \{1, \dots, M\}, \label{eq:bspline_dist_match} \\
    & \mathbf D_1\mathbf c_m\ge\varepsilon_{\mathrm c}\mathbf1,
    \quad m=1,\ldots,M. \label{eq:bspline_monotone}
\end{align}
Here $\varepsilon_{\mathrm c}>0$ enforces strict convexity. Linear independence of the B-splines and $x(1-x)>0$ almost everywhere imply $\mathbf Q_{\mathrm{sp}}\succ\mathbf0$. Since $C>0$ also implies $\mathbf c_m\ne\mathbf0$, every denominator is positive.

Introducing the minimum pairwise distance $\delta$ gives the smooth epigraph formulation
\begin{align}
    \max_{\delta \in \mathbb{R}, \{\mathbf{c}_1, \dots, \mathbf{c}_M\}} \quad & \delta \label{eq:bspline_obj_smooth} \\
    \text{subject to}\quad & \delta \le 1-
    \frac{\mathbf c_i^\top\mathbf Q_{\mathrm{sp}}\mathbf c_j}
    {\sqrt{\mathbf c_i^\top\mathbf Q_{\mathrm{sp}}\mathbf c_i}
    \sqrt{\mathbf c_j^\top\mathbf Q_{\mathrm{sp}}\mathbf c_j}},
    \quad 1\le i<j\le M, \label{eq:bspline_distance_bound}\\
    & \mathbf{v}^\top \mathbf{c}_m = 0, \quad \forall m \in \{1, \dots, M\}, \label{eq:bspline_zero_int_smooth} \\
    & \mathbf{u}^\top \mathbf{c}_m = C, \quad \forall m \in \{1, \dots, M\}, \label{eq:bspline_dist_match_smooth} \\
    & \mathbf D_1\mathbf c_m\ge\varepsilon_{\mathrm c}\mathbf1,
    \quad m=1,\ldots,M. \label{eq:bspline_monotone_smooth}
\end{align}

The cosine-distance constraints are non-convex. We use trust-region sequential convex programming, linearizing the smooth constraints and accepting steps by actual improvement. Under standard constraint qualification, bounded-iterate, and trust-region conditions, accumulation points are first-order stationary.

Purely optimization-driven spline bases provide flexibility but may have limited immediate behavioral interpretation. A hybrid dictionary therefore retains canonical anchors, such as Shannon entropy, and supplements them with optimized splines that capture departures from those benchmarks. Projecting anchor derivatives onto the spline space yields fixed control vectors, and the maximin objective can include distances between the new bases and these anchors.

\end{document}